\documentclass[10pt]{amsart}
\usepackage{lmodern}
\usepackage{amsmath}
\usepackage{amssymb}
\usepackage{amsfonts}
\usepackage{amsthm}
\usepackage{enumerate}
\usepackage[all]{xy}
\usepackage[latin1]{inputenc}  
\usepackage{graphicx}
\usepackage{mathdots}
\usepackage{mathrsfs}
\usepackage{color}
\input xy
\usepackage{tikz}
\usetikzlibrary{arrows}
\usetikzlibrary{decorations,arrows}
\usetikzlibrary{decorations.markings}
\usetikzlibrary{arrows.meta, bending}
\usetikzlibrary{positioning}

\makeatletter

\newtheoremstyle{definition}% name of the style to be used
        {5pt}% measure of space to leave above the theorem. E.g.: 3pt
        {3pt}% measure of space to leave below the theorem. E.g.: 3pt
        {}% name of font to use in the body of the theorem
        {0pt}% measure of space to indent
        {\scshape}% name of head font
        {.}% punctuation between head and body
        {5pt}% space after theorem head; " " = normal interword space
        {\thmname{#1} \thmnumber{#2} \thmnote{[#3]}} % Manually specify head

\newtheoremstyle{theorems}% name of the style to be used
        {5pt}% measure of space to leave above the theorem. E.g.: 3pt
        {3pt}% measure of space to leave below the theorem. E.g.: 3pt
        {\itshape}% name of font to use in the body of the theorem
        {0pt}% measure of space to indent
        {\scshape}% name of head font
        {.}% punctuation between head and body
        {5pt}% space after theorem head; " " = normal interword space
        {\thmname{#1} \thmnumber{#2}\thmnote{[#3]}} % Manually specify head

\swapnumbers \theoremstyle{theorems}

\newtheorem{Theo}{Theorem}[subsection]
\newtheorem{Prop}[Theo]{Proposition}
\newtheorem{Cor}[Theo]{Corollary}
\newtheorem{Lemma}[Theo]{Lemma}
\newtheorem{Prop(BG)}[Theo]{Proposition (Bongartz-Gabriel)}
\newtheorem{Lemma(Asashiba)}[Theo]{Lemma(Asashiba)}
\newtheorem{Lemma(Gab)}[Theo]{Lemma(Gabriel)}
\newtheorem{Theo(Mil)}[Theo]{Theorem (Milicic)}

\theoremstyle{definition}
\newtheorem{Defn}[Theo]{Definition}

\newtheorem{Defn(Asashiba)}[Theo]{Definition (Asashiba)}
\newtheorem{Remark}[Theo]{Remark}

\newcommand{\Hom}{{\rm Hom}}

\newcommand{\mmod}{{\rm mod}}
\newcommand{\rad}{{\rm rad}\hspace{.5pt}}
\newcommand{\soc}{{\rm soc}\hspace{.5pt}}
\newcommand{\ntop}{{\rm top}\hspace{.5pt}}

\newcommand{\Ker}{{\rm Ker}}

\newcommand{\hp}{\hspace{.5pt}}
\newcommand{\vp}{\vspace{.5pt}}

\def\Ga{\hbox{$\mathit\Gamma\hspace{-2pt}$}}

\begin{document}

\title[Representation bound]{\sc Artin algebras of small representation bound}

\vspace{5pt}

\author[Y. Yin]{Youqi Yin}

\author[S. Liu]{Shiping Liu \vspace{-10pt}}

\address{Youqi Yin \\ Departement of mathematics, Shaoxing University, Shaoxing, China.}
\email{yinyouqi@usx.edu.cn}

\address{{\it Corresponding author}: Shiping Liu \\ D\'epartement de math\'ematiques, Universit\'e de Sherbrooke, Sherbrooke, Qu\'ebec, 
Canada.}
\email{shiping.liu@usherbrooke.ca}



\thanks{The first named author is supported by NSF China 12371017.}

%\date{July 2024}

\begin{abstract}

\vspace{15pt}

This paper introduces a novel classification approach for represen\-tation-finite artin algebras in terms of the maximal length of their indecomposable modules of finite length, which we call the \emph{representation bound}. To this end, we develop methods to compute almost split sequences and establish lower bounds for the lengths of the Auslander-Reiten translates of certain mo\-dules over artin algebras. We apply these techniques and results to explicitly classify artin algebras of representation bound $n$ for each positive integer~$n \le 4.$

\vspace{-15pt}

\end{abstract}

\maketitle

\section*{Introduction}

The ultimate objective of the representation theory of artin algebras is to describe all indecomposable modules of finite length and the maps between them. In this paper, we shall concentrate on representation-finite algebras, namely, those having only finitely many non-isomorphic indecomposable modules of finite length. While representation-finite hereditary algebras can be classified by the type of their Ext-quiver, see \cite{DR, Gab, LiT, PAu}, the general case is significantly more complex. A celebrated result of Auslander states that an artin algebra is representation-finite if and only if there exists a bound for the lengths of its indecomposable modules of finite length, see \cite{Aus}, which is originally due to Roiter in the finite-dimensional setting; \cite{Roi}. In this context, we call the supremum of the lengths of these indecomposable modules the {\it representation bound} of the algebra. It is evident that an artin algebra has representation bound one if and only if it is semisimple. Moreover, the representation bound of a Nakayama algebra clearly coincides with its Loewy length, see \cite{Nak}, whereas the bound for a hereditary algebra of Dynkin type equals the order of its Coxeter transformation minus one; see \cite[(4.9)]{Zac}. Motivated by these classical examples, we propose to classify the representation-finite algebras in terms of their representation bound. 

A classification for an arbitrary bound is highly intricate. The main objective of this paper is to explicitly classify the artin algebras of representation bound $n$ for each integer $n \le 4$. While the classification  is trivial for the bound $n=1$ and straightforward for the bound $n=2,$ see (\ref{rb2}), increasing the bound introduces rapidly growing complexity. To overcome this, we develop methods to establish lower bounds for the lengths of the Auslander-Reiten translates of certain modules and to compute almost split sequences for certain uniserial modules over artin algebras. Then, we use the first technique to show that the artin algebras of representation bound at most $3$ are wedged-string algebras with radical cubed zero, see (\ref{3sa_def}), while those of representation bound at most $4$ are quadri-biserial algebras; see (\ref{4str_alg_def}). These algebras belong to the class of biserial algebras introduced by Fuller in \cite{Ful}. For the converses, we first reduce the representation theory of quadri-biserial algebras to that of quadri-string algebras with radical cubed zero; see (\ref{4ba-4sa}). Then, we apply the second technique to compute all almost split sequences over quadri-string algebras with radical cubed zero in Subsection 4.4. This enables us to show that wedged-string algebras with radical cubed zero have representation bound at most $3$, see (\ref{3sa_rb}), and quadri-biserial algebras have representation bound at most $4$; see (\ref{rb_4bsa}). This, in turn, yields an explicit description of all indecomposable modules and almost split sequences for these algebras; see (\ref{CL_3}) and (\ref{l4-4bsa}). Ultimately, these results complete the classification for the bounds $n=3$ and $n=4;$ see (\ref{rb3}) and (\ref{main_rb4}).

Note that the indecomposable modules and almost split sequences for string algebras defined by a quiver with relations have been explicitly described and computed by Butler and Ringel; see \cite{BuR}, and also \cite{SKW}. We should point out that their description and computation rely on the combinatorial structure of the bound quiver, which is not available for string artin algebras. In contrast, our technique only utilizes the definition of the dual of the transpose functor and some basic properties of almost split sequences. Moreover, our technique will be applicable in future research, especially for describing indecomposable modules and computing almost split sequences for artin algebras which are biserial or multiserial; see \cite{VHW}.

\section{Preliminaries} 

This section lays the foundation for this paper. We fix the notation and the terminology, which will be used throughout, and collect, as well as prove, some preliminary results on modules. Some of these results are not explicitly stated in the existing literature and may be applied in future study on modules.

\subsection{\sc Notation and terminology} Throughout this paper, $A$ denotes an artin $R$-algebra, where $R$ is an artinian commutative ring, and $J$ denotes the Jacobson radical of $A$. We write $\mmod A$ for the category of finitely generated left $A$-modules, with morphisms composed from right to left. A morphism in $\mmod A$ is called {\it radical} if it lies in the Jacobson radical of $\mmod A.$ Moreover, $A^\circ$ stands for the opposite algebra of $A$ and $J^\circ$ for the radical of $A^\circ$. For convenience, we shall identify $\mmod A^\circ$ with the category of finitely generated right $A$-modules.

Let $M$ be a module in $\mmod A$ with a submodule $N$. We call $N$ a {\it proper submodule} if $N\ne M,$ and $M/N$ a {\it proper quotient} if $N\ne 0$. 
When writing $L/N$ to denote a submodule of $M/N,$ we always assume that $L$ is a submodule of $M$ containing $N$. Two submodules of $M$ are called {\it comparable} if one is contained in the other. 
For simplicity, we shall denote by $P_M$ and $I_M$ a projective cover and an injective envelope of $M$ respectively. And we shall write $\rad M,$ $\soc M,$ and $\ntop M$ for the radical, the socle, and the top of $M$, respectively. Recall that $\rad(_AA)=\rad(A_A)=J$ and $\rad^i\hspace{-1pt}M=J^i\hspace{-1pt}M$ for $i\ge 0$. Furthermore, we shall denote by $\ell(M)$ the composition length (or simply, the length) of $M$, and by $\ell\ell(M)$ the {\it Loewy length} of $M.$ %which is the smallest integer $s\ge 0$ for which $\rad^s\hspace{-1pt}M=0.$ 
Note that $\ell\ell(_AA)=\ell\ell(A_A)$, which is called the {\it Loewy length} of $A$ and denoted by $\ell\ell(A)$. 

Finally we refer to \cite[Chapter V]{ARS}, and also to \cite{AuR1,AuR2,Bau}, for the Auslander-Reiten theory of irreducible maps and almost split sequences, which is an essential tool for our investigation. Note that we shall define the {\it Auslander-Reiten quiver} $\Ga_{\mmod A}$ of $\mmod A$ in such a way that its vertex set is a complete set of representatives of the indecomposable modules in $\mmod A.$ 

\subsection{\sc Top series} Let $M$ be a module in $\mmod A.$ Recall that 
$\rad^0\hspace{-1pt}M=M$ and $\rad^i\hspace{-1pt}M=0$ for $i\ge \ell\ell(M).$ Now for each integer $i\ge 0$, we put $\ntop^iM\hspace{-2pt}:=M/\rad^i\hspace{-1pt}M.$ In particular $\ntop^0\hspace{-1pt}M=0,$ $\ntop^1\hspace{-1pt}M=\ntop M,$ and $\ntop^i\hspace{-1pt}M=M$ for all $i\ge \ell\ell(M).$ The following statement is evident.

\begin{Lemma}\label{topi_LL}

Let $A$ be an artin algebra. Consider a module $M$ in $\mmod A$ of Loewy length $n.$ Then $\ell\ell(\rad^i\hspace{-1.2pt}M)=n-i,$ and $\ell\ell(\ntop^i\hspace{-1.2pt}M)=i$ for all $0\le i\le n$.

\end{Lemma}

\subsection{\sc Socle series} Let $M$ be a module in $\mmod A.$ Write $\soc^0\hspace{-1.2pt} M=0;$ and for each integer $i\ge 1$, write $\soc^i\hspace{-1.2pt}M$ for the submodule of $M$ containing $\soc^{i-1}\hspace{-1.2pt}M$ such that $\soc^i\hspace{-1.2pt}M/\soc^{i-1}\hspace{-1.2pt}M=\soc(M/\soc^{i-1}\hspace{-1.2pt}M).$ In particular, $\soc^1\hspace{-1.2pt}M=\soc M$. The following well-known characterization of $\soc^i\hspace{-1.2pt}M$ is useful; see, for example, \cite[p.346]{AnF}.

\begin{Lemma}\label{isoc_char}

Let $A$ be an artin algebra. Consider a module $M$ in $\mmod A$. Then
$\soc^i\hspace{-1.5pt}M=\{m\in M \mid J^im=0\},$ for $i\ge 0.$ In particular, $\soc^i\hspace{-1.5pt}M=M$ for $i\ge \ell\ell(A).$
    
\end{Lemma}

\smallskip

As an immediate consequence of Lemma \ref{isoc_char}, we obtain the following result.

\begin{Cor}\label{soci_LL}

Let $A$ be an artin algebra. Consider a module $M$ in $\mmod A$ of Loewy length $n$. Then $\ell\ell(\soc^i\hspace{-1.2pt}M)=i,$ and $\ell\ell(M/\soc^i\hspace{-1.2pt}M)=n-i$ for all $0\le i\le n$.

\end{Cor}

\smallskip

Applying Lemma \ref{isoc_char}, we obtain the following result. 

\begin{Lemma}\label{socs_sum}

Let $A$ be an artin algebra. Consider a module $M$ in $\mmod A$ such that $M=L+N$. If $\ell\ell(L)=s$, then $\soc^s\hspace{-1.2pt}M=L+\soc^s\hspace{-1.2pt}N.$

\end{Lemma}

\vspace{-1.5pt}

\noindent{\it Proof.} Assume that $\ell\ell(L)=s\ge 1$. By Lemma \ref{isoc_char}, $L=\soc^s \hspace{-1.2pt} L\hp ;$ consequently, $L+\soc^s\hspace{-1.2pt}N=\soc^s\hspace{-1pt}L+\soc^s\hspace{-1.2pt}N\subseteq \soc^s\hspace{-1.2pt}M.$ On the other hand consider $x=y+z\in \soc^s\hspace{-1.2pt}M,$ where $y\in L$ and $z\in N.$ Given $u\in J^s,$ by Lemma \ref{isoc_char}, $u z=ux - uy=0;$ and hence, $z\in \soc^s\hspace{-1.2pt}N.$ So, $\soc^s\hspace{-1.2pt}M\subseteq L+\soc^s\hspace{-1.2pt}N.$ The proof of the lemma is completed.

\smallskip

The following easy result will be frequently used in our later investigation. \vspace{-1pt}

\begin{Lemma}\label{soc_rad} Let $A$ be an artin algebra. Consider an indecomposable module $M$ in $\mmod A$. If $M$ is not simple, then $\soc M=\soc(\rad M);$ in particular, $\soc M\subseteq \rad M.$ 
%$\soc M=\soc(\rad M)$, and $\soc M=\rad M$ if and only if $\ell\ell(M)=2.$

%\begin{enumerate}[$(1)$]\vspace{-2.5pt}
%\item $\soc M=\soc(\rad M)\hp\hp;$ %In particular, $\soc M\subseteq \rad M.$
%\item $\soc M=\rad M$ if and only if $\ell\ell(M)\le 2.$ 
%\end{enumerate}

\end{Lemma}

\vspace{-1.5pt}

\noindent{\it Proof.} Assume that $M$ is not simple. Then, the zero submodule of $M$ is not maximal. Suppose that $\soc M\not\subseteq \rad M.$ Then $M$ has a simple sumodule $S$, which is not contained in a maximal submodule $L.$ Consequently $M=S\oplus L,$ absurd. Thus, $\soc M \subseteq \rad M.$ This yields $\soc M\subseteq \soc(\rad M) \subseteq \soc M.$ So, $\soc M=\soc(\rad M).$ The proof of the lemma is completed.

\subsection{\sc Standard duality} Our investigation relies crucially on the standard duality $D=\Hom_R(-, I_R): \mmod A \to \mmod A^\circ\hspace{-1pt},$ where $I_R$ is a minimal injective cogenerator for $\mmod A.$ Let $M$ be a module in $\mmod A.$ In order to explicitly compute the radical series and the socle series of $DM$, we consider $N^\perp\!: \, =\{f\in DM \mid f(N)=0\},$ for every submodule $N$ of $M$. The following result is well known.

\begin{Lemma}\label{dual_perp}

Let $A$ be an artin algebra. Consider a module $M$ in $\mmod A$. If $N, L$ are submodules of $M,$ then 
%. The following statement hold.
%If $\xymatrixcolsep{20pt}\xymatrix{0\ar[r] & N\ar[r]^q & M \ar[r]^p  & M/N \ar[r] & 0}$ is the canonical short exact sequence, then $N^\perp={\rm Im}(Dp),$ and consequently,

\begin{enumerate}[$(1)$]

\vspace{-.5pt}

\item $N^\perp \cong D(M/N)$ and $DM/N^{\perp}\cong DN;$ 

\vp\vp

\item $(L+N)^\perp=L^\perp \cap N^\perp$ and $(L\cap N)^\perp=L^\perp + N^\perp\hspace{-1pt}.$ 

\end{enumerate}\end{Lemma}

\smallskip

The following result will be used frequently in our later investigation.
%which allows us to compute the radical series and the socle series of the dual of a module in $\mmod A.$

\begin{Prop}\label{duality_radi_soci}

Let $A$ be an artin algebra. Given a module $M$ in $\mmod A$ and an integer $i\ge 0$, we have 

\begin{enumerate}[$(1)$]

\vspace{-1pt}

\item  $\rad^i\hspace{-.5pt}(DM)=(\soc^i\hspace{-1.5pt}M)^\perp$  and $\soc^i\hspace{-.5pt}(DM)=(\rad^i\hspace{-1.2pt}M)^\perp;$

\vp

\item $\rad^i(DM)\cong D(M /\soc^i \hspace{-1.2pt} M)$ and $\ntop^i(DM)\cong D(\soc^i \hspace{-1.2pt} M)\hp;$

\vp
\item $\soc^i(DM)\cong D(\ntop^i\hspace{-1.2pt}M)$ and $DM/\soc^i(DM)\cong D(\rad^i\hspace{-1.2pt}M).$

\end{enumerate} \end{Prop}

\noindent{\it Proof.} Statement (1) is well known; see, for example, \cite[Chapter 6, Exercise 7]{Webb}. And
Statements (2) and (3) follow easily from Statement (1) and Lemma \ref{dual_perp}(1). The proof of the proposition is completed.

%\smallskip

%As a consequence of Proposition \ref{duality_radi_soci}(2) and (3), we obtain the following result.

%\begin{Cor}\label{dual_ssf_tr}Let $A$ be an artin algebra. Consider a module $M$ in $\mmod A$. Then $\ntop(\rad(DM))\cong D(\soc(M/\soc M)).$\end{Cor}

%\noindent{\it Proof.} By Proposition \ref{duality_radi_soci}(2) and (3), $\rad(DM)\cong D(M/\soc M),$ and hence, $\ntop(\rad(DM))\cong \ntop(D(M/\soc M))\cong D(\soc(M/\soc M)).$ The proof of the corollary is completed.

\smallskip

We can compute the radical series of submodules of dual modules as follows.

\begin{Lemma}\label{rad_perp}

Let $A$ be an artin algebra. Consider a module $M$ in $\mmod A$ with a submodule $N$. If $\soc^i(M/N)=L/N$ for some $i\ge 0$, %where $N\subseteq L,$ 
then $\rad^i(N^\perp)=L^\perp.$ 

\end{Lemma}

\vspace{-1.5pt}

\noindent{\it Proof.} Assume that $\soc^i(M/N)=L/N$ with $N\subseteq L\subseteq M,$ for some $i\ge 0.$ Then, $L^\perp \subseteq N^\perp$. Since $J^iL\subseteq N$ by Lemma \ref{isoc_char}, we see that $\rad^i(N^\perp)= N^\perp J^i \subseteq L^\perp.$ Now, we deduce from Lemma \ref{dual_perp}(1) and Proposition \ref{duality_radi_soci}(2) that
\vspace{-2pt}
$$\rad^i(N^\perp) \cong \rad^i(D(M/N))
\cong D((M/N)/\soc^i(M/N))\cong D(M/L)\cong L^\perp.\vspace{-3pt}$$  

Therefore, $\rad^i(N^\perp)=L^\perp$. The proof of the lemma is completed.

%\vspace{-1pt}

\subsection{\sc Top-basis and soc-basis} Let $M$ be a module in $\mmod A$ and $e$ a primitive idempotent in $A.$ An element $m$ in $eM$ is called a {\it top-element} if $m\not\in \rad M$ and a {\it soc-element} if $0\ne m\in \soc M.$ Moreover, we put ${\rm ann}_{Ae}(m):=\{u\in Ae \mid u m=0\}.$ The following statement is evident. 
%
%\begin{Lemma}\label{smod} Let $A$ be an artin algebra with a primitive idempotent $e$. If $M$ is a module in $\mmod A$, then $M\cong Ae/Je$ if and only if $M=Am$, where $m\in eM$ such, for any $u\in Ae$, that $um=0$ if and only if $u\in Je.$\end{Lemma}

\begin{Lemma}\label{top-element} 

Let $A$ be an artin algebra with a primitive idempotent $e$. Consider a module $M$ in $\mmod A$. If $m$ is an element in $eM,$ then 

\begin{enumerate}[$(1)$]

\vspace{-1.5pt}

\item $m$ is a soc-element if and only if $Am\cong Ae/Je\hp;$ in this case, ${\rm ann}_{Ae}(m)=Je;$

\item $m$ is a top-element if and only if the element $m+\rad M$ in $\ntop M$ is such that $A(m+\rad M)\cong Ae/Je\hp;$ in this case, ${\rm ann}_{Ae}\hp (m+\rad M)=Je.$ 

\end{enumerate} \end{Lemma}

%\noindent{\it Proof.} (1) The sufficiency is evident. Assume that $0\ne m+\rad M\in \ntop M.$ Since $m\in eM$, we have a nonzero epimorphism $f: Ae\to A(m+\rad M),$ sending $e$ to $m+\rad M.$ Since $Je$ is the largest proper submodule of $Ae$, we have $\Ker f=Je.$ Therefore, $A(m+\rad M)\cong Ae/Je$ and ${\rm ann}_{Ae}(m+\rad M)=Je.$ 
%
%(2) The sufficiency is evident. Assume that $0\ne m \in \soc M.$ Since $\rad(\soc M)=0$, we see that $m$ is a top-element of $\soc M,$ and by Statement (1), $Am\cong Ae/Je$ and ${\rm ann}_{Ae}(m)=Je.$ The proof of the lemma is completed.

\smallskip

More generally, we introduce the following definition.

\begin{Defn}\label{top-basis} 

Let $A$ be an artin algebra. Consider a module $M$ in $\mmod A$.

\begin{enumerate}[$(1)$]

\vspace{-2pt}

\item  A set $\{m_1, \ldots, m_s\}$ \vp of top-elements in $M$ is called {\it top-free} if $\hspace{1pt} \sum_{i=1}^s \hspace{-2pt} A(m_i+\rad M)$ is a direct sum; and  a {\it top-basis} for $M$ if $\ntop M=\oplus_{i=1}^s A(m_i + \rad M)$. 

%\vp

\item A set $\{m_1, \ldots, m_s\}$ of soc-elements in $M$ is called {\it soc-free} if $\textstyle\sum_{i=1}^s Am_i$ is a direct sum, and a {\it soc-basis} for $M$ if $\soc M=\oplus_{i=1}^s Am_i$. 

\end{enumerate}\end{Defn}

\smallskip

The following result contains an alternative interpretation of a top-basis.

\begin{Lemma}\label{tb_gen}

Let $A$ be an artin algebra, and let $M$ be a module in $\mmod A$ with top-elements $m_i\in e_iM,$ where $e_i$ is a primitive idempotent in $A$, for $i=1, \ldots, s$. 

\begin{enumerate}[$(1)$]

\vspace{-2pt}

\item If $Ae_1, \ldots, Ae_s$ are pairwise non-isomorphic, then $\{m_1, \ldots, m_s\}$ is top-free.

\item The set $\{m_1, \ldots, m_s\}$ is a top-basis for $M$ if and only if it is a minimal gene\-rating set for $M.$

\end{enumerate} \end{Lemma}

\vspace{-1.5pt}

\noindent{\it Proof.} (1) Suppose that $Ae_i\not\cong Ae_j$ for $i\ne j$. Let $u_1m_1+\cdots +u_sm_s\in \rad M,$ where $u_i\in Ae_i$. Consider a complete set $\{f_1, \ldots, f_n\}$ of orthogonal primitive idempotents in $A$. Fix $1\le i\le s$. Given $1\le p\le n,$
we claim that $f_pu_i\in Je_i.$ If $Af_p \not\cong Ae_i$, then it is well known that $f_pu_i=f_pu_ie_i\in Je_i.$ Otherwise $Af_p \not\cong Ae_j,$ and hence, \vp \vp $f_pu_jm_j=(f_pu_je_j)m_j\in J M,$ for every $j\ne i.$ As a consequence, we have \vp $f_pu_im_i=f_p(\sum_{j=1}^su_jm_j) - \sum_{j\ne i} f_pu_jm_j \in \rad M.$ Then by Lemma \ref{top-element}(2), $f_pu_i\in Je_i.$ Our claim is established. Therefore, $u_i=\sum_{p=1}^nf_pu_i\in Je_i.$ \vspace{-.5pt} This shows that $\sum_{i=1}^s A(m_i+\rad M)$ is a direct sum, namely, $\{m_1, \ldots, m_s\}$ is top-free.

\vp \vp

(2) Since $\rad M$ is superfluous in $M,$ 
we see that $M=\sum_{i=1}^s A m_i $ if and only if $\ntop M=\sum_{i=1}^s A \overline m_i,$ where $\overline{m}=m+\rad M.$ And in this case, \vp $\ntop M=\oplus_{i=1}^s A \overline m_i$ if and only if $A \overline m_i \hp \cap \hp \sum_{j\ne i} A \overline m_j=0,$ for all $1\le i\le s$.
Since $A\overline m_i$ is simple by Lemma \ref{top-element}(2), 
the latter is equivalent to
$A\overline m_i \not\subseteq \sum_{j\ne i} A \overline m_j$ for all $1\le i\le s$, namely
$\{\overline m_1, \ldots, \overline m_s\}$ is a minimal generating set for $\ntop M$, or equivalently, $\{m_1, \ldots, m_s\}$ is a minimal generating set for $M$. The proof of the lemma is completed.

\smallskip

The following result allows us to construct top-bases and soc-bases for modules.

\begin{Lemma}\label{find_top-basis} Let $A$ be an artin algebra. Consider a module $M$ in $\mmod A$ and primitive idempotents $e_1, \ldots, e_s$ in $A$. Then 

\begin{enumerate}[$(1)$]

\vspace{-2.5pt}

\item $\ntop M \hspace{-1pt} \cong \hspace{-1pt} Ae_1  / \hspace{-1.5pt} Je_1 \oplus \cdots \oplus  Ae_r  / \hspace{-1.5pt} Je_r$ if and only if $M$ admits a top-basis $\{m_1, \ldots, m_r\},$ where $m_i\in e_iM,$ for $i=1, \ldots, r.$

\item $\soc M \hspace{-1pt} \cong \hspace{-1pt} Ae_1  / \hspace{-1.5pt} Je_1 \oplus \cdots \oplus  Ae_r  / \hspace{-1.5pt} Je_r$ if and only if $M$ admits a soc-basis $\{m_1, \ldots, m_r\},$ where $m_i\in e_iM,$ for $i=1, \ldots, r.$

\end{enumerate} \end{Lemma}

\noindent{\it Proof.} We shall only prove Statement (1). The sufficiency follows from Lemma  \ref{top-element}(2). Assume that $\ntop M = S_1 \oplus \cdots \oplus S_r$, where $S_i\cong Ae_i/Je_i.$ Note that 
$S_i=A(m_i+\rad M),$ for some $m_i+\rad M=e_i(m_i+\rad M)$, for $i=1, \ldots, r.$ Thus, $\{e_1m_1, \ldots, e_rm_r\}$ is a top-basis for $M.$
%(2) The sufficiency follows from Lemma \ref{top-element}(2). Let $\soc M = S_1 \oplus \cdots \oplus S_t$, where $S_i\cong Ae_i/Je_i,$ for $i=1, \ldots, t$. In particular, we have non-zero epimorphisms $f_i: Ae_i\to S_i$, for $i=1, \ldots, t.$ Write $f_i(e_i)=m_i$. Then $m_i=e_im_i\in e_iM$ and $S_i\cong Am_i$, for $i=1, \ldots, t.$ By definition, $\{m_1, \ldots, m_t\}$ is a soc-basis for $M$. 
The proof of the lemma is completed.

\smallskip

The following result, adapted from the finite-dimensional case; see \cite[(1.1)]{LiM}, tells us how to construct projective cover of modules. 

\vspace{-1pt}

%For each simple module $S$ in $\mmod A$, we fix a projective cover $\pi\hspace{-1pt}_{_S}: P_S\to S$ and an injective envelope $\iota\hspace{-1pt}_{_S}: S\to I_S.$

%To conclude this subsection, we show how to construct a protective cover and injective envelope for modules in $\mmod A$.

%\begin{Lemma}\label{maps_to_inj} Let $A$ be an artin algebra with a primitive element  $e$. If $M$ is a module in $\mmod A$ with a soc-element $m\in eM$, then there exists a map $g: M \to D(eA),$ sending $m$ to $e^*.$\end{Lemma} \noindent{\it Proof.} By Corollary \ref{soc_indec_inj}, $Ae^*=\soc D(eA).$ Let $M\in \mmod A$ with a soc-element $m\in eM$. If $am=bm$ with $a,b\in A$, then $(a-b)m=0$, and by Lemma \ref{top-element}, $a-b\in J.$ Then, by Lemma \ref{soc_indec_inj}, $(a-b)e^*=0.$ So $ae^*=be^*$. Therefore, we have a well-defined map $h: Am\to D(eA): am \mapsto a e^*,$ which is clearly $A$-linear. Since $D(eA)$ is injective, we have a commutative diagram  $$\xymatrixrowsep{30pt}\xymatrixcolsep{30pt}\xymatrix{ Am \ar@{>}[r]^q \ar[dr]_h& M \ar@{.>}[d]^g \\ & D(eA)}$$ in $\mmod A$. So $g(m)=gq(m)=h(m)=e^*.$ The proof of the lemma is completed.

\begin{Prop}\label{tb-projc-ienv}

Let $A$ be an artin algebra. Consider a module $M$ in $\mmod A$ with $m_i\in e_iM,$ where $e_i$ is a primitive idempotent in $A,$ for $i=1, \ldots, s.$ Then, $M$ admits a projective cover $f\hspace{-.5pt}: Ae_1\oplus \cdots \oplus Ae_s \to M,$ sending $e_i$ to $m_i,$ if and only if $\{m_1, \ldots, m_s\}$ is a top-basis for $M.$

\end{Prop}

\vspace{-1.5pt}

\noindent{\it Proof.} Suppose that $\ntop M= \oplus_{i=1}^s A(m_i+\rad M).$ Since $M=\sum_{i=1}^s Am_i;$ see (\ref{tb_gen}), we have an epimorphism $f: Ae_1 \oplus \cdots \oplus A e_s \hspace{-.5pt}\to \hspace{-.5pt} M,$ sending $e_i$ to $m_i.$ Assume that  $\sum_{i=1}^s u_i m_i=0,$ where $u_i\in Ae_i.$ In particular, $\sum_{i=1}^s (u_im_i + \rad M)=0.$ Thus, $u_im_i\in \rad M,$ and by Lemma \ref{top-element}(2), $u_i\in Je_i$, for $i=1, \ldots, s$. Thus $\Ker f\in Je_1 \oplus \cdots \oplus Je_s,$ that is, $f$ is a projective cover for $M$.

Let $f\hspace{-.5pt}: \hspace{-.5pt} \oplus_{i=1}^s Ae_i \to M$ be a projective cover, sending $e_i$ to $m_i$. In particular, $\ntop M=\sum_{i=1}^s A(m_i+\rad M)$. Assume that \vp $\sum_{i=1}^s u_i(m_i+\rad M)=0,$ where $u_i\in Ae_i.$ Then $\sum_{i=1}^s u_im_i\in \sum_{j=1}^s J m_i,$ say, $\sum_{i=1}^s u_i m_i=\sum_{i=1}^s v_im_i$, where $v_i\in Je_i.$ Then, $(u_1-v_1, \ldots, u_s-v_s)\in \Ker f \subseteq \oplus_{i=1}^s Je_i.$ Then, $u_i\in Je_i,$ and 
$u_i(m_i+\rad M)=0,$ for $1\le i\le s$. Therefore, $\ntop M= \oplus_{i=1}^s A(m_i+\rad M),$ 
that is, $\{m_1,\dots,m_s\}$ is a top-basis for $M$. 
The proof of the proposition is completed.

%(2) Suppose that $\{m_1, \ldots, m_s\}$ is a soc-basis for $M$. Then $\soc M=\oplus_{i=1}^sAm_i.$ Since $Am_i$ is simple, by Lemma \ref{maps_to_inj}, we have monomorphisms $h_i: Am_i \to D(e_iA),$ sending $m_i$ to $e_i^*,$ for each $1\le i\le s$. This yields a monomorphism$$h=\oplus_{i=1}^s h_i: \oplus_{i=1}^sAm_i \to \oplus_{i=1}^sD(e_iA).$$ Consider the inclusion map $q: \soc M \to M$. Since $\oplus_{i=1}^sD(e_iA)$ is injective, we have the following commutative diagram: $$\xymatrixrowsep{30pt}\xymatrixcolsep{30pt}\xymatrix{\oplus_{i=1}^sAm_i \ar@{>}[r]^q \ar[dr]_h& M \ar@{.>}[d]_g \\& \oplus_{i=1}^sD(e_iA).}$$Since $q$ is an essential monomorphism and $h$ is a monomorphism, we see that $g$ is a monomorphism. By Lemmas \ref{soc_indec_inj} and \ref{top-element}(2), we see that$$\soc(\oplus_{i=1}^sD(e_iA))=\oplus_{i=1}^s\soc(D(e_iA))=\oplus_{i=1}^s Ae_i^*.$$ Since $e_i^*=h_i(m_i)=h(m_i)=g(m_i)$, for $i=1, \ldots, s,$ $\soc(\oplus_{i=1}^sD(e_iA))\subseteq \nim(g)$. Therefore, $g$ is an essential monomorphism, that is, an injective envelope of $M.$ Conversely, assume that there exists an injective envelope $g: M\to \oplus_{i=1}^s D(e_iA)$, sending $m_i$ to $e_i^*.$ In view of Lemma \ref{soc_indec_inj}, we see that $$\textstyle g({\oplus}_{i=1}^s Am_i)=\oplus_{i=1}^s Ae_i^*=\oplus_{i=1}^s \soc(D(e_iA))=\soc (\oplus_{i=1}^s D(e_iA)).$$By Lemma \ref{ess_mono}, $g(\oplus_{i=1} Am_i)=g(\soc M)$. Since $g$ is a monomorphism, we see that $\soc M=\oplus_{i=1}^s Am_i.$ That is, $\{m_1, \ldots, m_s\}$ is a soc-basis for $M$. The proof of the proposition is completed.

\subsection{\sc Local and colocal modules} A module $M$ in $\mmod A$ is called {\it local} if $\ntop M$ is simple, or equivalently, if $\rad M$ is the largest proper submodule of $M.$ Dually, $M$ is called {\it colocal} if $\soc M$ is simple, or equivalently, if $\soc M$ is the smallest nonzero submodule of $M.$ 
Clearly, $M$ is local or colocal if and only if $DM$ is colocal or local respectively. 
The following statement is evident. 
%follows from Lemmas \ref{top-element}, \ref{tb_gen} and \ref{find_top-basis}.
\vspace{-1pt}

\begin{Lemma}\label{loc_gen}

Let $A$ be an artin algebra with a primitive idempotent $e.$ 

\begin{enumerate}[$(1)$]

\vspace{-2pt}

\item A module $M$ in $\mmod A$ is local with $\ntop M \cong Ae/Je$ if and only if $M=Am$ with $0\ne m\in eM.$ In this case, every nonzero quotient module of $M$ is local.

\item A module $M$ in $\mmod A$ is colocal with $\soc M\cong Ae/Je$ if and only if $\soc M=Am$ with $0\ne m\in eM.$ In this case, every nonzero submodule of $M$ is colocal.

\end{enumerate} \end{Lemma}

%\smallskip
%The following statement is evident.
%\begin{Lemma}\label{sim_soc}
%Let $A$ be an artin algebra. Consider a module $M$ in $\mmod A$.
%\begin{enumerate}[$(1)$]\vspace{-1.5pt}
%\item If $M$ is local, then 
%$\ntop(M/N)\cong \ntop M$, for any proper submodule $N$ of $M$. 
%\item If $M$ is colocal, then $\soc N=\soc M$, for any nonzero submodule $N$ of $M$. 
%\end{enumerate}\end{Lemma}

\smallskip

 The following easy result will be needed in our later investigation.

 \vspace{-1pt}

\begin{Lemma}\label{maxi_local}

Let $A$ be an artin algebra.  Consider a module $M$ in $\mmod A$.

\begin{enumerate}[$(1)$]

\vspace{-2pt}

\item If $M$ is local such that $\ell(M)\ge \ell(P)$ for every indecomposable projective module $P$ in $\mmod A$, then $M$ is projective.

\item If $M$ is colocal such that $\ell(M)\ge \ell(I)$ for every indecomposable injective module $I$ in $\mmod A$, then $M$ is injective.
    
\end{enumerate} \end{Lemma}

\vspace{-2.5pt}

\noindent{\it Proof.} We only prove Statement (1). Assume that $\ntop M\cong Ae/Je,$ where $e$ is a primitive idempotent in $A$. By Lemma \ref{loc_gen}, $M=Am$ for some $0\ne m\in eM$. Then we have a projective cover $f: Ae\to M$, where $Ae$ is indecompo\-sable. If $\ell(M)\ge \ell(Ae)$, then $f$ is an isomorphism. The proof of the lemma is completed.

%\begin{Lemma}\label{coloc_ll} Let $A$ be an artin algebra. Consider a module $M$ of Lowey length $n$ in $\mmod A$. If $M$ is colocal, then $\soc M=\rad^{n-1}M.$ \end{Lemma}
%\noindent{\it Proof.} By definition, $0\ne \rad^{n-1}M\subseteq \soc M$. If $\soc M$ is simple, then $\rad^{n-1}M=\soc M.$ The proof of the lemma is completed.

\vspace{-2pt}

\subsection{\sc Uniserial modules} A module $M$ in $\mmod A$ is called {\it uniserial} if any two of its submodules are comparable. This is the case if and only if $M$ has an unique composition series, or equivalently, $\ell(M)=\ell\ell(M)$. An artin algebra $A$ is called a {\it Nakayama algebra} if all indecomposable projective and injective modules in $\mmod A$ are uniserial. We will need the following easy result.

\begin{Prop}\label{uni_pi}

Let $A$ be an artin algebra. Consider a module $M$ in $\mmod A$ of length $s$ with projective cover $P$ and injective envelope $I$. 

\begin{enumerate}[$(1)$]

\vspace{-2pt}

\item If $\ntop^s\hspace{-1.2pt}P$ is uniserial, then $M\cong \ntop^s\hspace{-1pt}P$. In particular, $M$ is uniserial.

\item If $\soc^s\hspace{-1.2pt}I$ is uniserial, then $M\cong \soc^s\hspace{-1pt}I$. In particular, $M$ is uniserial.

\end{enumerate}\end{Prop}

\vspace{-1.5pt}

\noindent{\it Proof.} We only prove Statement (1). Let $f: P\to M$ be a projective cover, which induces an epimorhism $\bar{f}: \ntop^s\hspace{-1pt}P\to \ntop^s\hspace{-1pt}M.$ 
Assume that $\ntop^s\hspace{-1.5pt}P$ is uniserial. Then, so is $\ntop^s\hspace{-1pt}M.$ If $\rad^s\hspace{-1pt}M\ne 0$, then $s<\ell\ell(M)\le \ell(M)=s,$ absurd. So, we have an epimorphism $\bar{f}: \ntop^s\hspace{-1pt}P\to M$. Then, $s=\ell(M)\le \ell(\ntop^s\hspace{-1pt}P)=s\hp ;$ see (\ref{topi_LL}). Thus $\ell(M)=\ell(\ntop^s\hspace{-1pt}P).$ Hence, $\ntop^s\hspace{-1pt}P\cong M$. 
%
%(2) Note that $DM$ is of length $s$ with a projective cover $DI$. By Proposition \ref{duality_radi_soci}(2), $\ntop^s\hspace{-.5pt}(DI)\cong D(\soc^s\hspace{-1pt}I).$ If $\soc^s\hspace{-1pt}I$ is uniserial, then so is $\ntop^s\hspace{-.5pt}(DI);$ by Statement (1), $DM\cong \ntop^s\hspace{-.5pt}(DI)\cong D(\soc^s\hspace{-1pt}I)$. 
The proof of the proposition is completed.

\subsection{\sc Projective-injective modules} A module in $\mmod A$ is called {\it projective-injective} if it is projective and injective. We shall need the following easy result.

\begin{Lemma}\label{pi_corresp}

Let $A$ be an artin algebra with primitive idempotents $e_1, e_2$. 

\begin{enumerate}[$(1)$]

\vspace{-2.2pt}

\item If $Ae_1$ is injective such that $\soc(Ae_1)\cong Ae_2/Je_2$, then
$e_2A$ is injective such that $\soc(e_2A)=e_2 \hspace{1pt}\soc(Ae_1) A\cong e_1A/e_1J.$

%\vp

\item If $e_1A$ is injective such that $\soc(e_1A)\cong e_2A/e_2J,$ then $Ae_2$ is injective such that $\soc(Ae_2)=A \hspace{1pt} \soc(e_1A) e_2 \cong Ae_1/Je_1.$

\end{enumerate}\end{Lemma}

\vspace{-1.5pt}

\noindent{\it Proof.} We only prove Statement (1). Let $Ae_1$ be injective and $\soc(Ae_1)\cong Ae_2/Je_2$. Then, $Ae_1$ is an injective envelope of $Ae_2/Je_2,$ and consequently, $e_2A \cong D(Ae_1).$ By Proposition \ref{duality_radi_soci}, $\soc(e_2A) 
%\cong D(\ntop (Ae_1))
\cong e_1A/e_1J.$ Put $n:=\ell\ell(Ae_1)=\ell\ell(e_2A).$ Since $Ae_1$ and $e_2A$ are colocal, we see that $\soc(Ae_1)=J^{n-1}e_1$ and $\soc(e_2A)=e_2J^{n-1}.$ Thus, $0\ne e_2 \hspace{1pt}\soc(Ae_1) A = e_2 \hspace{1pt} J^{n-1}e_1 A = \soc(e_2A) \,e_1A \subseteq \soc(e_2A).$ Therefore, $\soc(e_2A)=e_2 \hspace{1pt}\soc(Ae_1) \, e_1 A\cong e_1A/e_1J.$ The proof of the lemma is completed. 

\vspace{-1pt}

\subsection{\sc Short exact sequences} For convenience of reference, we state several short exact sequences related to submodules of modules and leave their proofs to the reader. Most almost split sequences appearing in this paper are of these forms. 

\begin{Lemma}\label{ses_sum_itsec}

Let $A$ be an artin algebra. Consider a module $M$ in $\mmod A$ with two submodules $M_1$ and $M_2.$ Then we have two short exact sequences as follows$\,:$

\begin{enumerate}[$(1)$]

\vspace{-2.5pt}
    
\item $\xymatrix{0\ar[r]& M_1\cap M_2 \ar[r]^-{(q_1, \, q_2)^T}  & M_1 \oplus M_2 \ar[r]^-{(-q_3, \, q_4)} & M_1+M_2 \ar[r] & 0,}$

\vspace{1.5pt}

\item[] \hspace{2pt} where the $q_i$ are inclusion maps$\,;$ 

\vspace{-1pt}

\item $\xymatrixcolsep{26pt} \xymatrix{0\ar[r]& M/(M_1\cap M_2)\ar[r]^-{(p_1, \, p_2)^T}  & M/M_1 \!\oplus\! M/M_2 \ar[r]^-{(-p_3, \, p_4)} & M/(M_1+M_2) \ar[r] & 0,}$ 

\item[] where the $p_j$ are canonical projections. 

\end{enumerate}\end{Lemma}

%\smallskip

\vspace{1pt}

We also need the following straightforward result.

\vspace{-2pt}

\begin{Lemma} \label{ses_mix}

Let $A$ be an artin algebra. If $M$ is a module in $\mmod A$ with sub\-modules $N_0\subseteq N\subseteq L,$ then we have a short exact sequence \vspace{-6pt}  
$$\xymatrixcolsep{26pt} \xymatrix{0\ar[r]& L/N_0\ar[r]^-{(q_0, \, p_0)^T}  & M/N_0\oplus L/N  \ar[r]^-{(p, \, -q)} & M/N \ar[r] & 0,}\vspace{-5pt} $$  where $q_0, q$ are inclusion maps, and $p_0, p$ are canonical projections.

\end{Lemma}

%\begin{Lemma}\label{ses_mix} Let $A$ be an artin algebra with $M$ a module in $\mmod A$. If $S, N$ are submodules of $M$ with $S\subseteq N,$ then there exists a short exact sequence $$\xymatrixcolsep{32pt} \xymatrix{0\ar[r]& N\ar[r]^-{(q_1, \, p_1)^T}  & M\oplus N/S  \ar[r]^-{(-p_2, \, q_2)} & M/S \ar[r] & 0,}$$ where $q_1, q_2$ are inclusion maps, and $p_1, p_2$ are canonical projections.\end{Lemma}

\section{\sc Computing almost split sequences} 

Almost split sequences are explicitly computed for Nakayama algebras; see, for example, \cite[p.~ 197]{ARS}. Most prominently, for any string algebra given by a quiver with relations, Butler and Ringel described all almost split sequences by using the combinatorial structure of the bound quiver; see \cite{BuR}. In this section, we compute almost split sequences for certain uniserial modules over a general artin algebra, using only the definition of the dual of the transpose functor and some basic properties of almost split sequences. These ideas and results will be applied in Section 4 to quadri-biserial algebras, see (\ref{4str_alg_def}), and may be applied in the future to a broader class of biserial algebras and even to multiserial algebras studied in \cite{VHW}.

\subsection{\sc General criteria} Almost split sequences admit various characterizations, such as those stated in \cite[Chapter V]{ARS}. Applying Proposition 2.2 therein, we derive two sufficient conditions for a short exact sequence to be almost split.

\begin{Lemma}\label{ass-sim}

Let $A$ be an artin algebra. Consider a nonsplit short exact sequence \vspace{-9.5pt}
$$\xymatrix{\delta: \hspace{-20pt} & 0\ar[r] & X\ar[r]^-f & Y \ar[r]^-g & Z\ar[r] & 0}\vspace{0pt}$$ in $\mmod A$ such that $X\cong D{\rm Tr}Z$ or $Z\cong {\rm Tr} D X$. 

\begin{enumerate}[$(1)$]

\vspace{-2pt}

\item If $X$ or $Z$ is simple, then $\delta$ is an almost split sequence.

\item If $f$ or $g$ is irreducible, then $\delta$ is an almost split sequence.

\end{enumerate}\end{Lemma}

\vspace{-1.5pt}

\noindent{\it Proof.} If $X$ or $Z$ is simple, then both $X$ and $Z$ are indecomposable. In view of 
Proposition 2.2 in \cite[Chapter V]{ARS}, we obtain immediately Statement (1) from Schur's Lemma. For Statement (2), we consider only the case where $f$ is irreducible. Then, $Z$ is indecomposable and not projective; see \cite[(V.5.6)]{ARS}. Hence, $X$ is indecomposable and not injective. So we have a minimal left almost split monomorphism $(f, \, f')^T: X\to Y\oplus Y'.$ Since $Z\cong {\rm Tr}D X$, we obtain from Proposition 1.14 in \cite[Chapter V]{ARS} an almost split sequence \vspace{-5pt}
$$\hspace{15pt}\xymatrix{0\ar[r] & X\ar[r]^-{(f, \, f')^T} & Y\oplus Y' \ar[r]^-{(g', \, g'')} & Z\ar[r] & 0.}\vspace{-3pt}$$ 
Then, $\ell(Y)+\ell(Y')=\ell(Z)+\ell(X)=\ell(Y)$. Thus $Y'=0.$ So $f$ is minimal left almost split, and $\delta$ is an almost split sequence. The proof of the proposition is completed.

\subsection{\sc Uniserial end-terms} We begin by computing almost split sequences which start with uniserial quotients of submodules of indecomposable projective modules.

\begin{Theo}\label{ass_uni_topi}

Let $A$ be an artin algebra. Consider an indecomposable projective $P$ in $\mmod A$ with a proper submodule $M$ such that $\ntop^s\hspace{-1pt}M$ is uniserial and $\rad^{s-1}\hspace{-1pt}M\not\subseteq \rad^{s+1}\hspace{-1pt}P,$ for some $s\ge 1$. If $\soc^{s+1}\hspace{-.5pt}(I_{\hspace{.5pt}\soc(\ntop^s\hspace{-1pt}M)})$ is uniserial, then we have an almost split sequence \vspace{-5pt}
$$\xymatrixcolsep{28pt}\xymatrix{
0 \ar[r] & \ntop^s\hspace{-1pt}M \ar[r]^-{(q_1, \,p_1)^T\hspace{-2pt}}& (P/\rad^s\hspace{-1pt}M) \oplus 
\ntop^{s-1}\hspace{-1pt}M \ar[r]^-{(-p_2, \,q_2)} & P/\rad^{s-1}\hspace{-1pt}M \ar[r] & 0,
} \vspace{-4pt} $$ where $q_1, q_2$ are inclusion maps, and $p_1, p_2$ are canonical projections. 

\end{Theo}

\vspace{-1pt}

\noindent{\it Proof.} Considering the submodules $\rad^s\hspace{-1pt}M \subseteq \rad^{s-1}\hspace{-1pt}M \subseteq M$ of $P$, we obtain from Lemma \ref{ses_mix} the short exact sequence stated in the theorem, which clearly does not split. Note that %$\ntop^s\hspace{-1pt}M$ is indecomposable with 
$\soc(\ntop^s\hspace{-1pt}M)=\rad^{s-1}\hspace{-1.5pt}M/\rad^s\hspace{-1pt}M,
$ which is simple. Thus, every radical map $f: \ntop^s\hspace{-1pt}M\to \ntop^s\hspace{-1pt}M$ factors through $p_1: \ntop^s\hspace{-1pt}M\to \ntop^{s-1}\hspace{-1pt}M.$ Assume that $\soc^{s+1}\hspace{-.5pt}(I_{\hspace{.5pt}\soc(\ntop^s\hspace{-1pt}M)})$ is uniserial.  By Proposition 2.2 in \cite[Chapter V]{ARS}, it suffices to show that ${\rm Tr}D(\ntop^s\hspace{-1pt}M)\cong P/\rad^{s-1}\hspace{-1pt}M.$

We may assume that $P=Ae$ and $\soc(\ntop^s\hspace{-1pt}M)\cong Ae_1/Je_1,$ where $e, e_1$ are primitive idempotents in $A$. Observing that $\ntop(\rad^{s-1}\hspace{-1pt}M)\cong Ae_1/Je_1$, we infer from Lemma \ref{loc_gen}(1) that $\rad^{s-1}\hspace{-1pt}M=Au,$ for some $u\in e_1 (\rad^{s-1}M)e.$ By the hypothesis stated in the theorem, $u \in e_1J^se \backslash e_1J^{s+1} e$. Since $I_{\soc(\ntop^s\hspace{-1pt}M)}\cong D(e_1A),$ by Proposition \ref{duality_radi_soci}(2), $e_1A/e_1J^{s+1}\cong D(\soc^{s+1}\hspace{-.5pt}(I_{\soc(\ntop^s\hspace{-1pt}M)})),$ which is uniserial. Thus $e_1A/e_1J^s$ is uniserial, and $e_1J^s/e_1J^{s+1}$ is the simple top of $e_1J^s.$ Then $\{u\}$ is a top-basis for $e_1J^s.$ By Proposition \ref{tb-projc-ienv}, we have a minimal projective presentation \vspace{-10pt} $$\xymatrix{eA \ar[r]^{L_u} & e_1A \ar[r] & e_1A/e_1J^s\ar[r] & 0,}\vspace{-2.5pt}$$ where $L_u$ is the left multiplication by $u$. \vspace{0.5pt} Note that $D(\ntop^s\hspace{-1.5pt}M)$ is of length $s$ with $\ntop(D(\ntop^s\hspace{-1pt}M))\cong D(\soc(\ntop^s\hspace{-1pt}M))= D(\rad^{s-1}\hspace{-2pt}M/\rad^s\hspace{-1pt}M)\cong e_1A/e_1J.$ \vp Thus, $e_1A$ is the projective cover of $D(\ntop^s\hspace{-1.5pt}M)$. Since $e_1A/e_1J^s$ is uniserial, by Lemma \ref{uni_pi}(1), $e_1A/e_1J^s\cong D(\ntop^s\hspace{-1pt}M).$ This yields a minimal projective presentation \vspace{-5pt}
$$\xymatrix{Ae_1 \ar[r]^{R_u} & Ae \ar[r] \ar[r] & {\rm Tr}D(\ntop^s\hspace{-1pt}M) \ar[r] & 0,}\vspace{-5pt}$$ where $R_u$ is right multiplication by $u$. Since ${\rm Im}(R_u)=Au=\rad^{s-1}\hspace{-1pt}M$, we conclude that ${\rm Tr}D(\ntop^s\hspace{-1pt}M)\cong Ae/Au=P/\rad^{s-1}\hspace{-1pt}M.$ The proof of the theorem is completed.

\vspace{5pt}

\noindent{\sc Example.} Let $A=kQ/I,$ where $k$ is a field, $Q$ is the quiver \vspace{-5pt}
$$\xymatrix@!=8pt{a\ar[dr] && 2 \ar[r]^\beta & 3 \ar[dr]^\gamma && d \\ b \ar[r] & 1 \ar[ur]^\alpha \ar[r]^\delta &4\ar[rr]^\varepsilon && 5\ar[ur]^\eta \ar[r]^\zeta & c,}\vspace{-5pt}$$ and $I=\langle \gamma \beta\alpha-\varepsilon \delta\rangle.$ 
Note that $B=kQ'/I'$ is a special biserial algebra as defined in \cite{SKW}, where $Q'$ is the convex hull of $1$ and $5$ in $Q,$ and $I'=k(\gamma \beta\alpha-\varepsilon \delta).$ Given a string $\omega$ in $(Q', I')$, we denote by $M(\omega)$ the associated string $B$-module; see \cite{HuL}, which extends to an $A$-module. Let $M=k(\alpha, \beta \alpha, \gamma \beta \alpha, \eta\gamma \beta \alpha, \zeta \gamma \beta \alpha),$ a proper submodule of $P_1$ such that  $\rad M\not\subseteq \rad^3\hspace{-1pt}P_1$ and $\ntop^2\hspace{-1pt}M\cong M(\beta),$ which is uniserial. Moreover $\ntop M\cong S_2,$ while $P_1/\rad M\cong M(\alpha\delta^{-1})$ and $P_1/\rad^2\hspace{-1pt}M\cong M(\beta\alpha \delta^{-1}).$ Since $\soc(\ntop^2\hspace{-1pt}M)\cong S_3$ and $\soc^3I_3\cong M(\beta \alpha),$ which is uniserial, \vspace{-1.5pt} it follows from Theorem \ref{ass_uni_topi} that $\xymatrixcolsep{20pt}\xymatrix{
0 \ar[r] & M(\beta) \ar[r]& M(\beta\alpha \delta^{-1}) \oplus S_2 \ar[r] & M(\alpha\delta^{-1}) \ar[r] & 0}$ is an almost split sequence in $\mmod A.$

\smallskip

As a special case of Theorem \ref{ass_uni_topi}, we obtain the following useful result.

\begin{Theo}\label{ass_st_mod}

Let $A$ be an artin algebra. Consider an indecomposable projective module $P$ in $\mmod A$ with $\rad P=M\oplus N,$ where $M$ is uniserial of length $s >0$. If $\soc^{s+1}(I_{\soc M})$ is uniserial, then there exists an almost split sequence \vspace{-5.5pt}
$$\xymatrixcolsep{30pt}\xymatrix{
0 \ar[r] & M \ar[r]^-{(q_1, \,p_1)^T}& P \oplus (M/\soc M) \ar[r]^-{(-p_2, \,q_2)} & P/\soc M \ar[r] & 0
}\vspace{-5pt}$$ in $\mmod A$, where $q_1, q_2$ are inclusion maps,and $p_1, p_2$ are canonical projections. 

\end{Theo}

\noindent{\it Proof.} Observe that $\rad^s\hspace{-1pt}M=0$ and $\rad^{s-1}\hspace{-1.5pt}M=\soc M.$ Hence $\ntop^s\hspace{-1pt}M=M$ and $\ntop^{s-1}\hspace{-1.5pt}M=M/\soc M$. By the hypothesis, $\rad^{s+1}\hspace{-1pt}P=\rad^{s}\hspace{-1pt}M \oplus \rad^{s}\hspace{-1pt}N.$ Suppose that $\rad^{s-1}\hspace{-1.5pt}M\subseteq \rad^{s+1}\hspace{-1.5pt}P.$ Since $M\cap N=0$, in view of the modular law, we see that \vspace{-3pt} $$\rad^{s-1}\hspace{-1.5pt}M = \rad^{s-1}\hspace{-1pt}M\cap (\rad^{s}\hspace{-1pt}M+ \rad^{s}\hspace{-1pt}N)=\rad^{s}\hspace{-1.5pt}M + \rad^{s-1}\hspace{-1.5pt} M \cap \rad^{s}\hspace{-1.5pt}N=\rad^{s}\hspace{-1.5pt}M=0, \vspace{-3.5pt}$$ absurd. Therefore, $\rad^{s-1}\hspace{-1.5pt}M\not\subseteq \rad^{s+1}\hspace{-1.5pt}P.$ By Theorem \ref{ass_uni_topi}, we obtain an almost split sequence as stated in the theorem. The proof of the theorem is completed.

\smallskip

The following result is another application of Theorem \ref{ass_uni_topi}, which covers all almost split sequences over a Nakayama algebra; see \cite[p.197]{ARS}.

\begin{Theo}\label{ass_N_alg}

Let $A$ be an artin algebra. Consider a uniserial module $M$ in $\mmod A$ of length $s >1$ such that $\ntop^s\hspace{-.5pt}(P_{\hspace{.5pt}\ntop M})$ and $\soc^s\hspace{-.5pt}(I_{\hspace{.5pt}\soc M})$ are uniserial. Then we have an almost split sequence \vspace{-6pt}
$$\xymatrixcolsep{30pt}\xymatrix{
0\ar[r] & \rad M \ar[r]^-{(q_1, p_1)^T} & M \oplus (\rad M/\soc M) \ar[r]^-{(-p_2, q_2)} & M/\soc M \ar[r] & 0,}
\vspace{-6.5pt}$$ where $q_1, q_2$ are inclusion maps, and $p_1, p_2$ are canonical projections. 

\end{Theo}

\vspace{-1.5pt}

\noindent{\it Proof.} Write $P=P_{\hspace{.8pt}\ntop M}$ and $N=\rad P.$ By Proposition \ref{uni_pi}(1), we may assume that $M=\ntop^s\hspace{-1.5pt}P=P/\rad^{s-1}\hspace{-1.5pt}N$. Then $\soc M=\rad^{s-1}\hspace{-1.5pt}M=\rad^{s-2}\hspace{-1.5pt}N/\rad^{s-1}\hspace{-1.5pt}N.$ Therefore, there exists a canonical isomorphism $\varphi: M/\soc M \to P/\rad^{s-2}\hspace{-1pt}N.$ Observing that $\rad P/\rad^{s-2}\hspace{-1.5pt}N=\ntop^{s-2} \hspace{-1.5pt} N$, we see that $\varphi$ restricts to an isomorphism $\psi: \rad M/\soc M\to \ntop^{s-2} \hspace{-1.5pt} N.$ Thus, we have a commutative diagram \vspace{-6pt}
$$\xymatrixcolsep{30pt}\xymatrix{ 0\ar[r] & \rad M \ar[r]^-{(q_1, p_1)^T} \ar@{=}[d] & M \oplus (\rad M/\soc M) \ar[r]^-{(-p_2, q_2)} \ar[d]^{{\rm id} \oplus \psi}& M/\soc M \ar[r] \ar[d]^\varphi & 0\\ 
0 \ar[r] & \ntop^{s-1}\hspace{-1.5pt}N \ar[r]^-{(\tilde{q}_1, \,\tilde{p}_1)^T}& (P/\rad^{s-1}\hspace{-1.2pt}N) \oplus 
\ntop^{s-2}\hspace{-1.5pt}N \ar[r]^-{(-\tilde{p}_2, \,\tilde{q}_2)} & P/\rad^{s-2}\hspace{-1.5pt}N \ar[r] & 0} \vspace{-5pt} $$ in $\mmod A,$ where $\tilde{q}_1, \tilde{q}_2$ are inclusion maps, and $\tilde{p}_1, \tilde{p}_2$ are canonical projections. 

Now $\ntop^{s-1}\hspace{-1pt}N$ is uniserial of length $s-1$ by hypothesis, and $\rad^{s-2}\hspace{-1pt}N \not\subseteq \rad^s\hspace{-1pt}P$. Since $\soc(\ntop^{s-1}\hspace{-1.5pt}N)=\soc(\rad M)=\soc M;$ see (\ref{soc_rad}), $\soc^s(I_{\soc(\ntop^{s-1}\hspace{-1.5pt}N)})$ is uniserial by hypothesis. In view of Theorem \ref{ass_uni_topi}, we conclude that the lower sequence in the above diagram is an almost split sequence, and so is the upper sequence. The proof of the theorem is completed. 

\smallskip

Dually, we compute almost split sequences which end with uniserial submodules of quotients of indecomposable injective modules. \vspace{-2pt}

\begin{Theo}\label{ass_uni_soci}
    
Let $A$ be an artin algebra. Consider an indecomposable injective module $I$ in $\mmod A$ with a proper quotient $I/N$ such that $\soc^s\hspace{-.5pt}(I/N)$ is uniserial and $(\soc^{s+1}\hspace{-1.5pt}I+N)/N \not\subseteq \soc^{s-1}\hspace{-.5pt}(I/N),$ for some $s\ge 1$. Write $\soc^{s-1}\hspace{-.5pt}(I/N)=L/N$ and $\soc^s\hspace{-.5pt}(I/N)=M/N$ with $N\subset L \subseteq M\subseteq I$. If $\ntop^{s+1}\hspace{-1pt}(P{\hspace{.5pt}_{\ntop(\soc^s\hspace{-.5pt}(I/N))}})$ is uniserial, then there exists an almost split sequence \vspace{-6.5pt}
$$\xymatrixcolsep{30pt}\xymatrix{
0 \ar[r] & L \ar[r]^-{(q_1, p_1)^T}& M \oplus \soc^{s-1}\hspace{-.8pt}(I/N) \ar[r]^-{(p_2, -q_2)} & \soc^s\hspace{-.5pt}(I/N) \ar[r] & 0
} \vspace{-5pt} $$ in $\mmod A$, where $q_1, q_2$ are inclusion maps and $p_1, p_2$ are canonical projections.

\end{Theo}

\vspace{-1pt}

\noindent{\it Proof.} Considering the submodules $0\subset N \subset L$ of $M$, we obtain from Lemma \ref{ses_mix} a short exact sequence as stated in the theorem, which clearly does not split. Since $\soc^s\hspace{-.6pt}(I/N)$ is uniserial, $\soc^{s-1}\hspace{-1pt}(I/N)$ is the largest proper submodule of $\soc^s\hspace{-.8pt}(I/N)$. Thus, every radical map $f: \soc^s\hspace{-.5pt}(I/N) \to \soc^s\hspace{-.5pt}(I/N)$ factors through $q_2: \soc^{s-1}\hspace{-.5pt}(I/N) \to \soc^s\hspace{-.5pt}(I/N).$ Assume that $\ntop^{s+1}\hspace{-.5pt}(P{\hspace{.5pt}_{\ntop(\soc^s\hspace{-.5pt}(I/N))}})$ is uniserial. \vspace{.5pt} By Proposition 2.2 in \cite[Chapter V]{ARS}, it suffices to show that $D{\rm Tr}\hp(\soc^s\hspace{-.5pt}(I/N))\cong L.$ 

Consider the indecomposable projective module $DI$ in $\mmod A^\circ$ and its proper submodule $N^\perp$. Since $\soc^{s+1}\hspace{-1pt} I\not\subseteq L$ by the hypothesis, 
we see from Lemma \ref{rad_perp} and Proposition \ref{duality_radi_soci}(1) that $\rad^{s-1}(N^\perp)=L^\perp\not\subseteq (\soc^{s+1}\hspace{-1pt} I)^\perp=\rad^{s+1}\hspace{-1pt}(DI).$ Since $N^\perp\cong D(I/N);$ see (\ref{dual_perp}), we have $\ntop^s\hspace{-.5pt}(N^\perp) %\cong \ntop^s\hspace{-.5pt}(D(I/N)) 
\cong D(\soc^s\hspace{-.5pt}(I/N));$ see (\ref{duality_radi_soci}),
which is uniserial by the hypothesis. Since $\soc(\ntop^s\hspace{-.5pt}(N^\perp)) %\cong \soc(D(\soc^s\hspace{-.6pt}(I/N)))
\cong D(\ntop(\soc^s\hspace{-.6pt}(I/N)))$ by Proposition \ref{duality_radi_soci}(3), we have 
$I_{\soc(\ntop^s\hspace{-.5pt}(N^\perp))}\cong D(P_{\ntop(\soc^s\hspace{-.5pt}(I/N))})$. Applying Proposition \ref{duality_radi_soci}(3) again, we obtain \vspace{0.5pt}
$\soc^{s+1}\hspace{-.5pt}(I_{\soc(\ntop^s\hspace{-.5pt}(N^\perp))}) 
%\cong \soc^{s+1}\hspace{-.5pt} (D(P_{\ntop(\soc^s\hspace{-.5pt}(I/N))}))  
\cong D(\ntop^{s+1}\hspace{-.5pt}(P_{\ntop(\soc^s\hspace{-.5pt}(I/N))})),\vspace{0.5pt}$
which is uniserial by the assumption. Now by Theorem \ref{ass_uni_topi} and Lemma \ref{dual_perp}(2), ${\rm Tr}\hp D(\ntop^s\hspace{-.5pt}(N^\perp)) \cong  DI/\rad^{s-1}\hspace{-.6pt}(N^\perp)=DI/L^\perp \cong DL.$ As a consequence, we obtain $L\cong D{\rm Tr}\hp D(\ntop^s\hspace{-.5pt}(N^\perp))  \cong D{\rm Tr}\hp (\soc^s\hspace{-.5pt}(I/N)).$ The proof of the theorem is completed.

\smallskip

As a special case of Theorem \ref{ass_uni_soci}, we obtain the following result. %\vspace{-2.5pt}

\begin{Theo}\label{ass_st_mod_1}

Let $A$ be an artin algebra, and let $I$ be an indecomposable injective module in $\mmod A$ with $I/\soc I=(M/\soc I) \oplus (N/\soc I),$ where $M/\soc I$ is uniserial of length $s>0$ such that $\ntop^{s+1}\hspace{-.5pt}(P_{\,\ntop (M/\soc I)})$ is uniserial. Then we have an almost split sequence \vspace{-9.5pt}
$$\xymatrixcolsep{28pt}\xymatrix{
0 \ar[r] & \rad M + N \ar[r]^-{(q_1, \,p_1)^T\hspace{-5pt}}& I \oplus \rad (M/\soc I) \ar[r]^-{(p_2, -q_2)} & M/\soc I \ar[r] & 0,
}\vspace{-5pt}$$ where $q_1, q_2$ are inclusion maps, and $p_1, p_2$ are epimorphisms, sending $m+n$ to $m+\soc I$ for all $m\in M$ and $n\in N$.

\end{Theo}

\vspace{-1.5pt}

\noindent{\it Proof.} Since $M+N=I$ and $M\cap N=\soc I=\soc M,$ we have an isomorphism $\varphi: I/N\to M/\soc I,$ sending $(m+n)+N$ to $m+\soc I.$ In particular, $I/N$ is uniserial of length $s$. Thus, $\soc^s\hspace{-.5pt}(I/N)=I/N$. And since $\soc^{s-1}\hspace{-.8pt}(M/\soc I)=\rad M/\soc I\hspace{-1pt},$ we see that $\hspace{.5pt}\soc^{s-1}\hspace{-.5pt}(I/N)=(\rad M + N)/N,$ and  $\varphi$ restricts to an isomorphism $\varphi': \soc^{s-1}\hspace{-.8pt}(I/N)\to \rad M/\soc I$. Therefore, $\varphi^{-1}$ induces an isomorphism from the sequence stated in the theorem to the following sequence \vspace{-6.5pt} $$\xymatrixcolsep{25pt}\xymatrix{(*) \hspace{15pt} 0\ar[r] & \rad M +N \ar[r]^-{(q_1\hspace{-.5pt}, \hspace{1pt} p')^T\hspace{-3pt}} & I \oplus \soc^{s-1}\hspace{-.8pt}(I/N) \ar[r]^-{(-p, \hspace{1pt} j)} & \soc^{s}\hspace{-.8pt}(I/N) \ar[r] & 0,}\vspace{-5.5pt}$$ where $j$ is the inclusion map, and $p', p$ are canonical projections. 

Since $M$ is uniserial of length $s+1$, we see that $J^{s-1}\hspace{-1.5pt} M=\soc^2\hspace{-1pt} M.$ By Lemma \ref{socs_sum}, $\soc^{s+1}\hspace{-1pt}I = M +\soc^{s+1}\hspace{-1pt}N$. Thus, $J^{s-1}\hspace{-.5pt}(\soc^{s+1}\hspace{-1pt}I)=\soc^2\hspace{-1pt}M+J^{s-1}\hspace{-.5pt}(\soc^{s+1}\hspace{-1pt}N).$ If $J^{s-1}\hspace{-.5pt}(\soc^{s+1}\hspace{-1pt}I)\subseteq N$, then $\soc^2\hspace{-1pt}M \! \subseteq N,$ and so, $\soc^2\hspace{-1pt}M\subseteq M\cap N=\soc I=\soc M$, absurd. Thus by Proposition \ref{isoc_char}, $(\soc^{s+1}\hspace{-1pt}I+N)/N\not\subseteq \soc^{s-1}(I/N)$. Finally since $\ntop(\soc^{s}\hspace{-.5pt}(I/N))=\ntop(I/N)\cong \ntop(M/\soc I),$ by Theorem \ref{ass_uni_soci}, the sequence $(*)$ is an almost split sequence. The proof of the theorem is completed. 

\vspace{-4pt}

\section{\sc Biserial modules and biserial algebras}

Tachikawa classified finite-dimensional algebras over which every indecomposable module is local or colocal; see \cite{Tac1}. Based on this work, Fuller introduced biserial modules and biserial algebras; see \cite{Ful}. In this section, we slightly modify Fuller's definition of biserial modules to partition them into three disjoint classes. For later applications, we describe all indecomposable modules of length three, and certain indecomposable modules of length four which are biserial or sums of a uniserial module and a biserial module. These ideas may be generalized in the future to study more complex indecomposable modules over biserial algebras. Finally, for biserial algebras, we compute almost split sequences starting with projective modules or ending with injective modules. 

\subsection{\sc Biserial modules} Since biserial modules are sums of uniserial modules, we state the following preparatory result before formally introducing them.

\begin{Lemma}\label{ses_sum_2}

Let $A$ be an artin algebra. Consider a module $M=M_1+M_2$ in $\mmod A,$ where $M_1, M_2$ are proper submodules of $M.$ Then

\begin{enumerate}[$(1)$]

\vspace{-1.2pt}

\item $M/M_1\cap M_2 \cong \left(M_1/M_1\cap M_2\right)\oplus \left(M_2/M_1\cap M_2\right).$ 

\vp

\item If $M_1\cap M_2 \subseteq \rad M_1 \cap \rad M_2,$ then $\ntop M \cong \ntop M_1\oplus \ntop M_2.$

\item If $M_1\cap M_2=\soc M,$ which is simple, then $\ntop M\cong \ntop M_1\oplus \ntop M_2.$ 

\end{enumerate} \end{Lemma}

\vspace{-1.5pt}

\noindent{\it Proof.} (1) Since $M/(M_1+M_2)=0,$ in view of Lemma \ref{ses_sum_itsec}(2), we conclude that $M/M_1\cap M_2\cong M/M_1 \oplus M/M_2 \cong (M_2/M_1\cap M_2)  \oplus (M_1/M_1\cap M_2).$

(2) We have a well-defined epimorphism $f: M \to (M_1/\rad M_1) \oplus (M_2/\rad M_2),$ sending $m_1+m_2$ to $(m_1 + \rad M_1, m_2+\rad M_2),$ where $m_i\in M_i$ for $i=1,2.$ Since $\Ker f=\rad M_1+\rad M_2=\rad M,$ we have $M/\rad M\cong M_1/\rad M_1 \oplus M_2/\rad M_2.$ 

(3) Assume that $M_1\cap M_2=\soc M,$ which is simple.
Since $M_i\subsetneq M$ for $i=1, 2,$ we have $\ell(M_i)>1,$ for $i=1, 2$. By Lemma \ref{soc_rad}, $\soc M=\soc M_i\subseteq \rad M_i,$ for $i=1, 2$. Hence, $M_1\cap M_2\subseteq \rad M_1\cap \rad M_2.$ By Statement (2), $\ntop M\cong \ntop M_1\oplus \ntop M_2.$ 
The proof of the lemma is completed.

\smallskip

As promised, we slightly modify Fuller's definition of biserial modules as follows. 

\begin{Defn}\label{bsm_def}

Let $A$ be an artin algebra. An indecomposable module $M$ in $\mmod A$ is called {\it biserial} if it has two non-comparable uniserial submodules $M_1, M_2$ satisfying the following conditions$\hp :$ 

\begin{enumerate}[$(1)$]

\vspace{-1.5pt}

\item $M_1+M_2$ is $M$ or the largest proper submodule of $M;$

\item $M_1\cap M_2$ is zero or the smallest non-zero submodule of $M.$

\end{enumerate}\end{Defn}

%\smallskip

\begin{Remark}\label{rem_bsm} (1) A module $M$ is biserial if and only if $DM$ is biserial$;$
see \cite{Ful}. 

\noindent (2) A uniserial module is not biserial in our sense 
since its submodules form a chain.

\noindent (3) In contrast, Fuller's definition does not require that the two uniserial submodules be non-comparable; see \cite{Ful}. Technically, uniserial modules are biserial in his sense.

\end{Remark}

It follows easily from the definition that a biserial module is local or colocal$;$
see \cite{Ful}. Conversely, we have the following result.

\begin{Prop}\label{loc_cloc_bsm}

Let $A$ be an artin algebra. Consider a module $M$ in $\mmod A.$

\begin{enumerate}[$(1)$]

\vspace{-1.5pt}

\item If $M$ is local, then it is biserial if and only if $\rad M\!=\!M_1+M_2$, where $M_1, M_2$ are nonzero proper uniserial submodule of $\rad M$ such that $M_1\cap M_2$ is zero or the simple socle of $M;$ in this case, $\ntop(\rad M)\cong \ntop M_1\oplus \ntop M_2$. 

\vspace{.5pt}

\item If $M$ is colocal, then it is biserial if and only if $\soc M=M_1\cap M_2$, where $M_1, M_2$ are nonsimple uniserial such that $M_1+M_2$ is $M$ or the largest proper submodule of $M;$ in this case, $\soc(\hspace{-.8pt}M/\soc M\hspace{-.8pt}) \hspace{-.5pt} \cong \hspace{-.5pt} \soc(\hspace{-.8pt}M_1/\soc M_1\hspace{-.8pt})\oplus \soc(\hspace{-.8pt} M_2/\soc M_2\hspace{-.8pt}).$ 
    
\end{enumerate} \end{Prop}

\vspace{-1.5pt}

\noindent{\it Proof.} (1) Assume that $M$ is local. The sufficiency is evident. Let $M$ be biserial with uniserial submodules $M_1, M_2$ as in Definition \ref{bsm_def}. Then $M_i\subsetneq M_1+M_2;$ and hence, $M_i\subseteq \rad M,$ for $i=1,2$. By Definition \ref{bsm_def}(1), $\rad M=M_1+M_2$. If $M_1\cap M_2=0$, then $\rad M=M_1\oplus M_2$ and $\ntop(\rad M)=\ntop(M_1)\oplus \ntop(M_2).$ 
Otherwise, by Definition \ref{bsm_def}(2) and Lemma \ref{soc_rad},
$M_1\cap M_2=\soc M=\soc (\rad M),$ which is simple. By Lemma \ref{ses_sum_2}(3), $\ntop(\rad M) \hspace{-1pt} = \hspace{-1pt} \ntop(M_1)\oplus \ntop(M_2)$. 

(2) Assume that $M$ is colocal. The sufficiency is evident. Let $M$ be colocal biserial with uniserial submodules $M_1, M_2$ as in Definition \ref{bsm_def}. Then, $\soc M\subsetneq M_i$ for $i=1, 2$. By Definition \ref{bsm_def}(2), $M_1\cap M_2=\soc M=\soc M_1=\soc M_2.$ 

If $M_1+M_2\!=\! M$ then, by Lemma \ref{ses_sum_2}(1), $M/\soc M %= M/M_1\cap M_2 
\!\cong\! (M_1/\soc M_1) \oplus (M_2 /\soc M_2)$; and consequently, $\soc(M/\soc M) \hspace{-1pt} \cong \hspace{-1pt} \soc(\hspace{-1pt} M_1/\soc M_1) \hspace{1pt} \oplus \hspace{1pt} \soc(\hspace{-1pt}M_2/\soc M_2).$ 

Otherwise, $M_1+M_2 = \rad M\supseteq \soc M.$ Then, $M$ is local, and so is $M/\soc M.$ By Lemma \ref{ses_sum_2}(1), $\rad(M/\soc M) %=(\rad M)/M_1\cap M_2 
\cong (M_1/\soc M_1) \oplus (M_2 /\soc M_2);$ and by Lemma \ref{soc_rad}, $\soc(M/\soc M)=\soc(\rad (M/\soc M))\cong  \soc(M_1/\soc M_1)\oplus \soc( M_2/\soc M_2).$ 
The proof of the proposition is completed.

\smallskip

Now, we partition biserial modules into three disjoint classes as follows.

\begin{Defn}\label{as_bsm_def}

Let $A$ be an artin algebra. A biserial module in $\mmod A$ is called 

\begin{enumerate}[$(1)$]

\vspace{-2pt}

\item {\em astride biserial} if it is local and not colocal; 

\item {\em co-astride biserial} if it is colocal and not local; 

\item {\em diamond biserial} if it is local and colocal.
    
\end{enumerate}

\end{Defn}

These terms are motivated by the following two structural descriptions of biserial modules, the first of which %description %of astride biserial and co-astride biserial modules 
follows easily from Proposition \ref{loc_cloc_bsm} and Lemma \ref{ses_sum_2}(1).

\begin{Cor}\label{astride_bsm}

Let $A$ be an artin algebra. 

\begin{enumerate}[$(1)$]

\vspace{-2.5pt}
    
\item A module $M$ in $\mmod A$ is astride biserial if and only if $M$ is local such that $\rad M=M_1\oplus M_2,$ where $M_1, M_2$ are nonzero uniserial$\hp ;$ 

\item A module $M$ in $\mmod A$ is co-astride biserial if and only if $M$ is colocal such that
$M/\soc M=M_1/\soc M \oplus M_2/\soc M,$ where $M_1, M_2$ are nonsimple uniserial.

\end{enumerate} \end{Cor}

%\noindent{\it Proof.} Note that Statement (1) follows immediately from Proposition \ref{loc_cloc_bsm}(1), whereas  Statement (2) follows from Proposition \ref{loc_cloc_bsm}(2) and Lemma \ref{ses_sum_2}(1). The proof of the corollary is completed.

\vspace{2pt}

The second description characterizes the diamond biserial modules as follows. 
\vspace{-1pt}

\begin{Prop}\label{biloc_bis} 

Let $A$ be an artin algebra. Consider a local colocal module $M$ in $\mmod A.$ The following statements are equivalent$\,:$

\begin{enumerate}[$(1)$]

\vspace{-2.5pt}
    
\item $M$ is diamond biserial$\,;$

\item $\rad M$ is co-astride biserial$\,;$

\item $M/\soc M$ is astride biserial$\,;$

\vspace{-.5pt}

\item $\rad M\hspace{-1pt} = \hspace{-1pt}  M_1+M_2$ and $M_1\cap M_2\hspace{-1pt} = \hspace{-1pt}  \soc M,$ where $M_1, M_2$ are nonsimple uniserial. %submodules of $\rad M.$

\end{enumerate}\end{Prop}

\vspace{-2pt}

\noindent{\it Proof.} We may assume that $M$ is not simple. Then, $\soc M=\soc(\rad M)$ by Lemma \ref{soc_rad}. In particular, $\rad M$ is colocal and $M/\soc M$ is local. % such that $\rad M/\soc(\rad M)=\rad M/\soc M=\rad(M/\soc M).$ 
First, by definition, Statement (4) implies Statement (1). 

Suppose that Statement (1) holds. Since $\soc M$ is simple, by Proposition \ref{loc_cloc_bsm}(1), $\rad M=M_1+M_2,$ where $M_1, M_2$ are non-zero proper uniserial submodules of $\rad M$ such that $M_1\cap M_2=\soc M.$ So, Statement (4) holds. Further by Lemma \ref{ses_sum_2}(1), $\rad(M/\soc M)=\rad M/\soc M= (M_1/\soc M)\oplus (M_2/\soc M),$ where $M_1/\soc M$ and $M_2/\soc M$ are nonzero uniserial. By Corollary \ref{astride_bsm}(2), $\rad M$ is co-astride biserial. And since $\rad M/\soc(\rad M)=\rad M/\soc M,$ we see from  Corollary \ref{astride_bsm}(1) that $M/\soc M$ is astride biserial. Thus, Statement (2) and (3) also hold.

Finally, suppose that $\rad M$ is co-astride biserial or $M/\soc M$ is astride biserial. By Corollary \ref{astride_bsm}, 
$\rad M / \soc M=N_1\oplus N_2,$ where $N_1, N_2$ are nonzero uniserial mo\-dules. 
Write $N_i=M_i/\soc M,$ where $\soc M\subsetneq M_i\subsetneq \rad M$, for $i=1, 2$. 
Then,  $\rad M=M_1+M_2$ and $M_1\cap M_2=\soc M,$ where $M_1, M_2$ are non-comparable and uniserial. Hence, $M$ is biserial. The proof of the proposition is completed.

\vspace{-2pt}

\subsection{\sc Indecomposable modules of length three} It is evident that indecomposable modules of length one or two are uniserial. As shown below, indecomposable modules of length three are uniserial or biserial. 

\begin{Prop}\label{CL_3}

Let $A$ be an artin algebra. Then a module $M$ in $\mmod A$ is indecomposable of length three if and only if one of the following cases occurs$\,:$

\begin{enumerate}[$(1)$]

\vspace{-2pt}

\item $M$ is uniserial of length three$\hp\hp;$

\item $M$ is local with $\rad M=\soc M=S_1\oplus S_2$, where $S_1, S_2$ are simple$\hp;$ or equivalently, $M$ is astride biserial of length three$\hp\hp;$

\item $M$ is colocal with $\rad M=\soc M$ and $M/\soc M\cong S_1\oplus S_2,$ where $S_1, S_2$ are simple$\hp;$ or equivalently, $M$ is co-astride biserial of length three.

\end{enumerate} \end{Prop}

\vspace{-2pt}

\noindent{\it Proof.} The sufficiency is evident. Let $M$ be an indecomposable module in $\mmod A$ of length three. Then, $2\le \ell\ell(M) \le 3$ and $1\le \ell(\rad M)\le 2$. If $\ell\ell(M)=3,$ then $M$ is uniserial. Suppose now that $\ell\ell(M)=2$. In view of Lemma \ref{soc_rad}, $\rad M=\soc M.$  

If $\ell(\rad M)=2,$ then $M$ is local and $\rad M=S_1\oplus S_2,$ where $S_1, S_2$ are simple. By Corollary \ref{astride_bsm}(1), this is equivalent to $M$ being astride biserial. If $\ell(\rad M)=1,$ then $M$ is colocal and $M/\soc M=\ntop M=S_1\oplus S_2,$ where $S_1, S_2$ are simple. By Corollary \ref{astride_bsm}(2), this is equivalent to $M$ being co-astride biserial. The proof of the proposition is completed.

\smallskip

The following result says that we can always obtain an astride or co-astride biserial module of length three from a biserial module.

\begin{Lemma}\label{bsm_top2}

Let $A$ be an artin algebra. Consider a biserial module $M$ in $\mmod A$. 

\begin{enumerate}[$(1)$]

\vspace{-1.5pt}

\item If $M$ is local, then $\ntop^2\hspace{-1pt}M$ is astride biserial of length three.

\vspace{.5pt}

\item If $M$ is colocal, then $\soc^2\hspace{-1pt}M$ is co-astride biserial of length three.

\end{enumerate}\end{Lemma}

\vspace{-1.5pt}

\noindent{\it Proof.} We only prove Statement (1). Assume that $M$ is local. In view of Proposition \ref{loc_cloc_bsm}(1), $\ell(\rad(\ntop^2\hspace{-1pt}M))=\ell(\ntop(\rad M))=2$. Thus $\ntop^2\hspace{-1pt}M$ is local non-uniserial of length three. By Proposition \ref{CL_3}(2), $\ntop^2\hspace{-1pt}M$ is astride biserial of length three. The proof of the corollary is completed.

\smallskip

We also need to consider some decomposable modules of length three.

\vp

\begin{Lemma}\label{dec_mod_l3} Let $A$ be an artin algebra. Consider a module $M=S\oplus N$ in $\mmod A,$ where $S$ is simple and $N$ is uniserial of length two. 

\begin{enumerate}[$(1)$]

\vspace{-3pt}

\item If $T$ is a simple submodule of $M$, then $M=T\oplus N$ or $T=\rad N=\rad M.$ 

\item If $L$ is an indecomposable submodule of length two of $M$, then $M=S\oplus L$.

\end{enumerate}\end{Lemma} 

\vspace{-1.5pt}

\noindent{\it Proof.} First of all, $\rad N=\rad M.$ Let $T$ be a simple submodule of $M$. If $T\subseteq N,$ then $T=\rad N.$ Otherwise, $T\cap N=0;$ and consequently, $M=T\oplus N.$ Next, let $L$ be an indecomposable submodule of length two of $M$. Then, $L$ is a uniserial maximal submodule of $M$. If $S\subseteq L$, then $S=\rad L \subseteq \rad M=\rad N$, absurd. Thus $S\cap L=0;$ consequently, $M=S\oplus L.$ The proof of the lemma is completed.

%\vspace{-7pt}

\subsection{\sc Indecomposable modules of length four} It is not difficult to describe all indecomposable modules of length four. However, %To avoid an overly lengthy exposition, 
we focus only on those directly relevant to our main results. First, applying Corollary \ref{astride_bsm} immediately yields the following characterization of astride or co-astride biserial modules of length four.

\begin{Prop}\label{loc_mod_l4}

Let $A$ be an artin algebra. If $M$ is a module in $\mmod A,$ then

\begin{enumerate}[$(1)$]

\vspace{-2pt}

\item $M$ is astride biserial of length four if and only if $\rad M=S\oplus N,$ \vp where $S$ is simple and $N$ is uniserial of length two$\,;$

\vp

\item $M$ is co-astride biserial of length four if and only if $M/\soc M=S\oplus N,$ where $S$ is simple and $N$ is uniserial of length two.

\end{enumerate}\end{Prop}

%\noindent{\it Proof.} We shall only prove Statement (1). Suppose $M$ is astride biserial. Since $M$ is local, $\ell(\rad M)=3$, and by Corollary \ref{astride_bsm}(1), $\rad M=S\oplus N,$ where $S$ is simple and $N$ is uniserial of length one two. The proof of the proposition is completed.

\smallskip

Next, we describe diamond biserial modules of length four. \vspace{-1pt}

\begin{Prop}\label{lozenge_char}

Let $A$ be an artin algebra. Then a module $M$ in $\mmod A$ is diamond biserial of length four if and only if it satisfies
the following conditions$\,:$

\begin{enumerate}[$(1)$]

\vspace{-3pt}
    
\item $\ntop M$ is simple$\hspace{.5pt};$

\item $\rad^2\hspace{-1pt}M=\soc M,$ which is simple$\hspace{.5pt};$

\item $\ntop(\rad M)=S_1\oplus S_2,$ where $S_1, S_2$ are simple. 

\end{enumerate}\end{Prop}

 \vspace{-1.5pt}
 
\noindent{\it Proof.} We may assume that $M$ is local and colocal. Suppose first that $M$ is biserial of length four. By Proposition \ref{loc_cloc_bsm}(1), $\ntop(\rad M)=S_1\oplus S_2,$ where $S_1, S_2$ are simple. And by Proposition \ref{biloc_bis}, $\rad M$ is co-astride biserial of length three. Thus, by Proposition \ref{CL_3}(3) and Lemma \ref{soc_rad}, $\rad^2\hspace{-1pt} M=\soc(\rad M)=\soc M.$ 

Suppose now that $\rad^2\hspace{-1pt}M=\soc M$ and $\ntop(\rad M)=S_1\oplus S_2,$ where $S_1, S_2$ are simple. Then 
$\rad M/\soc(\rad M)=\rad M/\soc M=\ntop(\rad M).$ In view of Corollary \ref{astride_bsm}(2), $\rad M$ is co-astride biserial of length three. In view of  Proposition \ref{biloc_bis}, $M$ is diamond biserial of length four. The proof of the proposition is completed.

\smallskip

In what follows, for the sake of brevity, diamond biserial modules of length four will be called {\it lozenge modules}.

\begin{Cor}\label{rb4_lozenge} 

Let $A$ be an artin algebra, and let $M$ be a local and colocal module in $\mmod A.$ If $\ell(M)\le 4,$ then $M$ is a uniserial or lozenge module. 

\end{Cor} 

%\vspace{-1.5pt}

\noindent{\it Proof.} If $\ell(M)\le 3$, then $M$ is uniserial; see (\ref{CL_3}). Assume that $\ell(M)=4$. Then, $\rad M$ is colocal of length three. By Proposition \ref{CL_3},  $\rad M$ is uniserial or co-astride biserial of length three. Thus $M$ is uniserial or a lozenge module; see (\ref{biloc_bis}). The proof of the corollary is completed. 

\smallskip

Finally, we study a class of indecomposable modules of length four, which are neither local nor colocal.

\begin{Defn}\label{N_mod_def}

Let $A$ be an artin algebra. An indecomposable module $M$ in $\mmod A$ is called {\it N-shaped} if $M=M_1+M_2,$ where $M_1$ is astride biserial of length three and $M_2$ is uniserial of length two.

\end{Defn}

\noindent{\sc Example.} Let $A=kQ$, where $k$ is a field and $Q$ is the quiver $\hspace{-3pt}\xymatrixcolsep{22pt}\xymatrix{1 \ar@<0.4ex>[r] \ar@<-0.4ex>[r] & 2.}$ We write $\alpha, \beta$ for the two arrows and $e_1, e_2$ for the two trivial paths. Consider the left $A$-module $M$ with $e_1M=k\langle u_1, u_2\rangle$ and $e_2M=k\langle v_1, v_2\rangle,$ whose $A$-multiplication is given by $(\alpha u_1, \alpha u_2)=(v_1, 0)$ and $(\beta u_1, \beta u_2)=(0, v_2)$. Then $M=M_1+M_2$, where $M_1=A(u_1+u_2),$ which is astride biserial of length three;
and $M_2=Au_2$, which is uniserial of length two. However, $M$ is not N-shaped since it is decomposable.

\smallskip

The following two statements collect some basic properties of N-shaped modules.

\begin{Lemma}\label{N_mod_property_0}

Let $A$ be an artin algebra, and let $M=M_1+M_2$ be an N-shaped module in $\mmod A,$ where $M_1$ is astride biserial of length three and $M_2$ is uniserial of length two. Then 
%the following statements hold. 

\begin{enumerate}[$(1)$]

\vspace{-2.5pt}

\item $M_1\cap M_2=\soc M_2$ and $\rad M=\soc M=\soc M_1=\rad M_1\hspace{.5pt};$

\item $M$ is of length four with $\ntop M %= M/\soc M 
\cong \ntop M_1 \oplus \ntop M_2.$ 

\end{enumerate} \end{Lemma}

\vspace{-1.5pt}

\noindent{\it Proof.} By Proposition \ref{CL_3}(2), $\rad M_1=\soc M_1$. Note that $M_1\cap M_2\ne 0.$  \vspace{-1pt} If $M_1\cap M_2=M_2,$ then $M_2\subseteq \rad M_1=\soc M_1$, absurd. Thus, $M_1\cap M_2=\soc M_2$. Since $\rad M_2=\soc M_2,$ applying Lemma \ref{soc_rad} yields
\vspace{-4.5pt} $$\rad M= \rad M_1+\rad M_2= \soc M_1+\soc M_2= \soc M_1   \subseteq \soc M\subseteq \rad M. \vspace{-4.5pt}$$ 
Thus $\rad M=\soc M=\soc M_1.$ Now $\ell(M)=\ell(M_2)+\ell(M_1)-\ell(M_1\cap M_2)=4$ by Lemma \ref{ses_sum_2}(1). Finally since $M_1\cap M_2=\soc M_2\subseteq \rad M_1 \cap \hspace{.5pt}  \rad M_2,$ by Lemma \ref{ses_sum_2}(2), $\ntop M \hspace{-.5pt} \cong \hspace{-.5pt}  \ntop M_1 \oplus \ntop M_2.$ The proof of the lemma is completed. 

\vspace{2pt}

%We shall show that the notion of N-shaped modules is self-dual.

\begin{Lemma}\label{N_mod_property}

Let $A$ be an artin algebra, and let $M=M_1+M_2$ be an N-shaped module in $\mmod A,$ where $M_1$ is astride biserial of length three and $M_2$ is uniserial of length two. Then 

\begin{enumerate}[$(1)$]

\vspace{-2pt}

\item $\soc M = S \oplus \soc M_2,$ where $S$ is simple$\hp;$

\item $M/M_2$ is uniserial of length two, $\rad(M/M_2) \cong S$ and $\ntop(M/M_2) \cong \ntop M_1\hp;$

\item $M/S$ is co-astride biserial of length three such that $\soc(M/S)  \cong   \soc M_2$ and $(M/S)/\soc(M/S)  \cong \ntop M\cong \ntop M_1\oplus \ntop M_2.$ 

\end{enumerate} \end{Lemma}

\vspace{-1.5pt}

\noindent{\it Proof.} (1) By Lemma \ref{N_mod_property_0}, $M_1\cap M_2=\soc M_2\subseteq \soc M=\soc M_1.$ By Proposition \ref{CL_3}(2), $\ell(\soc M)=2;$ and hence, $\soc M=S  \oplus  \soc M_2,$ where $S$ is simple. 

(2) Since $M/M_2\cong M_1/\soc M_2,$ we have $\ntop(M/M_2) \cong \ntop M_1;$ consequently, $M/M_2$ is uniserial of length two. Moreover since $(S+\soc M_2)/\soc M_2\cong S$, we infer that $\soc(M/M_2)\cong \soc(M_1/\soc M_2) \cong S.$

(3) Since $\rad M=\soc M,$ we have $\rad (M/S)=(\soc M)/S\cong \soc M_2,$ and by Lemma \ref{N_mod_property_0}(2), $\ntop(M/S)\cong \ntop M\cong \ntop M_1\oplus \ntop M_2$.  
In particular, $M/S$ is neither local nor semisimple. 

Suppose that $M/S$ is decomposable. Then, $M/S=T\oplus L$, where $T$ is simple and $L$ is uniserial of length two. Note that $S\cap M_2=S\cap \soc M_2=0$. Thus, $M/S$ has a submodule $(M_2+S)/S \cong M_2,$ which is uniserial of length two. By Lemma \ref{dec_mod_l3}(1), $M/S=T\oplus (M_2+S)/S$. Write $T=N/S$ with $S\subsetneq N.$ Then $M=N+(M_2+S)=N+M_2$ and $N\cap (M_2+S)=S.$ As a consequence, we have $N\cap M_2=N\cap (M_2+S)\cap M_2=S\cap M_2=0,$ and hence, $M=N\oplus M_2,$ absurd. 

Hence, $M/S$ is indecomposable. Since $\ell\ell(M/S)\le \ell\ell(M)=2,$ we infer that $\soc(M/S)=\rad(M/S)\cong \soc M_2$. So, $M/S$ is colocal of length three; consequently, $M/S$ is co-astride biserial such that $(M/S)/\soc(M/S)=\ntop(M/S)\cong \ntop M;$ see (\ref{CL_3}). The proof of the lemma is completed. 

\smallskip

We conclude this subsection by characterizing N-shaped modules in terms of short exact sequences as follows.
%, which implies, in particular, that this notion is self-dual.

\begin{Prop}\label{N_mod_ses}

Let $A$ be an artin algebra. An indecomposable module $M$ in $\mmod A$ is N-shaped if and only if one of the following short exact sequences exists$\hp\hp :$

\begin{enumerate}[$(1)$]

\vspace{-3pt}

\item $\xymatrixcolsep{22pt}\xymatrix{0\ar[r] & S \ar[r]^-{\hspace{6pt} (f_3, f_4)^{\hp T}} & M_1\oplus M_2\ar[r]^-{(f_1, f_2)} & M \ar[r] \ar[r] & 0,}$ \hspace{-6pt} \vspace{-.5pt} where $S$ is simple, $M_1$ is astride biserial of length three and $M_2$ is uniserial of length two. 

%\vp

\item $\xymatrixcolsep{22pt}\xymatrix{0\ar[r] & M \ar[r]^-{\hspace{6pt} (g_1, g_2)^{\hp T}} & N_1\oplus N_2\ar[r]^-{(g_3, g_4)} & T \ar[r] \ar[r] & 0,}$ \hspace{-6pt} \vspace{-0pt}  where $T$ is simple, $N_1$ is co-astride biserial of length three and $N_2$ is uniserial of length two. 

\end{enumerate}

\end{Prop}

\vspace{-1.5pt}

\noindent{\it Proof.} Let $M$ be an indecomposable module in $\mmod A.$ Suppose first that the short exact sequence (1) stated in the proposition exists. Since $S$ is simple, $f_3, f_4$ are monomorphisms, and so are $f_1, f_2$. Thus $M=N_1+N_2,$ where $N_i={\rm Im} f_i\cong M_i$ for $i=1, 2.$ By definition, $M$ is an N-shaped module. 

Suppose now that $M=M_1+M_2,$ where $M_1$ is astride biserial of length three and $M_2$ is uniserial of length two. By Lemma \ref{N_mod_property_0}, $M_1\cap M_2=\soc M_2.$ By Lemma \ref{ses_sum_itsec}(1), we obtain the short exact sequence (1) stated in the proposition. Further write $\soc M_1=S_1\oplus \soc M_2,$ where $S_1$ is simple. By Lemma \ref{N_mod_property}, $M/S_1$ is co-astride biserial of length three and $M/M_2$ is uniserial of length two. Note that  $S_1\cap M_2=S_1\cap \soc M_2=0$ and $M_1\cap (S_1+M_2)=S_1+M_1\cap M_2=\soc M_1=\rad M_1.$ Thus, $M/(S+M_2)=(M_1+(S_1+M_2))/(S+M_2)\cong \ntop M_1.$ In view of Lemma \ref{ses_sum_itsec}(2), we obtain a short exact sequence \vspace{-7pt}
$$\xymatrixcolsep{25pt}\xymatrix{0\ar[r] & M \ar[r]^-{\hspace{6pt} (p_1, p_2)^{\hp T}} & M/S_1\oplus M/M_2\ar[r]^-{(-p_3, p_4)} & \ntop M_1\ar[r] \ar[r] & 0,}\vspace{-6pt}$$ where the $p_i$ are epimorphisms, the short exact sequence (2) in the proposition.

Finally, note that a module $L$ in $\mmod A$ is astride or co-astride biserial if and only if $DL$ is co-astride or astride biserial, respectively. Suppose that $\eta$ is a short exact sequence of form (2). Then, $D\eta$ is a short exact sequence of form (1), ending with $DM$. As shown at the beginning of the proof, $DM$ is N-shaped. As seen above, there exists also an short exact sequence $\delta$ of form (2), starting with $DM$. Then, $D\delta$ is a short exact sequence of form (1), ending with $M$. As has been shown, $M$ is N-shaped. The proof of the proposition is completed. 

\begin{Remark}\label{N_mod_dual} As seen in the proof of Proposition \ref{N_mod_ses}, a module $M$ in $\mmod A$ is N-shaped if and only if $DM$ is N-shaped. \end{Remark}

\subsection{\sc Biserial algebras} We begin by recalling the definition of biserial algebras from \cite{Ful} and reformulating that of string algebras from \cite{LiY}. 

\begin{Defn}\label{bsa_def} An artin algebra $A$ is called a {\it biserial} (respectively, {\it string}) {\it algebra} if every indecomposable non-uniserial projective module in $\mmod A$ or $\mmod A^{\rm o}$ is biserial (respectively, astride biserial).
    
\end{Defn}

\begin{Remark} 

Since the notion of biserial modules is self-dual, an artin algebra $A$ is biserial (respectively, string) if and only if every indecomposable non-uniserial injective module in $\mmod A$ or $\mmod A^\circ$ is biserial (respectively, co-astride biserial).
    
\end{Remark}

%\begin{Lemma}%\label{bs_rtb} Let $A$ be a biserial algebra. Consider a module $M$ in $\mmod A$. 
%\begin{enumerate}[$(1)$]\item If $M$ is local, then $\ell(\ntop(\rad M))\le 2,$ \vspace{-.5pt} where the equality occurs only if $P_{\hspace{.5pt}\ntop M}$ is biserial. \item If $M$ is colocal, then $\ell(\soc(M/\soc M))\le 2,$ \vspace{-1pt} where the equality occurs only if $I_{\hspace{.5pt}\soc M}$ is biserial.\end{enumerate}\end{Lemma}

%\noindent{\it Proof.} We only prove Statement (1). We may assume that $M$ is local non-simple. Then $M$ admits a top-basis $\{m_0\}$ and $\rad M$ admits a top-basis $\{m_1, \ldots, m_r\},$ where $m_i\in e_iM$ with $e_i$ a primitive idempotent in $A,$ for $i=0, 1, \ldots, r$. Write $m_i=u_im_0$, where $u_i\in e_iJe_0,$ for $i=1, \ldots, r.$ We have a projective cover $f: Ae_0\to Am_0$, sending $e_0$ to $m_0$. Note that $f$ restricts to a map $g: Je_0\to \rad M,$ sending $u_i$ to $m_i$ for $i=1, \ldots, r$. By Lemma \ref{map_tfs}(1), $\{u_1, \ldots, u_r\}$ is top-free. Since $Ae_0$ is uniserial or local biserial, by Proposition \ref{loc_cloc_bsm}(1), $\ell(\ntop(Je_0))\le 2$, and hence, $r\le 2$. If $r=2$, then $Je_0$ is not uniserial, and consequently, $Ae_0$ is biserial. The proof of the lemma is completed.

%\smallskip

As shown below, astride or co-astride biserial modules of length three over a biserial algebra are uniquely determined by their top or socle, respectively.

\begin{Prop}\label{bsa_ml3}

Let $A$ be a biserial algebra. Consider a module $M$ in $\mmod A$ with projective cover $P$ and injective envelope $I$. Then

\begin{enumerate}[$(1)$]

\vspace{-1.5pt}

\item $M$ is astride biserial of length three if and only if $P$ is biserial and $M \!\cong \! \ntop^2\hspace{-1pt}P\hspace{-.5pt};$

\vspace{1pt}

\item $M$ is coastride biserial of length three if and only if $I$ is biserial and $M \!\cong \!\soc^2\hspace{-1pt}I.$

\end{enumerate}\end{Prop}

\vspace{-1.5pt}

\noindent{\it Proof.} We only prove Statement (1). The sufficiency follows from Lemma \ref{bsm_top2}(1). Let $M$ be astride biserial of length three with a projective cover $f: P\to M.$ Since $\rad^2\hspace{-1pt}M=0;$ see (\ref{CL_3}), $f$ induces an epimorphism $\bar f: P / \rad^2\hspace{-1pt}P \to M.$ Since $M$ is not uniserial, $P$ biserial. By Lemma \ref{bsm_top2}(1), 
$\ell(P/\rad^2\hspace{-1pt}P)=3;$ consequently, $\bar f$ is an isomorphism. 
%(2) Suppose that $M$ is colocal. Then $DM$ is local with $\ntop(DM)\cong D(\soc M).$ Thus, $P_{\ntop(DM)}\cong D(I_{\soc M}),$ and hence, \vspace{.5pt} $\ntop^2(P_{\ntop(DM)})\cong D(\soc^2(I_{\soc M}))$ by Proposition \ref{duality_radi_soci}(2). Now, $M$ is biserial of length $3$ if and only if $DM$ is biserial of length $3$ if and only if $P_{\ntop(DM)}$ is biserial and $DM\cong \ntop^2\hspace{-.5pt}(P_{\ntop(DM)})$, that is, $I_{\hspace{.5pt}\soc M}$ is biserial and $M\cong \soc^2\hspace{-.5pt}(I_{\hspace{.5pt}\soc M}).$ 
The proof of the proposition is completed.  

\smallskip

Now, we turn to compute almost split sequences starting with biserial projective mo\-dules over biserial algebras.

\begin{Prop}\label{ass_p_astride}

Let $A$ be a biserial algebra, and let $P$ be a biserial projective module in $\mmod A$ with $\rad P=M_1\oplus M_2,$ where $M_i$ is uniserial of length $s_i$ such that $\soc^{s_i+1}(I_{\hspace{.5pt}\soc M_i})$ is uniserial, for $i=1, 2$. Then we have an almost split sequence \vspace{-6pt}
$$\xymatrixcolsep{30pt}\xymatrix{0\ar[r] & P \ar[r]^-{(p_1, p_2)^{\hp T}\hspace{-2pt}}& P /\soc M_1 \oplus P/\soc M_2  \ar[r]^-{(-p_3, p_4)} &  P/\soc P \ar[r] & 0,}\vspace{-4.5pt}$$ where the $p_i$ are canonical projections.

\end{Prop}

\vspace{-1.5pt}

\noindent{\it Proof.} Since $M_1\cap M_2=0$, by Lemma \ref{ses_sum_itsec}(2), we obtain the short exact sequence stated in the proposition. 
%It amounts to show that $(p_1, p_2)^T$ is minimal left almost split. 
By Theorem \ref{ass_st_mod}, we have almost split sequences \vspace{-6pt}
$$\xymatrixcolsep{26pt}\xymatrix{0\ar[r] & M_i\ar[r]^-{(f_i, q_i)^{\hp T}} & (M_i/\soc M_i) \oplus P  \ar[r]^-{(j_i, p_i)} & P/\soc M_i \ar[r] & 0}\vspace{-5pt}$$ for $i=1, 2,$ where $q_i, j_i$ are inclusion maps, and $p_i, f_i$ are canonical projections. Then $(p_1,p_2)^T: P\to P/\soc M_1 \oplus P/\soc M_2$ is irreducible; see \cite[(1.3)]{LiY}. Not being colocal, $P$ is not a direct summand of the radical of any indecomposable projective module in $\mmod A;$ see (\ref{astride_bsm}) and (\ref{biloc_bis}). \vp Since $(q_1, q_2): M_1\oplus M_2\to P$ is minimal right almost split, we conclude that $(p_1, \,p_2)^T$ is minimal left almost split. The proof of the proposition is completed.

\smallskip

Dually,  we compute almost split sequences ending with biserial injective modules.

\begin{Prop}\label{ass_endi}

Let $A$ be a biserial algebra, and let $I$ be a biserial injective module in $\mmod A$ with $I/\soc I=(M_1/\soc I) \oplus (M_2/\soc I),$ where $M_i/\soc I$ is uniserial of length $s_i>0$ such that $\ntop^{s_i+1}\hspace{-.5pt}(P_{\hp\hp \ntop(M_i/\soc I)})$ is uniserial, for $i=1, 2$. Then we have an almost split sequence \vspace{-5.5pt}
$$\xymatrixcolsep{30pt}\xymatrix{0\ar[r] & \rad I \ar[r]^-{(-q_3, q_4)^T\hspace{-4pt}}& (\rad M_1 + M_2) \oplus (M_1+\rad M_2) \ar[r]^-{(q_1, q_2)} &  I \ar[r] & 0,} \vspace{-3.5pt} $$ where the $q_i$ are inclusion maps.

\end{Prop}

\noindent{\it Proof.} First, $M_1+M_2=I$ and $M_1\cap M_2=\soc I=\soc M_i\subseteq \rad M_i,$ for $i=1, 2$. So by the modular law, 
$(\rad M_1 + M_2) \cap (M_1+\rad M_2)=\rad M_1+\rad M_2\!=\!\rad I.$ Then by Lemma \ref{ses_sum_itsec}(1), we have the short exact sequence stated in the proposition. It suffices to show that $(q_1, q_2)$ is minimal right almost split.\vspace{1pt} Indeed, there exists a minimal left almost split morphism $(p_1, p_2)^T: I\to \soc M_1/\soc I \oplus M_2/\soc I,$ where $p_1, p_2$ are epimorphisms, sending $m_1+m_2$ to $m_1+\soc I$ and $m_2+\soc I,$ respectively, for all $m_1\in M_1$ and $m_2\in M_2$. Write $L_1=\rad M_1+M_2$ and $L_2=M_1+\rad M_2$. By Theorem \ref{ass_st_mod_1}, we have almost split sequences \vspace{-4.5pt}
$$\xymatrix{0\ar[r] & L_i \ar[r]^-{(f_i, q_i)^T} & \rad M_i / \soc M_i  \oplus I \ar[r]^-{(-j_i, p_i)} &  M_i / \soc I\ar[r] & 0,} \vspace{-4pt} $$  for $i=1, 2,$ where $q_i, j_i$ are inclusion maps and $f_i$ is an epimorphism. \vp Thus, $(q_1, q_2): (\rad M_1 + M_2) \oplus (M_1+\rad M_2)\to I$ is irreducible; see \cite[(1.3)]{LiY}. Not being local, $I$ is not a direct summand of the socle-factor of any indecomposable injective module in $\mmod A;$ see (\ref{astride_bsm}) and (\ref{biloc_bis}). Since $(p_1, p_2)^T$ is minimal left almost split, $(q_1, q_2)$ is a minimal right almost split morphism. The proof of the proposition is completed.

\section{\sc Quadri-biserial algebras} 

In this section, we introduce quadri-biserial algebras and thoroughly study their representation theory, preparing for our main classification theorem later. First, we reduce the representation theory of quadri-biserial algebras to that of quadri-string algebras, and then apply our previous results and techniques to compute almost split sequences over quadri-string algebras. This enables us to show that the indecomposable modules over a quadri-biserial algebra are uniserial or biserial module of length at most four, or lozenge modules, which are of length four by definition. Combining this with our previous results, we obtain an explicit description of all indecomposable modules and almost split sequences for quadri-biserial algebras.

\subsection{\sc Definition} We begin by introducing the definition of quadri-biserial algebras, which form the central object of study for this section. 

\begin{Defn}\label{4str_alg_def}

A biserial (respectively, string) artin algebra $A$ is called  {\it quadri-biserial} (respectively, {\it quadri-string}) \vp if the indecomposable projective modules $P$ in $\mmod A$ and $\mmod A^\circ$ have the following properties:

\begin{enumerate}[$(1)$]

\vspace{-1pt}

\item The length of $P$ is at most four;

\vp
    
\item If $\ntop(\rad P)=S_1\oplus S_2,$ where $S_1, S_2$ are simple, then $I_{S_1}$ or $I_{S_2}$ is uniserial;

\vp

\item If $\rad P=S\oplus M,$ \vspace{-1pt} where $S$ is simple and $M$ is uniserial of length $2$, then $I_S$ and $I_{\hspace{.8pt}\soc M}$ are uniserial.

\end{enumerate} \end{Defn}

\begin{Remark} \label{4sa_rem} 

In view of Proposition \ref{duality_radi_soci}, an artin algebra $A$ is quadri-biserial if and only if the indecomposable injective modules in $\mmod A$ and $\mmod A^\circ$ have the dual properties of those stated in Definition \ref{4str_alg_def}.

\end{Remark}

\smallskip

We collect some basic properties of indecomposable projective modules over quadri-biserial algebras.

\smallskip

\begin{Lemma}\label{4sa_bproj}

Let $A$ be a quadri-biserial algebra. Consider an astride biserial projective module $Ae$ in $\mmod A$ and a top-basis $\{u_1, u_2\}$ for $Je$ with $u_i\in e_iJe$, where $e, e_1, e_2$ are primitive idempotents in $A$. If $e_2A$ is biserial, then $Au_1$ is simple and $Je=Au_1\oplus Au_2.$ 

\end{Lemma}

\vspace{-1.5pt}

\noindent{\it Proof.} By Lemmas \ref{tb_gen}(2) and \ref{find_top-basis}(1), $Je=Au_1+Au_2,$ where $Au_1, Au_2$ are proper submodules of $Je$, and $\ntop(Je)\cong Ae_1/Je_1\oplus Ae_2/Je_2.$ Assume that $e_2A$ is biserial, but $Au_1$ was not simple. Since $Ae$ is astride biserial of length at most four, $Je=S\oplus N$, where $S$ is simple and $N$ is uniserial of length two. Now since $\ell(Je)=3,$ we have $\ell(Au_1)=2.$ By Lemma \ref{dec_mod_l3}(2), $Je=S\oplus Au_1,$ and hence, $\ntop(Je)\cong S\oplus \ntop(Au_1) = S\oplus Ae_1/Je_1.$ This implies that $S\cong Ae_2/Je_2$. So $I_S\cong D(e_2A)$, which is biserial by the assumption, a contradiction to Definition \ref{4str_alg_def}(3). Thus, $Au_1$ is simple; and consequently, $Je=Au_1\oplus Au_2.$ The proof of the lemma is completed.

\smallskip

As shown below, quadri-biserial algebras with radical cubed zero have the essential defining properties of special biserial algebras; see \cite[Section 1]{SKW}.

\begin{Lemma}\label{sbsa} 

Let $A$ be a quadri-biserial algebra with radical cubed zero. Consider primitive idempotents $e, e_0, e_1, e_2$ in $A$ with $Ae_1\not\cong Ae_2,$ or equivalently, 
$e_1A\not\cong e_2A.$ 
 
\begin{enumerate}[$(1)$]

\vspace{-2pt}

\item If $u\in e_0Je$ and $u_i \in e_iJe_0$ for $i=1, 2$, then $u_1u=0$ or $u_2u=0.$

\item If $v\in e_0Je$ and $v_i\in eJe_i$ for $i=1, 2$, then $vv_1=0$ or $vv_2=0.$
    
\end{enumerate}\end{Lemma}

\vspace{-1.5pt}

\noindent{\it Proof.} We only prove Statement (1). Let $u\in e_0Je$ and $u_i \in e_iJe_0,$ for $i=1, 2,$ such that $u_1u\ne 0$. By Lemma \ref{loc_gen}(1), $Au$ is local. We claim that $Au$ is uniserial. Indeed, we may assume that $Ae$ is biserial. Since $Je$ is not local by Proposition \ref{loc_cloc_bsm}(1), $Au \subsetneq Je \subsetneq Ae.$ Since $\ell(Ae)\le 4$, we have $\ell(Au)\le 2.$ So, $Au$ is uniserial. This proves our claim. Then $Ju$ is uniserial. Since $J^3=0$, we have $u_1u\in Ju\backslash J^2u$. Thus $Ju=Au_1u.$ In particular, $u_2u=a(u_1u)$ with $a\in e_2Ae_1$. Since $Ae_1\not\cong Ae_2,$ we have $e_2Ae_1 \subseteq e_2Je_1.$ Thus, $u_2u\in J^3=0.$ The proof of the lemma is completed.

\subsection{\sc Indecomposable modules of length four} In this subsection, we classify the indecomposable modules of length four over quadri-biserial algebras. We begin with some properties of the composition factors of astride and co-astride biserial modules of length three.

\begin{Lemma} \label{4bs_loc_l3}

Let $A$ be a quadri-biserial algebra. Consider a module $M$ in $\mmod A.$

\begin{enumerate}[$(1)$]

\vspace{-2pt}

\item  If $M$ is astride biserial of length three with $\rad M=S_1\oplus S_2,$ \vspace{-1pt} where $S_1, S_2$ are simple, then $P_{\hspace{.5pt} \ntop M}$ is biserial, whereas $I_{S_1}$ or $I_{S_2}$ is uniserial.

\vspace{.5pt}

\item If $M$ is co-astride biserial of length three with $M/\soc M=S_1\oplus S_2,$ \vspace{-1pt} where $S_1, S_2$ are simple, then $I_{\soc M}$ is biserial, whereas $P_{S_1}$ or $P_{S_2}$ is uniserial.
    
\end{enumerate}\end{Lemma}

\noindent{\it Proof.} We only prove Statement (1). Let $M$ be astride biserial with $\rad M=S_1\oplus S_2,$ where $S_1, S_2$ are simple. Then it follows from Lemma \ref{bsa_ml3}(1) that $P_{\hspace{1pt}\ntop M}$ is biserial and $M\cong \ntop^2 \hspace{-.2pt}( \hspace{-.5pt}P_{\hspace{1pt}\ntop M}\hspace{-1pt}).$ Hence, $\ntop(\rad(P_{\hspace{1pt}\ntop M})) =\rad(\ntop^2(P_{\hspace{1pt}\ntop M}))\cong S_1\oplus S_2.$ By Definition \ref{4str_alg_def}(2), $I_{S_1}$ or $I_{S_2}$ is uniserial. The proof of the lemma is completed.

\smallskip

We also need some properties of the composition factors of N-shaped modules.

\begin{Lemma}\label{N_mod_property-2}
Let $A$ be a quadri-biserial algebra. Let $M$ be an N-shaped module in $\mmod A$ with $M=M_1+M_2$, where $M_1$ is astride biserial of length three and $M_2$ is uniserial of length two such that $\soc M_1=S\oplus \soc M_2.$ Then

\begin{enumerate}[$(1)$]

\vspace{-2pt}

\item $P_{\hspace{1pt}\ntop M_1}$ is biserial and $P_{\hspace{1pt}\ntop M_2}$ is uniserial$\hspace{.8pt};$

\vspace{.5pt}

\item $I_S$ is uniserial, and $I_{\hspace{.5pt}\soc M_2}$ is biserial 
such that $\soc^2(I_{\hspace{.5pt}\soc M_2}) \cong M/S$.
 
 \end{enumerate}\end{Lemma}

\vspace{-1.5pt}

\noindent{\it Proof.} Since $\rad M_1=\soc M_1$, we see from Lemma \ref{4bs_loc_l3}(1) that $P_{\hspace{1pt}\ntop M_1}$ is biserial, 
%and $M_1\cong \ntop^2 \hspace{-.2pt}(\hspace{-.5pt}P_{\hspace{1pt}\ntop M_1}\hspace{-1pt}).$ Then, $\ntop(\rad(P_{\hspace{1pt}\ntop M_1})) =\rad(\ntop^2(P_{\hspace{1pt}\ntop M_1}))\cong \rad M_1=\soc M_1=S\oplus \soc M_2.$ By Definition \ref{4str_alg_def}(2), 
and either $I_S$ or $I_{\soc M_2}$ is uniserial. On the other hand, by Proposition \ref{N_mod_property}(3), $M/S$ is co-astride biserial of length three with $\soc(M/S)  \cong  \soc M_2$ and $(M/S)/\soc(M/S)  \cong \ntop M_1\oplus \ntop M_2.$ And by Lemma \ref{4bs_loc_l3}(2), $I_{\soc M_2}$ is biserial, and either $P_{\ntop M_1}$ or $P_{\ntop M_2}$ is uniserial. Thus, we conclude that $I_S$ and $P_{\ntop M_2}$ are uniserial. The proof of the lemma is completed.

\smallskip 

We are ready to state the main result of this subsection. 

\begin{Theo}\label{l4-4bsa}

Let $A$ be a quadri-biserial algebra. Consider an indecomposable module $M$ of length four in $\mmod A$. Then exactly one of the following cases occurs$\hp:$

\begin{enumerate}[$(1)$]

\vspace{-1pt}

\item $M$ is an N-shaped module$\hp;$

%\vp

\item $M$ is an astride biberial projective module$\hp;$ % with $\rad M=S\oplus N,$ where $S$ is simple and $N$ is uniserial of length two$\hp;$ 

%\vp

\item $M$ is a co-astride biserial injective module$\hp;$
%with $M/\soc M=S\oplus N,$ where $S$ is simple and $N$ is uniserial of length two. 

%\vp

\item $M$ is a projective-injective module, which is uniserial or lozenge.

\end{enumerate} \end{Theo}

\vspace{-2pt}

\noindent{\it Proof.} Suppose that $M$ is local and not colocal. In view of Lemma \ref{maxi_local}(1) and Definition \ref{4str_alg_def}(1), $M$ is a biserial projective module. By Definition \ref{as_bsm_def}, $M$ is astride biserial. Similarly, if $M$ is colocal and not local, then $M$ is a co-satride injective module. In case $M$ is local and colocal, by Lemma \ref{maxi_local}, 
$M$ is projective-injective; and by Corollary \ref{rb4_lozenge}, $M$ is a uniserial or lozenge module. 

Finally, consider the case where $M$ is neither local nor colocal. By Lemma \ref{soc_rad}, we have $4=\ell(\ntop M)+\ell(\rad M)\ge 2 + \ell(\soc M)\ge 4.$
Thus, $\ell(\ntop M)=\ell(\soc M)=2$ and $\soc M=\rad M$. Thus, $\ell\ell(M)=2$. Write $\ntop M= S_1\oplus S_2$ and $\rad M=\soc M= S_3 \oplus S_4,$ where $S_i\cong Ae_i/Je_i$ with $e_i$ a primitive idempotent in $A$, for $i=1, 2, 3, 4.$ Choose a top-basis $\{m_1, m_2\}$ with $m_i\in e_iM$ for $M.$ By Lemma \ref{tb_gen}(2), $M=M_1+M_2,$ where $M_1=Am_1$ and $M_2=Am_2$. %are not comparable. 
Since $M$ is indecomposable, $0\subsetneq M_1\cap M_2 \subsetneq M_i\subsetneq M$, for $i=1, 2$. Thus, $1\le \ell( M_1\cap M_2) \le 2$ and $2\le \ell(M_i)\le 3,$ for $i=1, 2.$ 

We claim that one of $\ell(M_1)$ and $\ell(M_2)$ is $3$, and the other is $2$. Indeed, by Lemma \ref{ses_sum_itsec}(1), $\ell(M_1)+\ell(M_2)=\ell(M)+\ell(M_1\cap M_2)=\ell(M_1\cap M_2)+4$. Thus, $\ell(M_1)=3$ or $\ell(M_2)=3.$ Suppose on the contrary that $\ell(M_1)=\ell(M_2)=3.$ Then, $\ell(M_1\cap M_2)=2=\ell(\rad M)$. By Lemma \ref{loc_gen}(2), $M_i$ is local with $\ntop(M_i)\cong S_i$ and $\ell\ell(M_i)\le \ell\ell(M)=2,$ and by Proposition \ref{CL_3}(2), $M_i$ is astride biserial with $\rad M_i=\soc M_i,$ for $i=1, 2$. Then $M_1\cap M_2\subseteq \rad M_i\subseteq \rad M,$ and therefore, $\rad M_i=\rad M=S_3 \oplus S_4,$ for $i=1, 2$. By Lemma \ref{4bs_loc_l3}(1), we obtain the following 

Statement: Both $Ae_1$ and $Ae_2$ are biserial, whereas either $I_{S_3}$ or $I_{S_4}$ is uniserial.

We shall contradict this statement by considering $M/S_3$ and $M/S_4$, both of which have length three. For $i=3, 4$, since $S_i\subsetneq \rad M$, we see that $\ell\ell(M/S_i)=2$ and $\ntop(M/S_i)\cong \ntop M=S_1\oplus S_2.$ 

Suppose that both $M/S_3$ and $M/S_4$ are indecomposable. By Lemma \ref{soc_rad}, $\soc(M/S_3)=\soc(\rad(M/S_3))=\soc((S_3\oplus S_4)/S_3)\cong S_4,$ and $\soc(M/S_4)\cong S_3$. Being colocal of Loewy length two, by Proposition \ref{CL_3}(3), $M/S_3$ and $M/S_4$ are co-astride biserial of length three. And by Lemma \ref{bsa_ml3}(2), both $I_{S_4}$ and $I_{S_3}$ are biserial, a contradiction to Statement. 

So we may assume that $M/S_3$ is decomposable. Since $\ell(\ntop(M/S_3))=2,$ we may write $M/S_3=N/S_3 \oplus L/S_3$, where $S_3\subseteq N \cap L$ such that $N/S_3$ is simple and $L/S_3$ is uniserial of length two. Then, $M=N+L$ and $N\cap L=S_3.$ Moreover, $\ell(N)=2$ and $\ell(L)=3$. If $N$ is decomposable, then $N=T\oplus S_3,$ where $T$ is simple. This yields $M=T+S_3+L=T+L=T\oplus L,$ absurd. Thus, $N$ is uniserial and $S_3=\rad N.$ Suppose that $L$ is decomposable. Since $\ell(\ntop(L/S_3))=1$, we see that $L=S_3\oplus L_1$. Then, $M=N+L_1$ and $N\cap L_1=N\cap L\cap L_1=S_3\cap L_1=0$. That is, $M=N\oplus L_1$, absurd. Thus, $L$ is indecomposable with $S_3\subseteq \rad L;$ see (\ref{soc_rad}). Hence, $\ntop L\cong \ntop(L/S_3),$ which is simple. Since $\ell\ell(L)\le \ell\ell(M)=2$, by Proposition \ref{CL_3}(2), $L$ is astride biserial. Since $N\cap L= S_3\subseteq \rad N \cap \rad L,$ by Lemma \ref{ses_sum_2}(2), $\ntop N\oplus \ntop L\cong \ntop M\cong S_1\oplus S_2.$ Thus, $\ntop L\cong S_1$ or $S_2$, and by Lemma \ref{bsa_ml3}(1), $Ae_1$ or $Ae_2$ is biserial, a contradiction to Statement. 

Therefore, our claim holds. Then, we may assume that $\ell(M_1)=3$ and $\ell(M_2)=2$. As mentioned above, $M_1$ and $M_2$ are local of Loewy length at most two. Thus, $M_1$ is astride biserial of length three by Proposition \ref{CL_3}, and $M_2$ is uniserial of length two. That is, $M$ is an N-shaped module. The proof of the theorem is completed.

%\vspace{-3pt}

\subsection{\sc Reduction of quadri-biserial algebras} In this subsection, we show that the representation theory of quadri-biserial algebras reduces to that of quadri-string algebras with radical cubed zero. To do so, we determine which projective modules are necessarily injective.

\begin{Lemma}\label{4sa_ninj_prj}

Let $A$ be a quadri-biserial algebra. Consider an indecomposable projective module $P$ in $\mmod A$ or $\mmod A^\circ\hspace{-1pt}.$

\begin{enumerate}[$(1)$]

\vspace{-2pt}
    
\item If $P$ is of Loewy length four, then $P$ is uniserial and projective-injective.

\item If $P$ is biserial, then $P$ is projective-injective if and only if it is colocal, \vspace{-.5pt} or equivalently, it is a lozenge module.
    
\end{enumerate}\end{Lemma}

\vspace{-1pt}

\noindent{\it Proof.} By Definition \ref{4str_alg_def}(1), $\ell(P)\le 4$ and $\ell(\rad P)\le 3$. If $\ell\ell(P)=4,$ then $P$ is a uniserial of length four, and it is injective by Lemma \ref{maxi_local} and Definition \ref{4str_alg_def}(1). 
Suppose next that $P$ is biserial. If $P$ is injective or lozenge, then it is clearly colocal. Conversely if $P$ is colocal, then  $P$ is lozenge; see (\ref{rb4_lozenge}); and since $\ell(P)=4,$  by Lemma \ref{maxi_local}(2), $P$ is injective. The proof of the lemma is completed.

\smallskip

We are ready to obtain the main result of this subsection.

\begin{Prop}\label{4ba-4sa}

Let $A$ be a connected non-simple quadri-biserial algebra with ${}_A\hspace{-.5pt}A=P\oplus L$, \vspace{-.5pt} where $P$ is a maximal injective direct summand of ${}_A\hspace{-.5pt}A$. Then $\soc P$ is a two-sided ideal in $A$ such that $A/\soc P$ is a quadri-string algebra with radical cubed zero. Moreover, every indecomposable non-projective-injective module in $\mmod A$ is a module over $A/\soc P$.

%\begin{enumerate}[$(1)$]
%\item $\hspace{2pt}\overline{\hspace{-2.8pt}A}= A/\soc P$ is a $4$-string algebra with radical cubed zero$\,;$
%\item the indecomposable not projective-injective modules in $\mmod A$ are $\hspace{2pt}\overline{\hspace{-2.8pt}A}$-modules. \end{enumerate}

\end{Prop}

\vspace{-1.5pt}

\noindent{\it Proof.} If $P=0,$ then $A$ is a quadri-string algebra with radical cubed zero by Lemma \ref{4sa_ninj_prj}. Otherwise, $P=Ae_1\oplus \cdots\oplus Ae_s$ and $L=Ae_{s+1}\oplus \cdots \oplus Ae_t$, where $\{e_1, \ldots, e_s,\ldots, e_t\}$ is a complete orthogonal set of primitive idempotents in $A$. It is well-known that $\soc P$ is a two-sided ideal in $A,$ which annihilates all indecomposable non-projective-injective modules in $\mmod A;$ see \cite[(4.2)]{AuR}. Write 
$\hspace{2pt}\overline{\hspace{-2.8pt}A}:=A/\soc P=\{\bar a \mid a\in A\}$ and $\bar J=\{\bar a \mid a\in J\},$ where $\bar a=a+\soc P.$

\vspace{.5pt}

{\sc Statement 1.} {\it The set \vp $\{\bar e_1, \ldots, \bar e_s, \ldots, \bar e_t\}$ is a complete orthogonal set of primitive idempotents in $\hspace{2pt}\overline{\hspace{-2.8pt}A}$ such that $\hspace{2pt}\overline{\hspace{-2.5pt}A} \bar e_i / \bar J \bar e_i \cong Ae_i/Je_i$ for all $1\le i\le t.$ Moreover, $\hspace{2pt}\overline{\hspace{-2.5pt}A} \bar e_i \cong Ae_i/\soc(Ae_i)$ for $1\le i\le s,$ and $\hspace{2pt}\overline{\hspace{-2.5pt}A} \bar e_j \cong Ae_j$ for $s<j\le t$.}

Since $A$ is connected and not simple, $Ae_i$ is not simple, and by Lemma \ref{soc_rad}, $\soc(Ae_i)\subseteq Je_i$, for $1\le i\le s$. Thus $\soc P\subseteq J$. It is then well-known that $\{\bar e_1, \ldots, \bar e_t\}$ is a complete orthogonal set of primitive idempotents in $\hspace{2pt}\overline{\hspace{-2.8pt}A}$ such that $\hspace{2pt}\overline{\hspace{-2.5pt}A} \bar e_i \cong Ae_i/(\soc P)e_i$ and $\hspace{2pt}\overline{\hspace{-2.5pt}A} \bar e_i / \bar J \bar e_i \cong Ae_i/Je_i,$ for $i=1, \ldots, t$. Moreover since $\soc(Ae_i) =\soc(Ae_i)e_i$ for $1\le i\le s,$ we have $(\soc P)e_i=\soc(Ae_i)$ for $1\le i\le s,$ and $(\soc P)e_j=0$ for $s<j\le n$. This establishes Statement 1.

\vspace{1pt}

{\sc Statement 2.} {\it If $\ell(\hspace{2pt}\overline{\hspace{-2.5pt}A} \bar e_i)\ge 3$ for some $1\le i\le t,$ then $\ntop(\bar J \bar e_i) =\ntop(Je_i).$}

\vspace{1pt}

Suppose that $\ell(\hspace{2pt}\overline{\hspace{-2.5pt}A} \bar e_i)\ge 3,$ for some $1\le i\le t.$ Then $\ell(Je_i)\ge 2$. By Statement 1, we may assume that $1\le i\le s$. Since $\soc(Ae_i)$ is simple, $\soc(Ae_i)\subseteq J^2e_i$. By Statement 1 again, $\ntop(\bar J \bar e_i)\cong \ntop(Je_i/\soc(Ae_i))\cong \ntop(Je_i).$ Statement 2 holds.

\vspace{1pt}

{\sc Statement 3.} {\it If $1\le i\le t,$ then $\hspace{2pt}\overline{\hspace{-2.8pt}A} \hspace{.5pt}\bar e_i$ is uniserial or astride biserial such that $\ell(\hspace{2pt}\overline{\hspace{-2.8pt}A} \bar e_i)\le 4$ and $\ell\ell(\hspace{2pt}\overline{\hspace{-2.8pt}A} \bar e_i)\le 3.$}  

\vspace{1pt}

First of all, $\ell(\hspace{2pt}\overline{\hspace{-2.8pt}A} \bar e_i)\le \ell (Ae_i) \le 4$ for all $1\le i\le t.$ Consider $s<i\le t.$ By Statement 1, $\hspace{2pt}\overline{\hspace{-2.8pt}A} \bar e_i\cong Ae_i,$ which is not injective. By Lemma \ref{4sa_ninj_prj}, $\hspace{2pt}\overline{\hspace{-2.8pt}A} \bar e_i$ is uniserial of length at most three or astride biserial of length at most four. In the first case,  $\ell\ell(\hspace{2pt}\overline{\hspace{-2.8pt}A} \bar e_i) \le \ell(\hspace{2pt}\overline{\hspace{-2.8pt}A} \bar e_i)\le 3.$ In the second case, $\rad (J \bar e_i)$ is a direct sum of two uniserial of length at most two. Thus, 
$\ell\ell(\hspace{2pt}\overline{\hspace{-2.8pt}A} \bar e_i)= \ell\ell(J\bar e_i)+1 \le 3.$
Now, consider $1\le i\le s$. \vp By Statement 1, $\hspace{2pt}\overline{\hspace{-2.8pt}A} \bar e_i\cong Ae_i/\soc(Ae_i).$ %where $\soc (Ae_i)$ is simple. 
If $Ae_i$ is uniserial, then $\hspace{2pt}\overline{\hspace{-2.8pt}A} \hspace{.5pt} \bar e_i$ is uniserial with $\ell\ell(\hspace{2pt}\overline{\hspace{-2.8pt}A} \hspace{.5pt} \bar e_i)\le 3.$ Otherwise by Lemma \ref{4sa_ninj_prj}(2), $Ae_i$ is lozenge. Thus by Proposition \ref{biloc_bis}(3), $\hspace{2pt}\overline{\hspace{-2.8pt}A} \hspace{.5pt} \bar e_i$ is astride biserial of length three, and by Proposition \ref{CL_3}(2),  $\ell\ell(\hspace{2pt}\overline{\hspace{-2.8pt}A} \bar e_i)=2.$ This establishes Statement 3. 

%\vspace{.5pt}

{\sc Statement 4.} {\it There exists a permutation $\sigma$ on $\{1, \ldots, t\}$ such that $e_{\sigma(i)}A$ is injective if and only if $1\le i\le s;$ and in this case, $\soc(e_{\sigma(i)}A)\cong e_iA/e_iJ$ and $\soc(Ae_i)\cong Ae_{\sigma(i)}/Je_{\sigma(i)}$.  Moreover, $\soc P=\oplus_{i=1}^s \hspace{.5pt} \soc(e_{\sigma(i)}A)$. }

Note that $Ae_i\cong Ae_j$ if and only if $e_iA\cong e_jA,$ for $1\le i, j\le t.$ So the first part of Statement 4 follows from Lemma \ref{pi_corresp}. Put $P':=\oplus_{i=1}^s \hspace{.5pt} e_{\sigma(i)}A,$ a two-sided ideal in $A$. By Lemma \ref{pi_corresp}, $\soc(e_{\sigma(i)}A)=e_{\sigma(i)} \hspace{.6pt} \soc(Ae_i) A \subseteq \soc P$ for $1\le i\le s$. Thus, $\soc P'\subseteq \soc P.$ Similarly, $\soc P\subseteq \soc P'.$ This establishes Statement 4.

We see from Statement 4 that Statements (2) and (3) hold for the $\bar e_i \hspace{2pt}\overline{\hspace{-2.8pt}A}$ with $1\le i\le t$. In particular, $\hspace{2pt}\overline{\hspace{-2.8pt}A}$ is a string algebra with radical cubed zero, which satisfies Definition \ref{4str_alg_def}(1). Assume that $\ntop(\bar J \bar e_i)\cong \hspace{2.8pt}\overline{\hspace{-2.8pt}A} \bar e_{i_1}/\bar J \bar e_{i_1} \oplus \hspace{2.8pt}\overline{\hspace{-2.8pt}A} \bar e_{i_2}/\bar J \bar e_{i_2}$, where $1\le i, i_1, i_2\le t$. By Statement 2, $\ntop(Je_i)\cong  A e_{i_1}/ J e_{i_1} \oplus  A e_{i_2}/ J e_{i_2},$   and by Definition \ref{4str_alg_def}(2), $e_{i_1}A$ or $e_{i_2}A$ is uniserial. By Statement 2, so is $\bar e_{i_1} \hspace{2.8pt}\overline{\hspace{-2.8pt}A}$ or $\bar e_{i_2}\hspace{2.8pt}\overline{\hspace{-2.8pt}A}.$ Now, suppose that $\bar J \bar e_i = S \oplus M,$ where $S\cong \hspace{2.8pt}\overline{\hspace{-2.8pt}A} \bar e_{p_i} / \bar J \bar e_{p_i}$ and $M$ is uniserial of length two with $\soc M\cong \hspace{2.8pt}\overline{\hspace{-2.8pt}A} \bar e_{q_i} / \bar J \bar e_{q_i}$, for some $1\le p_i, q_i\le t.$ By Statement 2, $J e_i \cong S \oplus M,$ and by Definition \ref{4str_alg_def}(3), $e_{p_i}A$ and $e_{q_i}A$ are uniserial. By Statement 2, so are $\bar e_{p_i}\hspace{2.8pt}\overline{\hspace{-2.8pt}A}$ and $\bar e_{q_i}\hspace{2.8pt}\overline{\hspace{-2.8pt}A}$. \vspace{-1pt} Thus, the indecomposable projective modules in $\mmod \hspace{2.8pt}\overline{\hspace{-2.8pt}A}$ satisfy the three conditions stated in Definition \ref{4str_alg_def}. Similarly, so do those in $\mmod \hspace{2.8pt}\overline{\hspace{-2.8pt}A}^\circ.$ The proof of the proposition is completed.

\vspace{-2pt}

\subsection{\sc Almost split sequences} In this subsection, we compute more almost split sequences over quadri-biserial algebras, particularly those for N-shaped modules over quadri-string algebras with radical cubed zero. \vspace{-1pt}

\begin{Prop}\label{ass_p4_1} 

Let $A$ be a quadri-biserial algebra. Consider a biserial projective module $P$ in $\mmod A$ with $\rad P=S\oplus M,$ where $S$ is simple and $M$ is uniserial of length two such that $I_{\hp\ntop M}$ is uniserial. Then we have two almost split sequences

\begin{enumerate}[$(1)$]

\vspace{-3pt}

\item $\xymatrixcolsep{22pt}\xymatrix{0\ar[r] & \ntop M\ar[r] & \ntop^2\hspace{-1pt} P \ar[r]  & P/M \ar[r] & 0;}$ 

\vspace{-3pt}

\item $\xymatrixcolsep{22pt}\xymatrix{0\ar[r] & \ntop^2\hspace{-1pt} P \ar[r] & P/\soc P\oplus P/M \ar[r] & \ntop P \ar[r] & 0.}$

\end{enumerate}\end{Prop}

\vspace{-3pt}

\noindent{\it Proof.} Note that $\rad^2\hspace{-1pt}P=\rad M=\soc M.$ By Definition \ref{4str_alg_def}(3), $I_S$ and $I_{\soc M}$ are uniserial. And by Propositions \ref{ass_p_astride} and \ref{ass_st_mod}, we have two almost split sequences \vspace{-1pt}

\noindent $(*) \hspace{14pt} \xymatrixcolsep{28pt}\xymatrix{ 0\ar[r]& P \ar[r]^-{(p_1, \,p_2)^T\hspace{-2pt}} & P/S \oplus \ntop^2\hspace{-1pt} P \ar[r]^-{(-p_3, \, p_4)\hspace{-2pt}} & P/\soc P \ar[r] & 0,}$

\noindent where the $p_i$ are canonical projections, and

\noindent $(**) \hspace{10pt}\xymatrixcolsep{28pt}\xymatrix{0 \ar[r]& M \ar[r]^-{(q_0, -p_0)^T\hspace{-3.5pt}} & P \oplus \ntop M \ar[r]^-{(p_2, q)} & \ntop^2\hspace{-1pt} P \ar[r] & 0,}$ 

\noindent where $q_0, q$ are inclusion maps, and $p_0$ is the canonical projection. Consider the canonical short exact sequence 

\noindent $(1) \hspace{14pt} \xymatrixcolsep{28pt}\xymatrix{0\ar[r] & \ntop M\ar[r]^-q & \ntop^2\hspace{-1pt} P \ar[r]^-p & P/M\ar[r] & 0.}$

\vp\vp

To show that this is an almost split sequence, we may assume that $P=Ae_1$ and $\ntop M \cong Ae_2/Je_2,$ where $e_1, e_2$ are primitive idempotents in $A$. Then $M=Au,$ for some $u\in e_2Je_1\backslash e_2 J^2 e_1$. This yields minimal projective presentations \vspace{-5pt}
$$\xymatrixcolsep{20pt}\xymatrix{Ae_2 \ar[r]^{R_u} & Ae_1 \ar[r] & P/M \ar[r] & 0} \mbox{ and } \xymatrixcolsep{20pt}\xymatrix{e_1A \ar[r]^{L_u} & e_2A \ar[r] & {\rm Tr}(P/M) \ar[r] & 0,}\vspace{-5.5pt}$$ where $R_u$ and $L_u$ denote right and left multiplications by $u$, respectively. Therefore, ${\rm Tr}(P/M)\cong e_2A/uA$. \vspace{.5pt} Since $D(e_2A)\cong I_{\hp\ntop M}$, we see that $e_2A$ is uniserial. Hence $e_2J=uA,$ and ${\rm Tr}(P/M)\cong e_2A/e_2J$. So $D{\rm Tr}(P/M)\cong Ae_2/Je_2\cong \ntop M,$ which is simple. By Lemma \ref{ass-sim}(1), the sequence $(1)$ is an almost split sequence. 

Next, since $\soc P\!=\!\soc(\rad P)\!=\!S+\soc M,$ we have $(\soc P) \hp \cap M \!=\! \soc M \!=\! \rad^2\hspace{-1pt}P$ and $\soc P+M \!=\! S+M \!=\! \rad P$. By Lemma \ref{ses_sum_itsec}(2), we have a short exact sequence  

\noindent $(2) \hspace{15pt} \xymatrixcolsep{28pt}\xymatrix{0\ar[r] & \ntop^2\hspace{-1pt} P \ar[r]^-{(p_4, \hp p)^T} & P/\soc P \oplus P/M \ar[r]^-{(p_5, \hp\hp p_6)} & \ntop P\ar[r] & 0,}$ 

\noindent where $p_5, p_6$ are canonical projections. \vp Viewing the almost split sequences $(*)$ and $(1)$, we see that $(p_4,p)^T$ is irreducible; see \cite[(1.3)]{LiY}. Since $A$ is biserial and $\ntop^2\hspace{-1pt}P$ is not colocal; see (\ref{bsm_top2}), by Proposition \ref{loc_cloc_bsm}(1), \vp $\ntop^2\hspace{-1pt}P$ is not a direct summand of any indecomposable projective module in $\mmod A.$ And since $(p_2, q)$ is minimal right almost split, $(p_2,p)^T$ is minimal left almost split. Thus the sequence $(2)$ is also an almost split sequence. The proof of the proposition is completed.  

\smallskip

Dually, we have the following result. \vspace{-1pt}

\begin{Prop}\label{ass_i4_1} 

Let $A$ be a quadri-biserial algebra. Consider a biserial injective module  $I$ in $\mmod A$ with $I/\soc I = (M_1/\soc I) \oplus (M_2/\soc I),$ where $M_1/\soc I$ is simple, and $M_2/\soc I$ is uniserial of length two such that $P_{\soc(M_2/\soc I)}$ is uniserial. Then we have two almost split sequences

\begin{enumerate}[$(1)$]

\vspace{-2pt} 

\item $\xymatrixcolsep{22pt}\xymatrix{0\ar[r]& M_1 \ar[r] & \soc^2\hspace{-.8pt} I \ar[r] & \soc\hp(\hspace{-.5pt}M_2/\soc I) \ar[r] & 0\hp;}$

\vspace{-2.5pt}

\item $\xymatrixcolsep{22pt}\xymatrix{0\ar[r]& \soc I \ar[r] & M_1 \oplus \rad M_2 \ar[r] & \soc^2\hspace{-1.2pt}I \ar[r] & 0.}$

\end{enumerate}\end{Prop}

\vspace{-.5pt}

\noindent{\it Proof.} Note that $I=M_1+M_2$ and $M_1\cap M_2=\soc I=\soc M_1=\rad M_1,$ where $M_1, M_2$ are uniserial of length two and three respectively. Thus $\rad I=\rad M_2.$ Consider the biserial projective module $DI$ in $\mmod A^\circ$. By Corollary \ref{dual_perp} and Proposition \ref{duality_radi_soci}(3), $\rad(DI)=(\soc I)^\perp=M_2^\perp \oplus M_1^\perp \hspace{-1pt}.$ Since $I/M_2\cong M_1/\soc I$, by Lemma \ref{dual_perp}(1), $M_2^\perp\cong D(M_1/\soc I),$ which is simple. Similarly $M_1^\perp\cong D(M_2/\soc I),$ which is uniserial of length two, \vp and $\ntop(M_1^\perp) \cong D(\soc(M_2/\soc I)).$ Therefore, $I_{\hp\ntop(M_1^\perp)}\cong D(P_{\soc(M_2/\soc I)}),$ \vp which is uniserial by the hypothesis. \vp Now by Proposition \ref{ass_p4_1}, we have two almost split sequences 

\vp\vp

$\xymatrixcolsep{22pt}\xymatrix{0\ar[r] & \ntop(M_1^\perp) \ar[r] & \ntop^2(DI) \ar[r]  & DI/M_1^\perp \ar[r] & 0;}$

\vspace{-2pt}

$\xymatrixcolsep{22pt}\xymatrix{0\ar[r] & \ntop^2(DI) \ar[r] & DI/\soc(DI) \oplus DI/M_1^\perp \ar[r] & \ntop(DI) \ar[r] & 0.}$ 

\vspace{1pt}

Note that $DI/\soc(DI) \cong D(\rad I)=D(\rad M_2)$ and  $\ntop^i(DI)\cong D(\soc^i\hspace{-1pt}I);$ see (\ref{duality_radi_soci}), for $i=1, 2$. Moreover, $DI/M_1^\perp \cong DM_1;$ see (\ref{dual_perp}). Applying the duality $D$ to the above almost split sequences, we obtain two desired almost split sequences stated in the proposition. The proof of the proposition is completed.

%\begin{Lemma}\label{4sa_uni_ideal}Let $A$ be a $4$-biserial algebra with $u\in e_2Je_1\backslash e_2J^2e_1,$ where $e_1, e_2$ are primitive idempotents. Then
%\begin{enumerate}[$(1)$]\vspace{-2pt}
%\item  $Au$ is uniserial, which is a direct summand of $Je_1$ in case $Je_1$ is decomposable$\,:$
%\item $uA$ is uniserial, which is a direct summand of $e_2J$ in case $e_2J$ is decomposable.\end{enumerate} \end{Lemma}
%\noindent{\it Proof.} We shall only prove Statement (1). Note that $Au$ is a local submodule of $Je_1$ and $Au\not\subseteq J^2e_1.$ Since $\ell(Ae_1)\le 4$, we have $\ell(Je_1)\le 3.$ If $Ae_1$ is uniserial, then so is $Je_1$. Hence, $Au=Je_1$. Statement (1) holds in this case. 
%
%Suppose that $Ae_1$ is biserial. By Proposition \ref{loc_cloc_bsm}, $\ell(\ntop(Je_1))=2$. In particular, $Au\ne Je_1,$ and hence, $1\le \ell(Au)\le 2.$ Therefore, $Au$ is uniserial. Suppose that $Je_1$ is decomposable. If $J^2e_1=0,$ then $Au$ is a direct summand of $Je_1$. Otherwise, $\ell(Je_1)=3$ and $\ell(Ae_1)=4$. Since $\ell(\ntop(Je_1))=2$, we have $Je_1=S\oplus M$, where $S$ is simple and $M$ is uniserial of length $2$. Since $Au\not\subseteq J^2e_1= \rad M,$ by Proposition \ref{dec_mod_l3}, $Au$ is a direct summand of $Je_1$. The proof of the lemma is completed.

%\smallskip

\vspace{5pt}

Next, we turn to compute almost split sequences for N-shaped modules over quadri-string algebras with radical cubed zero.

%\vspace{-1pt}

\begin{Prop}\label{ass_Nmod}

Let $A$ be a quadri-string algebra with radical cubed zero. Consider an N-shaped module $M=M_1+M_2$ in $\mmod A$ with $\soc M = S \oplus \soc M_2,$ where $M_1$ is astride biserial of length three, $M_2$ is uniserial of length two, and $S$ is simple. Then we have three almost split sequences as follows$\hp:$

\begin{enumerate}[$(1)$]

\vspace{-1pt}
    
\item $\xymatrixcolsep{20pt}\xymatrix{0\ar[r] & M_1\cap M_2 \ar[r] & M_1\oplus M_2 \ar[r] & M \ar[r] & 0\hp;}$ 

\vspace{-3pt}

\item $\xymatrixcolsep{20pt}\xymatrix{0\ar[r] & M \ar[r] & (M / S) \oplus (M/M_2) \ar[r] & \ntop M_1 \ar[r] & 0;}$

\vspace{-3pt}

\item $\xymatrixcolsep{20pt}\xymatrix{0\ar[r] & M_2 \ar[r] & M \ar[r] & M/M_2 \ar[r] & 0.}$

\end{enumerate}\end{Prop}

\vspace{-2.5pt}

\noindent{\it Proof.} By Lemma \ref{N_mod_property_0}, $M_1\cap M_2=\soc M_2$ and $\soc M=\soc M_1=\rad M_1$. Write $M_i=Am_i,$ where $m_i\in e_iM_i$ with $e_i$ a primitive idempotent in $A,$ for $i=1, 2.$ Suppose that $S=Am_3$ and $\soc M_2=Am_4,$ where $m_j\in e_j M_1$ with $e_j$ a primitive idempotent in $A,$ for $j=3, 4$. Since $\soc M_1=\rad M_1$ and $\soc M_2=\rad M_2$, we see that $m_3=u_1m_1$ for some $u_1\in e_4Je_1\backslash e_4J^2e_1,$ and $m_4=v_1m_1=v_2m_2$ for some $v_j\in e_4Je_j\backslash e_4J^2e_j$ for $j=1, 2$. We claim that the following statements hold;

\vspace{1pt}

{\sc Statement 1.} {\it $Au_1$ is simple, $Je_2 = Av_2,$ and $Je_1=Au_1\oplus Av_1.$}

{\sc Statement 2.} {\it $v_2A$ is simple, $e_3J = u_1A,$ and $e_4J=v_1A\oplus v_2A.$}

Indeed by Lemma \ref{N_mod_property-2}, 
$Ae_1$ and $e_4A$ are biserial, while $Ae_2$ and $e_3A$ are uniserial. Thus, $Je_2=Av_2$ and $e_3J=u_1A.$ Since $e_3A\not\cong e_4A,$ we have $Ae_3\not\cong Ae_4.$ Hence by Propositions \ref{tb_gen}(1) and \ref{loc_cloc_bsm}(1), $\{u_1, v_1\}$ is a top-basis for $Je_1$. Since $e_4A$ is biserial, it follows from Lemma \ref{4sa_bproj} that $Au_1$ is simple, and $Je_1=Au_1\oplus Av_1$. Dually since $Ae_1\not\cong Ae_2$ and $Ae_1$ is biserial, we conclude that $v_2A$ is simple, and $e_4J=v_1A\oplus v_2A$. This establishes our claim. Next, we shall verify one by one that the three sequences stated in the proposition are almost split sequences.

(1) By Lemma \ref{ses_sum_itsec}(1), we have a non-splitting short exact sequence \vspace{-6pt} $$\xymatrixcolsep{28pt}\xymatrix{0\ar[r] & M_1\cap M_2 \ar[r]^-{(q_1, q_2)} & M_1\oplus M_2 \ar[r]^-{(-q_3, q_4)} & M \ar[r] & 0,}\vspace{-5.5pt}$$ where the $q_i$ are inclusion maps. Since $M_1\cap M_2$ is simple, by Lemma \ref{ass-sim}(1), it amounts to show that $D{\rm Tr}M\cong M_1\cap M_2$. \vspace{-.5pt} Observing that $\{m_1, m_2\}$ is a top-basis for $M$, we obtain from Proposition \ref{tb-projc-ienv} a projective cover $f: Ae_1\oplus Ae_2\to M$ such that $f(e_1, 0)=m_1$ and $f(0, e_2)=m_2.$ Since $v_1m_1=v_2m_2=m_4$, we see that $(v_1, -v_2)\in e_4 (\Ker f).$ 
We shall show that $\Ker f=A(v_1, -v_2).$ 

Consider $u\in \Ker f.$ We may assume that $u\in e_5Je_1\oplus e_5Je_2,$ for some primitive idempotent $e_5$ in $A$. By Statement 1, we have $e_5Je_1=(e_5Ae_3)u_1+(e_5Ae_4)v_1$ and $e_5Je_2=(e_5Ae_4)v_2.$ Thus, $u=(a u_1 + b_1 v_1, b_2v_2),$ for some $a\in e_5Ae_3$ and $b_1, b_2\in e_5Ae_4.$ Then, $0=f(u)=(au_1+b_1v_1) m_1 + b_2v_2 m_2 = am_3+(b_1+b_2)m_4.$
Since $\{m_3, m_4\}$ is a soc-basis for $M,$ it follows from Lemma \ref{top-element}(2) that $a\in e_5Je_3$ and $b_1+b_2\in e_5Je_4.$ Since $Au_1$ is simple, $au_1=0$. Thus, $u=(b_1v_1, b_2v_2)$ and $(b_1+b_2)m_4=0.$ Set $v:=b_1+b_2\in e_5Je_4.$ Since $Ae_1\not\cong Ae_2$ as previously seen, by Lemma \ref{sbsa}(2), $vv_1=0$ or $vv_2=0$. Thus,
$u=((v-b_2)v_1, b_2v_2)=(-b_2)(v_1, -v_2)$ or $u=(b_1v_1, (v-b_1)v_2)=b_1(v_1, -v_2)$. This shows that $\Ker f=A(v_1, -v_2).$ 
Now, since $(v_1, -v_2)\notin e_4J^2e_1\oplus e_4J^2e_2,$
we have a minimal projective presentation \vspace{-5pt}
$$\xymatrixcolsep{40pt}\xymatrix{Ae_4\ar[r]^-{\hspace{4pt}(R_{v_1}, R_{-v_2})^T} & Ae_1\oplus Ae_2 \ar[r]^-f & M\ar[r] & 0,}\vspace{-3.5pt}$$ where $R_{u}$ is the right multiplication by $u$, and a minimal projective presentation \vspace{-5pt}
$$\xymatrixcolsep{40pt}\xymatrix{e_1A\oplus e_2A\ar[r]^-{(L_{v_1}, \, L_{-v_2})} & e_4A \ar[r]^-g & {\rm Tr} M\ar[r] & 0,}\vspace{-3pt}$$ where $L_{u}$ is the left multiplication by $u$. Then ${\rm Im}(L_{v_1}, \, L_{-v_2})=v_1A+(-v_2)A=e_4J$ by Statement 2. So ${\rm Tr}M \cong e_4A/e_4J,$ and  
$D{\rm Tr}M\cong Ae_4/Je_4\cong \soc M_2=M_1\cap M_2.$ 

\vp

(2) Since $\soc M_1= S \oplus \soc M_2,$ we have $S\cap M_2= (S\cap M_1)\cap M_2=S\cap \soc M_2=0.$ In view of Lemma \ref{ses_sum_itsec}(2), we obtain a short exact sequence \vspace{-5.5pt}
$$\xymatrixcolsep{30pt}\xymatrix{0\ar[r] & M \ar[r]^-{(p_1, \,p_2)^T} & (M / S) \oplus (M/M_2)\ar[r]^-{(-p_3, \,p_4)} & M/ (S + M_2) \ar[r] & 0} \vspace{-4.5pt}$$ in $\mmod A,$ where the $p_i$ are canonical projections. Note that $M=M_1+(S+M_2)$ and $M_1\cap (S+M_2)=S+M_1\cap M_2=S+\soc M_2=\soc M_1=\rad M_1.$ This yields $M/(S+M_2)\cong \ntop M_1.$ By Lemma \ref{ass-sim}(1), it suffices to show that ${\rm Tr}D M\cong \ntop M_1.$ Indeed, since $S\cap M_2=0,$ by Lemma \ref{dual_perp}, $DM=S^\perp+M_2^\perp,$ where $S^\perp \cong D(M/S)$ and $M_2^\perp\cong D(M/M_2).$ By Lemma \ref{N_mod_property}, $S^\perp$ is astride biserial of length three and $M_2^\perp$ is uniserial of length two. By definition, $DM$ is N-shaped. Observe that $DM_1$ is co-astride of length three. By Proposition \ref{CL_3}(3), $\rad(DM_1)= \soc(DM_1).$ Now, 
$S^\perp \cap M_2^\perp=(S+M_2)^\perp=(\soc M_1)^\perp=\rad(DM_1)= \soc(DM_1)\cong D(\ntop M_1).$ As has been shown, $D{\rm Tr}(DM) \cong S^\perp \cap M_2^\perp\cong D(\ntop M_1);$ consequently  ${\rm Tr}D M \cong \ntop M_1.$ 

(3) Consider the canonical short exact sequence 
\vspace{-4.5pt}
$$\xymatrix{0\ar[r] & M_2 \ar[r]^{q_2} & M \ar[r]^{p_2} & M/M_2 \ar[r] & 0.}\vspace{-4pt}$$ In view of the almost split sequences (1) and (2), $q_2, p_2$ are irreducible. It is well known that this is an almost split sequence; see \cite[(V.5.9)]{ARS}. The proof of the proposition is completed.

\smallskip

Finally, we shall find some conditions for the existence of N-modules over quadri-string algebras and compute the almost split sequences around them. 

\begin{Prop}\label{4str_ass_smr} 

Let $A$ be a quadri-string algebra with radical cubed zero. Consider a biserial projective module $P$ in $\mmod A$ with $\rad P=S_1\oplus M,$ where $S_1$ is simple and $M$ is uniserial of length one or two such that $I_{\ntop M}$ is biserial. Then we have four almost split sequences as follows$\hp:$ 

\begin{enumerate}[$(1)$]

\vspace{-1.5pt}

\item $\xymatrixcolsep{20pt}\xymatrix{0\ar[r] & N\ar[r] & \soc^2\hspace{-.5pt}(I_{\ntop M})\oplus P/M \ar[r] & \ntop P \ar[r] & 0\hp ;}$

\vspace{-2.5pt}

\item $\xymatrixcolsep{20pt}\xymatrix{0\ar[r] & \ntop M \ar[r] & \ntop^2\hspace{-1pt}P \oplus L\ar[r] & N\ar[r] & 0\hp ;}$

\vspace{-2.5pt}

\item $\xymatrixcolsep{20pt}\xymatrix{0\ar[r] & L \ar[r] & N\ar[r] & P/M\ar[r] & 0\hp ;}$

\vspace{-3.5pt}

\item $\xymatrixcolsep{20pt}\xymatrix{0\ar[r] & \ntop^2\hspace{-1pt}P \ar[r] & N\oplus P/(S_1\oplus \rad M) \ar[r] & \soc^2\hspace{-.5pt}(I_{\ntop M}) \ar[r] & 0,}$

\end{enumerate}

\vspace{-2pt}

\noindent where $N$ is an N-shaped module and $L$ is a uniserial module of length two.

\end{Prop}

\vspace{-1.5pt}

\noindent{\it Proof.} We may assume that $P= Ae$ with $S_1\cong Ae_1/Je_1$ and $\ntop M\cong Ae_2/Je_2$, where $e, e_1, e_2$ are primitive idempotents in $A$. So $\ntop(Je)\cong Ae_1/Je_1\oplus Ae_2/Je_2$. Since $D(e_2A)$ is biserial by hypothesis, we see from Definition \ref{4str_alg_def}(2) that $D(e_1A)$ is uniserial. Thus, $e_2A$ is biserial and $e_1A$ is uniserial. In particular, $e_1A\not\cong e_2A$. Write $S_1=Au_1$ and $M=Au_2$, where $u_i\in e_iJe\backslash e_iJ^2e$ for $i=1, 2$. Then, $\{u_1, u_2\}$ is a top-basis for $Je.$ Setting $S:=Ae/Je$, we obtain a minimal projective presentation \vspace{-10pt}
$$\xymatrixcolsep{35pt}\xymatrix{Ae_1\oplus Ae_2 \ar[r]^-{(R_{u_1}, R_{u_2})} & Ae \ar[r] & S\ar[r] & 0,}\vspace{-2pt}$$ where $R_{u_i}$ is the right multiplication by $u_i$, and a minimal projective presentation \vspace{-4pt}
$$\xymatrixcolsep{35pt}\xymatrix{ eA \ar[r]^-{(L_{u_1}\hspace{-1.5pt}, \,L_{u_2})^T\hspace{-5pt}} & e_1A\oplus e_2A \ar[r] & {\rm Tr}S\ar[r] & 0,}\vspace{-3.5pt}$$ where $L_{u_i}$ is the left multiplication by $u_i$. Therefore, ${\rm Tr}S\cong (e_1A\oplus e_2A)/(u_1, u_2)A$ and $\ell(\ntop({\rm Tr} S))=\ell(\ntop(e_1A\oplus e_2A))=2.$ We claim that $(u_1, u_2)J=e_1J^2\oplus e_2J^2,$ or equivalently,  $(e_1J^2, 0)\subseteq (u_1, u_2)J$ and $(0, e_2J^2)\subseteq (u_1, u_2)J.$ 

Indeed we may assume that $e_1J^2\ne 0$ and $e_2J^2\ne 0$. Since $e_2A$ is biserial, $\{u_2\}$ extends to  a top-basis $\{u_2, v_2\}$ for $e_2J,$ where $v_2\in e_2Je_3$ with $e_3$ a primitive idempotent in $A$. Since $Ae$ is biserial, by the dual of Lemma \ref{4sa_bproj}, $v_2A$ is simple and $e_2J=u_2A\oplus v_2A.$ So, $e_2J^2=u_2J$.
Since $0< \ell(u_2J) < \ell(u_2A) < \ell(e_2J)\le 3$, we see that 
$u_2J$ is simple. Then $u_2J=u_2uA$, for some $u\in eJe_4,$ where $e_4$ is a primitive idempotent in $A$. Since $e_1A\not\cong e_2A$, by Lemma \ref{sbsa}(2), $u_1u=0$. This yields $(0, e_2J^2)=(0, u_2u)A=(u_1, u_2)uA\subseteq (u_1, u_2)J.$  
Next, since $e_1A$ is uniserial, so are $e_1J$ and $e_1J^2$. Thus, $e_1J=u_1A$ and $e_1J^2=u_1J$. Then $e_1J^2=u_1v A$, where $v\in eJe_5$ with $e_5$ a primitive idempotent in $A$. By Lemma \ref{sbsa}(1), $u_2v=0$. 
Hence, $(e_1J^2, 0)=(u_1v, 0)A=(u_1, u_2)vA\subseteq (u_1, u_2)J.$ This establishes our claim. 

(1) Since $e_1J^2\oplus e_2J^2=(u_1, u_2)J\subset (u_1, u_2)A \subset e_1A\oplus e_2A$, we obtain a non-splitting short exact sequence \vspace{-7pt}
$$\hspace{-13pt}(\dagger) \hspace{5pt} \xymatrixcolsep{22pt}\xymatrix{0\ar[r] & (u_1, u_2)A/(u_1, u_2)J \ar[r]^-{(q_1, q_2)^T\hspace{-5pt}} & e_1A/e_1J^2 \oplus  e_2A/e_2J^2 \ar[r]^-{(q_3, q_4)} & {\rm Tr}S \ar[r] & 0,}\vspace{-3pt}$$ where $(u_1, u_2)A/(u_1, u_2)J\cong eA/eJ\cong DS;$ see (\ref{loc_gen}). Since ${\rm Tr} S \cong {\rm Tr} D(DS),$ by Proposition \ref{ass-sim}(1), $(\dagger)$ is an almost split sequence. Since $e_1A$ is uniserial, $e_1A/e_1J^2$ is uniserial of length two. And since $e_2A$ is biserial, $e_2A/e_2J^2$ is astride biserial of length three; see (\ref{bsm_top2}). By Proposition \ref{N_mod_ses}(1), ${\rm Tr}S$ is N-shaped, and so is $D{\rm Tr} S=:N$. Applying the duality $D$ to $(\dagger)$ yields an almost split sequence \vspace{-3pt}
$$\hspace{-68pt} (*) \hspace{6pt} \xymatrixcolsep{21pt}\xymatrix{0\ar[r] & N\ar[r] & D(e_1A/e_1J^2) \oplus D(e_2A/e_2J^2) \ar[r] & S \ar[r] & 0.} \vspace{-3pt}$$

Now, by Proposition \ref{duality_radi_soci}(3), $D(e_2A/e_2J^2)\cong \soc^2(D(e_2A))\cong \soc^2(I_{{\rm top}M});$ and $D(e_1A/e_1J^2)\cong \soc^2(D(e_1A))=\soc^2 (I_{S_1}),$ which is uniserial of length two. 
On the other hand, since $\rad(P/M)=(S_1\oplus M)/M\cong S_1,$ we see that $P/M$ is uniserial of length two with $\soc(P/M)\cong S_1.$ By Proposition \ref{uni_pi}(2), $P/M\cong D(e_1A/e_1J^2).$ So,  $(*)$ is an almost split sequence as stated in Statement (1) of the proposition. 

(2) As seen above, $e_2J=u_2A\oplus v_2A,$ where $v_2A\cong e_3A/e_3J.$ 
Since $\ell(e_2J)\le 3$, by Lemma \ref{loc_gen}(1), $u_2A$ is uniserial of length at most two with $\ntop(u_2A)\cong eA/eJ$. Thus, $\ntop(e_2J)\cong eA/eJ \oplus e_3A/e_3J.$ Since $Ae$ is biserial, by Definition \ref{4str_alg_def}(2), $Ae_3$ is uniserial. Thus $Ae/J^2e$ is astride biserial of length three, see (\ref{bsm_top2}), and $Ae_3/J^2e_3$ is uniserial of length two. Note that $\ntop(e_2A)\cong D(\ntop M)$. Similar to the almost split sequence $(\dagger)$, we obtain an almost split sequence \vspace{-6pt} 
$$(\dagger\dagger) \hspace{5pt} \xymatrixcolsep{21pt}\xymatrix{0\ar[r] & A(v_2, u_2)/\! J(v_2, u_2) \ar[r]^-{(f_3, f_4)^T\hspace{-4pt}} & Ae/\! J^2e \! \oplus \! Ae_3/J^2e_3 \ar[r]^-{(f_1, f_2)} & {\rm Tr}D(\ntop M) \ar[r] & 0,}\vspace{-5pt}$$ where $Ae/J^2e=\ntop^2\hspace{-1pt}P$ and $A(v_2, u_2)/J(v_2, u_2)\cong Ae_2/Je_2=\ntop M;$ see (\ref{loc_gen}). 

Observing that $f_1, f_2$ are monomorphisms, we see that $N_1\!:\hp ={\rm Im}(f_1)\cong Ae/J^2e$ and $L\!:\hp ={\rm Im}(f_2)\cong Ae_3/J^2e_3$ such that ${\rm Tr}D(\ntop M)=N_1+L$ and $N_1\cap L\cong \ntop M.$ Applying Proposition \ref{ass_Nmod}(2) to the N-shaped module ${\rm Tr}D(\ntop M)$, we see that \vspace{-2.2pt}
$${\rm Tr}D(\ntop M)\cong D{\rm Tr}(\ntop N_1)\cong D{\rm Tr}(\ntop(Ae/J^2e))= D{\rm Tr}S=N.\vspace{-3.5pt}$$ Thus, $(\dagger\dagger)$ is an almost split sequence as stated in Statement (2) of the proposition.

(3) As seen above, we may assume that $N=N_1+L$, where $N_1$ is astride biserial of length three and $L$ is uniserial of length two. \vspace{-1pt} By Proposition \ref{ass_Nmod}(3), the canonical short exact sequence $\xymatrixcolsep{20pt}\xymatrix{0\ar[r] & L\ar[r] & N\ar[r] & N/ L\ar[r] & 0} \vspace{-2pt}$ is almost split. Since $\ell(N/L)=2$ and $\ell(\soc^2\hspace{-.5pt}(I_{\ntop M}))=3$, by viewing the almost split sequence (1) stated in the proposition, we infer that $N/L\cong P/M$. Thus, this canonical sequence is an almost split sequence as stated in Statement (3) of the proposition.

(4) Since $\ntop^2\hspace{-1.2pt}P$ is not colocal, it is not injective. Viewing the almost split sequence (2) stated in the proposition, we have an irreducible map $q: \ntop^2\hspace{-1.2pt} P\to N.$ Since $\ell(\ntop^2P)=3$ and $\ell(L)=2,$ ${\rm Tr}D(\ntop^2P)\not\cong {\rm Tr} L\cong P/M.$ In view of the almost split sequence (1) stated in the proposition, ${\rm Tr}D(\ntop^2P)\cong \soc^2(I_{\ntop M}).$
We claim that there exists an irrducible map $p: \ntop^2\hspace{-1.2pt} P \to P/(S_1\oplus \rad M).$ In case $\ell(M)=1$, we have $\ntop^2\hspace{-1.2pt}P=P$ and $P/(S_1\oplus \rad M)=P/\soc S_1$. Since $I_{S_1}$ is uniserial as seen above, the claim follows from Proposition \ref{ass_st_mod}. Suppose that $\ell(M)=2.$ Then $\soc M = \rad M = \rad^2\hspace{-1.2pt}P.$ Thus $\ntop^2 \! = \! P/\soc M$ and $S_1\oplus \rad M \! = \! \soc(\rad P) \! = \! \soc P.$ 
Since  
$I_{S_1}$ and $I_{\soc M}$ are uniserial by Definition \ref{4str_alg_def}(3), the claim follows from Proposition \ref{ass_p_astride}. In any case, our claim holds. Now since $N\not\cong P/(S_1\oplus \rad M)$, we obtain an irreducible map $(q, p)^T: \ntop^2\hspace{-1.2pt} P \to N\oplus P/(S_1\oplus \rad M);$ see \cite{Bau}. Therefore, there exists an almost split sequence \vspace{-5pt}
$$\xymatrixcolsep{26pt}\xymatrix{0\ar[r] & \ntop^2\hspace{-1.2pt}P \ar[r]^-{(q, p, f)^T\hspace{-6pt}} & N\oplus P/(S_1\oplus \rad M) \oplus V \ar[r]^-{(p_1, q_1, g)} & \soc^2(I_{\ntop M}) \ar[r] & 0.}\vspace{-4pt}$$ Noting that $\ell(P/(S_1\oplus \rad M))=2$ and $\ell(N)=4$, we conclude that $V=0$. Thus, the above sequence is an almost split sequence as stated in Statement (4) of the proposition. The proof of the proposition is completed.

\smallskip

Dually we obtain the following result.

\vspace{-1pt}

\begin{Prop}\label{ass_14_2} 

Let $A$ be a quadri-string algebra with radical cubed zero. Consider a biserial injective module $I$ in $\mmod A$ with $I/\soc I=M_1/\soc I\oplus M_2/\soc I,$ where $M_1/\soc I$ is simple and $M_2/\soc I$ is uniserial of length one or two such that $P_{\soc (M_2/\soc I)}$ is biserial. Then we have four almost split sequences as follows$\hp:$ 

\begin{enumerate}[$(1)$]

\vspace{-2.5pt}

\item $\xymatrixcolsep{20pt}\xymatrix{0\ar[r] & \soc I \ar[r] & 
\ntop^2\hspace{-.5pt}(P_{\soc (M_2/\soc I)})\oplus M_1 \ar[r] & N \ar[r] & 0\hp ;}$ 

\vspace{-2.5pt}

\item $\xymatrixcolsep{20pt}\xymatrix{0\ar[r] & N  \ar[r] & \soc^2\hspace{-1pt}I \oplus L \ar[r] & \soc(M_2/\soc I) \ar[r] & 0\hp ;}$

\vspace{-2.5pt}

\item $\xymatrixcolsep{20pt}\xymatrix{0\ar[r] & M_1  \ar[r] & N \ar[r] & L \ar[r] & 0\hp ;}$

\vspace{-2.5pt}

\item $\xymatrixcolsep{20pt}\xymatrix{0\ar[r] & \ntop^2(P_{\soc(M_2/\soc I)}) \ar[r] & N\oplus \soc^2\hspace{-1pt} M_2 \ar[r] & \soc^2\hspace{-1pt}I \ar[r] & 0\hp ;}$

\end{enumerate}

\vspace{-4pt}

\noindent where $N$ is an N-shaped module and $L$ is a uniserial module of length two.

\end{Prop}

\vspace{-1.5pt}

\noindent{\it Proof.} Observe that $I=M_1+M_2$ and $\soc I=M_1\cap M_2$, where $M_1$ and $M_2$ are uniserial with $\ell(M_2)=2$ and $2\le \ell(M_2)\le 3.$ In view of Lemma \ref{ses_sum_itsec}(1), we have
$\ell(I)=\ell(M_1)+\ell(M_2)-\ell(M_1\cap M_2)=\ell(M_2)+1.$ 

Now, consider the biserial projective module $DI$ in $\mmod A^\circ.$ By Proposition \ref{duality_radi_soci}(1) and Lemma \ref{dual_perp}, $\rad(DI)=(\soc I)^\perp=(M_1\cap M_2)^\perp=M_2^\perp \oplus M_1^\perp,$ where $M_2^\perp\cong D(I/M_2),$ which is simple; $M_1^\perp\cong D(I/M_1),$ which is uniserial of length one or two. Put $S:=\soc(M_2/\soc I)$. Since $I/M_1\cong M_2/\soc I$, we see from Proposition \ref{duality_radi_soci}(2) that $\ntop(M_1^\perp) \cong \ntop(D(I/M_1))\cong D(\soc(I/M_1))\cong DS.$ Thus \vspace{1pt} $I_{\ntop(M_1^\perp)}\cong I_{DS}= D(P_S),$ which is biserial by the hypothesis. THus we obtain from Proposition \ref{4str_ass_smr} the following almost split sequences\hp:

\vspace{2pt}

$\hspace{-10pt} (1') \hspace{5pt}  \xymatrixcolsep{20pt}\xymatrix{0\ar[r] & N'\ar[r] & 
\soc^2(I_{DS})\oplus DI/M_1^\perp \ar[r] & \ntop(DI) \ar[r] & 0;}$ 

\vspace{-2.5pt}

$\hspace{-10pt}(2') \hspace{5pt} \xymatrixcolsep{20pt} \xymatrix{0\ar[r] & DS \ar[r] & \ntop^2\hspace{-.5pt}(DI) \oplus L'\ar[r] & N'\ar[r] & 0;}$

\vspace{-2.5pt}

$\hspace{-10pt}(3') \hspace{5pt} \xymatrixcolsep{18pt}\xymatrix{0\ar[r] & L' \ar[r] & N' \ar[r] & (DI)/M_1^\perp \ar[r] & 0;}$

\vspace{-2.5pt}

$\hspace{-10pt}(4') \hspace{5pt} \xymatrixcolsep{18pt}\xymatrix{0\ar[r] & \ntop^2(DI)  \ar[r] & N' \oplus (DI)/(M_2^\perp+\rad M_1^\perp)\ar[r] & \soc^2(I_{DS})\ar[r] & 0,}$

\noindent where $N'$ is an N-shaped module, and $L'$ is uniserial of length two. 
Write $L=DL'$, which is uniserial of length two, and $N=DN',$ which is N-shaped; see (\ref{N_mod_dual}). 

Now by Proposition \ref{duality_radi_soci}, $\soc^2(I_{DS})=\soc^2(D(P_S))\cong D(\ntop^2(P_{\soc(M_2/\soc I)}))$ and  $\ntop^2(DI) \cong D(\soc^2I).$ Moreover by Lemma \ref{dual_perp}(1), $DI/M_1^\perp\cong DM_1.$ Thus, the duals of the sequences (1'), (2') and (3') are almost split sequences as stated in Statements (1), (2) and (3) of the proposition. 

Finally since $M_1\cap M_2=\soc M_2,$ we have an isomorphism $\varphi: I/M_1\to M_2/\soc M_2,$ sending $(m_1+m_2)+M_1$ to $m_2+\soc M_2.$ And since $\soc(M_2/\soc M_2)=\soc^2M_2/\soc M_2,$ we see that \vp $\soc(I/M_1)=(\soc^2\hspace{-1.2pt}M_2+M_1)/M_1.$ 
Thus, $\rad(M_1^\perp)=(\soc^2\hspace{-1.2pt}M_2+M_1)^\perp$ by Lemma \ref{rad_perp},. Since $M_2 \cap (\soc^2\hspace{-1.2pt}M_2+M_1)=\soc^2\hspace{-1pt}M_2+M_1\cap M_2=\soc^2\hspace{-1.2pt}M_2,$ we deduce from Lemma \ref{dual_perp} \vp that $M_2^\perp+\rad(M_1^\perp) 
= (M_2 \cap (\soc^2\hspace{-1.2pt}M_2+M_1))^\perp = (\soc^2\hspace{-1.2pt}M_2)^\perp,$ 
and hence, $(DI)/(M_2^\perp+\rad(M_1^\perp)) = (DI)/(\soc^2\hspace{-1.2pt}M_2)^\perp \cong D(\soc^2\hspace{-1.2pt}M_2).$ \vp Furthermore, $\soc^2(I_{DS}) \cong D(\ntop^2(P_S))=D(\ntop^2(P_{\soc(M_2/\soc I)}))$ and $\ntop^2(DI) \cong D(\soc^2I).$ Thus applying the duality $D$ to (4') yields an almost split sequence as stated in Statement (4) of the proposition. The proof of the proposition is completed.

\subsection{\sc Representation bound} The main objective of this subsection is to show that quadri-biserial algebras have a representation bound of at most four. Our strategy is to reduce the problem to quadri-string algebras with radical cubed zero, and show that the class of indecomposable modules of length at most four is closed under irreducible maps in this case. To this end, we apply previous results on almost split sequences to analyze the minimal right almost split morphisms for modules of length up to four case-by-case.

\vspace{-1pt}

\begin{Lemma}\label{4sa_rasm_l1}

Let $A$ be a quadri-string algebra with radical cubed zero. Consider a nonzero minimal right almost split morphism $f: K\to S$ in $\mmod A,$ where $S$ is simple. Then $K$ is a direct sum of at most two indecomposable modules of length at most four.

\end{Lemma}

\vspace{-1.5pt}

\noindent{\it Proof.} By the hypothesis, $S$ admits a nonsimple projective cover $P.$ We start with the case where $P$ is biserial. Since $A$ is a string algebra, $\rad P=S_1\oplus M$, where $S_1$ is simple and $M$ is uniserial of length one or two. First, suppose that $I_{\ntop M}$ is biserial. By Proposition \ref{4str_ass_smr}(1), $K\cong \soc^2(I_{\ntop M})\oplus P/M,$ where  $P/M$ is uniserial of length two and $\soc^2(I_{\ntop M})$ is colocal of length three; see (\ref{bsm_top2}). %; see (\ref{astride_bsm}). 

Next, suppose that $I_{\ntop M}$ is uniserial. If $\ell(M)=2,$ then $K\cong P/\soc P \oplus P/\soc M$ by Proposition \ref{ass_p4_1}(2), where $P/\soc P$ and $P/M$ are uniserial of length two. Otherwise, $M$ is simple and $I_M$ is uniserial. Then $\soc P=\rad P,$ and $P/\soc P=S$. If $I_{S_1}$ is also uniserial then, by Proposition \ref{ass_p_astride}, $K\cong P/S_1\oplus P/M,$ where $P/S_1$ and $P/M$ are uniserial of length two. If $I_{S_1}$ is biserial then, by Proposition \ref{4str_ass_smr}(1), $K\cong \soc^2(I_{S_1}) \oplus P/S_1,$ where  $P/S_1$ is uniserial of length two, and $\soc^2(I_{S_1})$ is colocal of length three; see (\ref{bsm_top2}). %; see (\ref{astride_bsm}). 
This establishes the lemma in case $P$ is biserial.

Let us consider the second case where $P$ is uniserial. Write $\ntop (\rad P)=S_2.$ First, suppose that $I_{S_2}$ is uniserial. Set $U\!\!:\hp\hp =\ntop^2P,$ which is uniserial of length two. Then $\soc \hp U=\rad \hp U=S_2$ and $U/\soc U=\ntop \,U\cong \ntop P=S.$  Note that $P_{\ntop \, U}\cong P$ and $I_{\soc \, U}\cong I_{S_2}$, both are uniserial. And by Theorem \ref{ass_N_alg}, $K\cong U.$ 

Next, suppose that $I_{S_2}$ is biserial, which is co-astride biserial. By Proposition \ref{astride_bsm} and Definition \ref{4str_alg_def}(1), $I_{S_2}/S_2=S_3\oplus L,$ where $S_3$ is simple and $L$ is uniserial of length one or two. Since $S_2=\ntop (\rad P_S)$, it is well known that $S$ is a direct summand of $\soc(I_{S_2}/S_2)$; see \cite[(III.1.15)]{ARS} and also \cite[(1.1)]{LiY}. Thus, $S\cong S_3$ or $S\cong \soc L.$ Assume that $S\cong S_3$. Then, $I_{S_2}/S_2\cong S\oplus L.$ Since $P_S=P,$ which is uniserial, we deduce from Proposition \ref{ass_st_mod_1} that $K\cong I_{S_2},$ which is indecomposable of length at most four.

Finally assume that $S \cong \soc L.$ If $L$ is simple, then $I_{S_2}/S_2=S_3\oplus S;$ as in the previous case, $K\cong I_{S_2}.$ Suppose that $\ell(L)=2$. Since $P_{\soc L}=P,$ which is uniserial, we deduce from Proposition \ref{ass_i4_1}(1) that $K\cong \soc^2(I_{S_2}),$ which is indecomposable of length three; see (\ref{bsa_ml3}). The proof of the lemma is completed.

%\smallskip

\vspace{2pt}

\begin{Lemma}\label{4sa_rasm_l2}

Let $A$ be a quadri-string algebra with radical cubed zero. Consider a minimal right almost split morphism $f: K\to M$ in $\mmod A,$ where $M$ is of length two. Then $K$ is a direct sum of at most two indecomposable modules of length at most four.

\end{Lemma}

\vspace{-1.5pt}

\noindent{\it Proof.} Set $S_0:=\ntop M$ and $S_1:=\rad M,$ both are simple. 
Write $P=P_{S_0}.$ Then $2\le \ell(P)\le 4$. If $M$ is projective, then $K\cong \rad P,$ which is simple. \vspace{-1pt} Otherwise, we have a canonical short exact sequence
$\xymatrixcolsep{20pt}\xymatrix{\hspace{-2pt} 0\ar[r] & L \ar[r] & P\ar[r] & M\ar[r] & 0,}$ \vspace{-1pt} where $L$ is a submodule of $P$ of length one or two such that $M\cong P/L.$

Consider the first case where $P$ is uniserial. Since $J^3=0$ by the hypothesis, $\ell(P)=3$ and $\ell(L)=1$. Then $L=\soc P,$ and consequently $P/\soc P \cong M$ and $\rad P/\soc P\cong \rad M=S_1.$ 
Suppose that $I_{\soc P}$ is uniserial. Since $\rad P$ is indecomposable and $P/\soc(\rad P)=P/\soc P\cong M,$ see (\ref{soc_rad}), we deduce from Proposition \ref{ass_st_mod} that $K\cong P\oplus \rad P/\soc(\rad P)$, where $\rad P/\soc(\rad P) = \rad P/\soc P \cong S_1.$ 

Suppose next that $I_{\soc P}$ is biserial. Put $I:=I_{\soc P} \not\cong P.$ We may assume that $P$ is a proper submodule of $I$ with $\soc P=\soc I$. Since $A$ is a quadri-string algebra, $I/\soc I=S\oplus U,$ where $S$ is simple and $U$ is uniserial of length two. Since $P/\soc I$ is uniserial of length two, by Lemma \ref{dec_mod_l3}(1), $I/\soc I=S\oplus (P/\soc I)$. Note that $P_{\ntop(P/\soc I)}=P,$ which is uniserial. Since $P/\soc I=P/\soc P\cong M,$ by
Proposition \ref{ass_st_mod_1}, $K\cong I\oplus \rad(P/\soc I),$ where $ \rad(P/\soc I)=P/\soc P\cong S_1.$

Consider now the second case where $P$ is biserial. Then, $\rad P=T\oplus V$, where $T$ is simple and $V$ is uniserial of length one or two. 
%In particular, $\soc P=T\oplus \soc\, V.$ 
Note that $P/T$ is uniserial of length two or three, and $\ntop^2\hspace{-1pt}P$ is astride biserial of length three; see (\ref{bsm_top2}).

Suppose that $L$ is simple. Then $L=\soc L=\ntop L,$ and $M\cong P/L=P/\soc L.$ Moreover, $\ell(P)=3$ and $\ell(\rad P)=2$. So, we may assume that $\rad P=T\oplus L$. If $I_L$ is uniserial, then $K\cong P$ by Theorem \ref{ass_st_mod}. If $I_L$ is biserial, then $K$ is an N-shaped module by Proposition \ref{4str_ass_smr}(3). 

Suppose that $\ell(L)=2.$ Then $\ell(P)=4$ and $\ell(V)=2.$ %So, $\soc V=\rad^2\hspace{-1pt}P$.
Thus by Definition \ref{4str_alg_def}(3), $I_T$ and $I_{\soc V}$ are uniserial. If $L$ is decomposable, then $L=\soc P,$ and hence, $M\cong P/L= P/\soc P.$ Now by Proposition \ref{ass_p_astride}, $K\cong P/\soc\, T \oplus P/\soc L,$ where $P/\soc\, T=P/T,$ and $P/\soc L=P/\soc P,$ which is uniserial of length two.

Finally assume that $L$ is indecomposable, which is uniserial of length two. By Lemma \ref{dec_mod_l3}(1), $\rad P=T\oplus L$. By Definition \ref{4str_alg_def}(3), $I_T$ and $I_{\soc L}$ are uniserial. Recall that $M\cong P/L.$
If $I_{\ntop L}$ is uniserial then, by Proposition \ref{ass_p4_1}(1), $K\cong \ntop^2\hspace{-1pt} P.$ If $I_{\ntop L}$ is biserial, then $K$ is an N-shaped module by Proposition \ref{4str_ass_smr}(3). The proof of the lemma is completed.

%\smallskip

\vspace{2.5pt}

\begin{Lemma}\label{4sa_rasm_l3}

Let $A$ be a quadri-string algebra with radical cubed zero. Consider a minimal right almost split morphism $f: K\to M$ in $\mmod A,$ where $M$ is of length three. Then $K$ is a direct sum of at most two indecomposable modules of length at most four.

\end{Lemma}

\noindent{\it Proof.} We start with the case where $M$ is local module. If $M$ is projective, then $K\cong \rad M,$ which is uniserial of length two or a direct sum of two simple modules. Otherwise, we have a short exact sequence 
$\xymatrixcolsep{20pt}\xymatrix{\hspace{-3pt} 0\ar[r] & S\ar[r] & P \ar[r] & M\ar[r] & 0,}\vspace{-.5pt}$ where $P$ is an indecomposable projective module of length four and $S$ is a simple sub\-module of $\rad P$. Then $\rad P=T\oplus L$, where $T$ is simple and $L$ is uniserial of length two. By Definition \ref{4str_alg_def}(3), $I_{\soc L}$ is uniserial. And by Lemma \ref{dec_mod_l3}(2), $S=\soc L$ or $\rad P=S\oplus L.$ If $S=\soc L,$ then $M\cong P/S= P/\soc L,$ and $K\cong  P\oplus (L/\soc\hp L)$ by Proposition \ref{ass_st_mod}, where $L/\soc \,L$ is simple. If $\rad P=S\oplus L,$ then $I_S$ is uniserial by Definition \ref{4str_alg_def}(3). Since $M\cong P/\soc S$, by Proposition \ref{ass_st_mod}, $K\cong P.$

Now, consider the case where $M$ is not local. By Proposition \ref{CL_3}, $M$ is co-astride biserial of length three. Write $I=I_{\soc M}$. Then $I$ is biserial and $M\cong \soc^2 \hspace{-1.2pt} I;$ see (\ref{bsa_ml3}). Thus, $I/ \soc I= (M_1/\soc I) \oplus (M_2/\soc I),$ where $M_1$ is uniserial of length two and $M_2$ is uniserial of length two or three. Then $I=M_1+M_2$ and $M_1\cap M_2=\soc I.$
%$=\soc M_1=\soc M_2.$ 
If $P_{\hp \soc (M_2/\soc I)}$ is biserial, then
$K\cong N\oplus \soc^2(M_2)$ by Proposition \ref{ass_14_2}(4), where $N$ is N-shaped and $\soc^2(M_2)$ is uniserial of length two. 

Suppose that $P_{\hp \soc(M_2/\soc I)}$ is uniserial. If $\ell(M_2)=3,$ then $K\cong M_1 \oplus \rad  M_2$ by Proposition \ref{ass_i4_1}(2), where $M_1$ and $\rad M_2$ are uniserial of length two. Otherwise, both $M_1/\soc I$ and $M_2/\soc I$ are simple; hence, $I=\soc^2\hspace{-1pt}I=M$ and $P_{\hp M_2/\soc I}$ is uniserial. Moreover, $\soc I=M_1\cap M_2=\rad M_1=\rad M_2.$ Thus, $\rad M_1+M_2=M_2$ and $M_1+\rad M_2=M_1$. If $P_{\hp M_1/\soc I}$ is also uniserial, \vp then $K\cong M_2\oplus M_1$ by Proposition \ref{ass_endi}. 
Otherwise, $K\cong N\oplus M_1$ by Proposition \ref{ass_14_2}(4), where $N$ is N-shaped
and $\soc^2(M_1)=M_1.$ The proof of the lemma is completed.

%\smallskip

\vspace{2pt}

\begin{Lemma}\label{4sa_rasm_l4}

Let $A$ be a quadri-string algebra with radical cubed zero. Consider a minimal right almost split morphism $f: K\to M$ in $\mmod A,$ where $M$ is of length four. Then $K$ is a direct sum of at most two indecomposable modules of length at most three.

\end{Lemma}

\vspace{-1.5pt}

\noindent{\it Proof.} In case $M$ is local, by Lemma \ref{maxi_local}(1), $M$ is a projective module, which is uniserial or astride-biserial, Thus, $K\cong \rad P$, which is uniserial of length three, or a direct sum of two uniserial modules of length one and two. 

Suppose that $M$ is colocal and not local. By Theorem \ref{l4-4bsa}(4), $M$ is a co-astride injective module with $M/\soc M=(M_1/\soc M)\oplus(M_2/\soc M)$, where $M_1$ and $M_2$ are uniserial of length two and three, respectively. Then  $M=M_1+M_2$ and $M_1\cap M_2=\soc M_1=\soc M_2$. Thus, $\rad M_1+ M_2=\soc M_1+M_2=M_2;$ and by Lemma \ref{socs_sum}, $M_1+\rad M_2=M_1+\soc^2\hspace{-.8pt} M_2=\soc^2\hspace{-1.2pt}M.$ Then by Proposition \ref{ass_endi}, $K\cong M_2\oplus \soc^2\hspace{-1.5pt}M,$ where $\soc^2\hspace{-1.5pt}M$ is co-astride biserial of length three; see (\ref{socs_sum}).

Finally, suppose that $M$ is neither local nor colocal. By Theorem \ref{l4-4bsa}(1), $M$ is N-shaped. Write $M=M_1+M_2$, where $M_1$ is astride  biserial of length three and $M_2$ is uniserial of length two. Then by Proposition \ref{ass_Nmod}, $K\cong M_1 \oplus M_2.$ The proof of the lemma is completed.

%\smallskip

\vspace{4pt}

We are ready to obtain the main result of this subsection.

\begin{Theo}\label{rb_4bsa}

Let $A$ be a quadri-biserial algebra. Then, every indecomposable module in $\mmod A$ is of length at most four.

\end{Theo}

\noindent{\it Proof.} We may assume that $A$ is connected and nonsimple. By Definition \ref{4str_alg_def}(1), the projective-injective modules in $\mmod A$ are of length at most four. In view of by Theorem \ref{l4-4bsa}, we may further assume that $A$ is a quadri-string algebra with radical cubed zero. Denote by $\mathcal F$ the set of modules of length at most four in the Auslander-Reiten quiver $\Ga_{\mmod A}$ of $\mmod A$. In view of Lemmas \ref{4sa_rasm_l1}, \ref{4sa_rasm_l2}, \ref{4sa_rasm_l3} and \ref{4sa_rasm_l4}, we see that $\mathcal F$ is closed under predecessors in $\Ga_{\mmod A}$. Since $A^\circ$ is also a quadri-string algebra with radical cubed zero, the set of modules of length at most four in $\Ga_{\mmod A^\circ}$ is also closed under predecessors. Applying the duality $D$, we see that $\mathcal F$ is closed under successors. Let $\mathcal C$ be a connecetd component of $\Ga_{\mmod A}$ containg a simple module $S$. Since $S\in \mathcal{F}$, we conclude that $\mathcal{C}\subseteq \mathcal{F},$ namely, all modules in $\mathcal C$ are of length at most four. It is then well known that $\Ga_{\mmod A}=\mathcal C;$ see, for example, \cite[(VI.1.4)]{ARS}. That is, every indecompoable module in $\mmod A$ is of length at most four. The proof of the theorem is completed. 

\begin{Remark}

In view of Proposition \ref{CL_3} and Theorems \ref{rb_4bsa} and \ref{l4-4bsa}, we obtain an explicit description of all indecomposable modules over a quadri-biserial algebra. In particular, they are all uniserial, biserial or N-shaped.
    
\end{Remark}

%\smallskip

It is well known that an artin algebra is biserial if every almost split sequence has at most two indecomposable nonprojective direct summands in its middle term; see \cite{AuR}. The converse holds for special biserial algebras given by a quiver with relations; see \cite{BuR, SKW}; as shown below, this converse also holds for quadri-biserial algebras.

\begin{Theo}\label{beta_A} 

Let $A$ be a quadri-biserial algebra. Then, every almost split sequence in $\mmod A$ has at most three indecomposable direct summands in its middle term$\hp ;$ and those with three indecomposable direct summands are of the form$\hp\hp:$ \vspace{-4pt}
$$\xymatrix{0\ar[r] & \rad P \ar[r] & P  \oplus S_1 \oplus S_2  
\ar[r] & P/\soc P \ar[r] & 0,} \vspace{-5pt}$$ where $P$ is lozenge, and $S_1, S_2$ are simple such that $S_1\oplus S_2  \!\cong\! \rad P/\soc P.$

\end{Theo}

 \vspace{-1.5pt}

\noindent{\it Proof.} We may assume that $A$ is connected and nonsimple. First, let $P$ be a lozenge module in $\mmod A$ with $\rad P/\soc P=S_1 \oplus S_2$, where $S_1, S_2$ are simple. By Lemma \ref{maxi_local}(1), $P$ is projective-injective. It is well known that there exists an almost split sequence in $\mmod A$ as stated in the theorem; see \cite[(V.5.5)]{ARS}. 

\vspace{-.5pt} 

Now, consider an almost split sequence 
$(*) \hspace{5pt} \xymatrixcolsep{20pt}\xymatrix{0\ar[r] & L \ar[r] & M  \ar[r] & N \ar[r] & 0} \vspace{-.5pt} $ in $\mmod A.$ Write $M=M_1 \oplus \cdots \oplus M_s,$ where the $M_i$ are indecomposable. Suppose that one of the $M_i$, say $M_1,$ is projective-injective. By Proposition 5.5 in \cite[Chapter V]{ARS}, $M\cong M_1\oplus (\rad M_1/\soc M_1)$ with $2\le \ell(M)\le 4;$ see (\ref{rb_4bsa}). If $M_1$ is uniserial, then $\rad M_1/\soc M_1$ is zero or indecomposable, and hence, $1\le s\le 2.$ Otherwise,  by Corollary \ref{rb4_lozenge}, $M_1$ is a lozenge module; and consequently, $(*)$ is an almost split sequence as stated in the theorem. 

Next, suppose that none of the $M_i$ with $1\le i\le s$ is projective-injective. Note that  neither $L$ nor $N$ is projective-injective. By Proposition \ref{4ba-4sa}, $(*)$ is an almost split sequence in $\mmod \hspace{3pt}\overline{\hspace{-2.8pt}A},$ where $\hspace{2pt}\overline{\hspace{-2.8pt}A}$ is a quadri-string algebra with radical cubed zero.  In view of Lemmas \ref{4sa_rasm_l1}, \ref{4sa_rasm_l2}, \ref{4sa_rasm_l3} and \ref{4sa_rasm_l4}, we conclude that that $s\le 2.$ The proof of the theorem is completed.

\smallskip

\begin{Remark} 

In view of Theorem \ref{beta_A}, Propositions \ref{ass_p_astride}, \ref{ass_endi}, \ref{4ba-4sa}, and the results in Subsection 4.4, we obtain an explicit description of all almost split sequences over a quadri-biserial algebra.

\end{Remark}

\section{\sc Algebras of representation bound up to four}

The objective of this section is to present a complete classification for representa\-tion-finite artin algebras of representation bound up to four. It is evident that an algebra is of representation bound one if and only if it is semisimple. We shall show that the algebras of representation bound two are Nakayama algebras of Loewy length two. Furthermore, those of representation bound three are wedged-string algebras with radical cubed zero, excluding the Nakayama algebras with radical squared zero. Finally, the algebras of representation bound four are quadri-biserial algebras which are not wedged-string algebras with radical cubed zero.

\subsection{\sc Length of Auslander-Reiten translates} In this subsection, we provide a general method to study the representation bound of artin algebras, by establishing lower bounds for the lengths of the Auslander-Reiten translates of certain modules. Let us start with simple modules.

\begin{Prop}\label{S_ART_length} 

Let $A$ be an artin algebra. Consider $S=Ae/Je,$ where $e$ is a primitive idempotent in $A$.

\begin{enumerate}[$(1)$]

\vspace{-1.2pt}

\item If $Je/J^2e\cong \oplus_{i=1}^s Ae_i/Je_i,$ \vspace{1.3pt} where the $e_i$ are primitive idempotents in $A$, then $\ell(D{\rm Tr}S) \ge (s-1)+\sum_{i=1}^s \ell(e_iJ/e_iJ^2)\ge 2s-1.$

\vspace{1.5pt}

\item If $eJ/eJ^2 \cong \oplus_{i=1}^t e_iA/e_iJ$, \vspace{1.3pt} where the $e_i$ are primitive idempotents in $A$, then $\ell({\rm Tr}D S) \ge (t-1) + \sum_{i=1}^t\ell(Je_i/J^2e_i)\ge 2t-1.$

\end{enumerate}\end{Prop}

\noindent{\it Proof.} (1) Assume that $Je/J^2e\cong \oplus_{i=1}^s Ae_i/Je_i$, \vspace{.5pt} where the $e_i$ are primitive idempotents in $A$. By Lemma \ref{find_top-basis}, $Je$ admits a top-basis 
$\{u_1, \ldots, u_s\}$ with $u_i\in e_i J e.$ 
%In particular, $\ell(e_iJ/e_iJ^2)\ge 1$, for $i=1, \ldots, s$.
By Proposition \ref{tb-projc-ienv}, we have a minimal projective presentation
\vspace{-4.5pt}
$$\xymatrixcolsep{18pt}\xymatrix{Ae_1\oplus \cdots \oplus Ae_s \ar[rr]^--{(R_{u_1}, \ldots, R_{u_s})} && Ae \ar[r] & S \ar[r]  & 0,}
\vspace{-4pt}$$
where $R_{u_i}$ is the right multiplication by $u_i,$ and a minimal projective presentation 
\vspace{-3pt}
$$\xymatrixcolsep{19pt}\xymatrix{e A \ar[rr]^--{(L_{u_1}, \ldots, L_{u_s})^T\hspace{-5pt}} && e_1A\oplus \cdots\oplus e_sA \ar[r] &  {\rm Tr}S\ar[r]  & 0,} \vspace{-1pt}$$ where $L_{u_i}$ is the left multiplication by $u_i$. This yields \vspace{-2pt}
$${\rm Tr}S  \cong  \left((e_1A\oplus \cdots\oplus e_sA)/(u_1,\ldots,u_s)J\right)/\left((u_1,\ldots,u_s)A/(u_1,\ldots,u_s)J\right).
\vspace{-2pt}$$
Since $(u_1,\ldots,u_s)\!=\!(u_1,\ldots,u_s)\hp e$, \vspace{1pt} we have $(u_1,\ldots,u_s)A / (u_1,\ldots,u_s)J \!\cong\! eA/eJ$ by Lemma \ref{loc_gen}(1); consequently, $\ell({\rm Tr}S) = \ell\left((e_1A\oplus \cdots\oplus e_sA)/(u_1,\ldots,u_s)J\right)-1.$ Further since $(u_1,\ldots,u_s)J\subseteq e_1J^2\oplus \cdots\oplus e_sJ^2\hspace{-1pt},$
we obtain \vspace{-2pt}
$$\begin{array}{rcl}\ell\left((e_1A\oplus \cdots\oplus e_sA)/(u_1,\ldots,u_s)J\right) \hspace{-5pt} &  \ge \hspace{-5pt} &\ell\left((e_1A\oplus \cdots\oplus e_sA)/(e_1J^2\oplus \cdots\oplus e_sJ^2)\right)\vspace{1.3pt} \\ 
\hspace{-5pt} &= \hspace{-5pt}& \ell(e_1A/e_1^2J) + \cdots + \ell(e_sA/e_sJ^2) \vspace{1.3pt} \\
\hspace{-5pt} &= \hspace{-5pt}& \ell(e_1J/e_1^2J) + \cdots + \ell(e_sJ/e_sJ^2)+s.
\vspace{-3.5pt}
\end{array}$$
Therefore,
$\ell(D{\rm Tr}S)=\ell({\rm Tr}S)\ge (s-1) + \sum_{i=1}^s\ell(e_iJ/e_iJ^2)\ge 2s-1.$ 

\vspace{1,5pt}

(2) Observe that $DS\cong eA/eJ$ in $\mmod A^\circ.$ Assume that $eJ/eJ^2 \cong \oplus_{i=1}^t e_iA/e_iJ$, where the $e_i$ are primitive idempotents in $A$. As argued above, \vp we conclude that ${\rm Tr}DS \!\ge\! (t\!-\!1) \hspace{-1pt} +\! \sum_{i=1}^s\hspace{-1pt}\ell(Je_i/J^2e_i) \!\ge\! 2t\!-\!1$. The proof of the proposition is completed.

\smallskip

Next, under certain conditions, we consider the socle factor of indecomposable projective modules and the radical of indecomposable injective modules.

\begin{Prop}\label{SF_ART_length} 

Let $A$ be an artin algebra. Consider an indecomposable projective module $P$ and an indecomposable injective module $I$ in $\mmod A$. 

\begin{enumerate}[$(1)$]

\vspace{-2pt}

\item Suppose that $\rad P=M_1\oplus \cdots \oplus M_s,$ where $M_i$ is of Loewy length $c_i$ such that $\soc M_i\cong Ae_i/Je_i$ with $e_i$ a primitive idempotent in $A$, for $i=1, \ldots, s$. \vspace{-3.5pt} Then 
$$\ell(\hspace{-.5pt} D{\rm Tr}\hspace{.5pt} (P/\soc P)) \ge \textstyle\sum_{i=1}^s  \ell(e_iA/e_iJ^{c_i+1}) - 1 \ge  (s-1) + \sum_{i=1}^s c_i\ge 2s-1.$$

\vspace{-2.5pt}

\item Suppose that $I/\soc I=N_1\oplus \cdots \oplus N_t,$ where $N_i$ is of Loewy length $c_i$ such that $\ntop N_i\cong Ae_i/Je_i$ with $e_i\in A$ a primitive idempotent, for $i=1, \ldots, t$. Then  \vspace{-4pt}
$$\ell({\rm Tr} D \hspace{.5pt} (\rad I)) \ge \textstyle \sum_{i=1}^t \ell(Ae_i/J^{c_i+1} e_i) - 1 \ge  (t-1) + \sum_{i=1}^tc_i\ge 2t-1.$$

\end{enumerate}\end{Prop}

 \vspace{-2pt}
 
\noindent{\it Proof.} (1) We may assume that $P=Ae$ with $e$ a primitive idempotent in $A$. Since $\soc M_i$ is simple, $\soc \hspace{.5pt} M_i=\rad^{c_i-1} \hspace{-.5pt} M_i\subseteq J^{c_i}e$, and hence, $\soc M_i=Au_i,$ for some $u_i \in e_i J^{c_i} e$, for $i=1, \ldots, s$. By Lemma \ref{soc_rad},  $\soc(Ae)=\soc(Je)=\oplus_{i=1}^s Au_i.$ Put $N:=Ae/\soc(Ae)$. Since $\{u_1, \ldots, u_s\}$ is clearly a top-basis for $\soc(Ae)$, we obtain a minimal projective presentation \vspace{-6pt}
$$\xymatrixcolsep{20pt}\xymatrix{Ae_1 \oplus\cdots \oplus Ae_s \ar[rr]^-{(R_{u_1}, \ldots, R_{u_s})} && Ae \ar[r] & N \ar[r] & 0,}\vspace{-2.5pt}$$ where $R_{u_i}$ is the right multiplication by $u_i,$ and a minimal projective presentation \vspace{-6pt}
$$\xymatrixcolsep{20pt}\xymatrix{eA \ar[rr]^-{(L_{u_1}, \ldots, L_{u_s})^T\hspace{-4pt}} && e_1A \oplus\cdots \oplus e_sA \ar[r] & {\rm Tr}N \ar[r]& 0,}\vspace{-3pt}$$ where $L_{u_i}$ is the left multiplication by $u_i$. 
Therefore, \vspace{-2pt}
$${\rm Tr}N \cong ((e_1A \oplus\cdots \oplus e_sA)/(u_1, \ldots, u_s)J)/((u_1, \ldots, u_s)A/(u_1, \ldots, u_s)J).\vspace{-4pt}$$
Since $(u_1,\ldots,u_s)\in Ae$, by Lemma \ref{loc_gen}(1), $(u_1,\ldots,u_s)A / (u_1,\ldots,u_s)J \cong eA/eJ.$ Since $0\ne u_i\in e_iJ^{c_i}$, we have $\ell\ell(e_iA)\ge c_i+1;$ and by Lemma \ref{topi_LL}, we infer that $\ell(e_iA/e_iJ^{c_i+1}) \ge \ell\ell(e_iA/e_iJ^{c_i+1})\ge c_i+1$, for $i=1, \dots, s.$ \vp Now, observing that $(u_1, \ldots, u_s)J\subseteq e_1J^{c_1+1} \oplus \cdots \oplus e_s J^{c_s+1}$ yields \vspace{-2.5pt}
$$\begin{array}{rcl}
\ell(\hspace{-1pt} (e_1A \hspace{-1pt} \oplus \hspace{-1pt}  \cdots \hspace{-1pt} \oplus \hspace{-1pt} e_sA)/(u_1, \ldots, u_s)J) \hspace{-7pt}
&\ge & \hspace{-7pt}
\ell(\hspace{-1pt} (e_1A\hspace{-1pt} \oplus \hspace{-1pt} \cdots \hspace{-1pt} \oplus \hspace{-1pt}  e_sA)/(e_1J^{c_1+1}\hspace{-1pt} \oplus \hspace{-1pt} \cdots \hspace{-1pt} \oplus \hspace{-1pt} e_s J^{c_s+1})) \vspace{2pt} \\
&=& \hspace{-7pt} \ell(e_1A/e_1J^{c_1+1})+\cdots + \ell(e_sA/e_sJ^{c_s+1})  \\
&\ge & \hspace{-7pt} (c_1+1)+\cdots +(c_s+1). \vspace{-4pt} 
\end{array}$$
Thus, \vspace{1pt} $\ell(D{\rm Tr}N)=\ell({\rm Tr}N)\ge \textstyle \sum_{i=1}^s \ell(e_iA/e_iJ^{c_i+1})-1\ge (s-1)+ \sum_{i=1}^s c_i\ge 2s-1.$ 

\vspace{1pt}

(2) Consider the indecomposable projective module $DI$ in $\mmod A^\circ.$
By Proposition \ref{duality_radi_soci}(2), \vspace{.5pt} $\rad(DI)\cong D(I/\soc I)\cong DN_1 \oplus \cdots \oplus DN_t,$ where $DN_i$ is of Loewy length $c_i$ such that $\soc(D N_i)\cong D(\ntop N_i)\cong e_iA/e_iJ,$ for $i=1, \ldots, t$. Since $D(\rad I)\cong DI/\soc(DI),$ as argued above, we obtain \vspace{-2pt}
$$\ell({\rm Tr}D(\rad I)) \hspace{-1pt} = \hspace{-1pt} \ell({\rm Tr}(DI/\soc\hp (DI))) \hspace{-1pt}\ge \hspace{-1pt} \textstyle\sum_{i=1}^t \hspace{-1pt} \ell(Ae_i/J^{c_i+1} e_i) \!-\! 1 \ge (t\!-\!1) \!+\!\sum_{i=1}^tc_i.\vspace{-2pt}$$
The proof of the proposition is completed.

\subsection{\sc Algebras of representation bound two} As shown below, the characterization of artin algebras of representation bound two is straightforward. 

\begin{Prop}\label{rb2}

Let $A$ be an artin algebra. Then $A$ is of representation bound two if and only if $A$ is a Nakayama algebra of Loewy length two. 

\end{Prop}

\noindent{\it Proof.} By definition, $A$ is of representation bound two if and only if every indecomposable module in $\mmod A$ is of length at most two and at least one of them is of length two, or equivalently, every indecomposable module in $\mmod A$ is uniserial of length at most two and at least one of them is of length two. The latter is clearly equivalent to $A$ being a Nakayama algebra of Loewy length two. The proof of the proposition is completed.

\subsection{\sc Algebras of representation bound three.} To characterize algebras of representation bound three, we reformulate the notion of wedged-string algebras as follows; compare \cite[(3.5)]{LiY}.

\begin{Defn}\label{3sa_def} 

An artin algebra $A$ is called {\it wedged-string} if every indecomposable projective module $P$ in $\mmod A$ or $\mmod A^\circ$ is uniserial or astride biserial with $\rad P=S_1\oplus S_2$, where $S_1$ and $S_2$ are simple such that $I_{S_1}$ and $I_{S_2}$ are uniserial.

\end{Defn}

\begin{Remark} \label{3sa_rem} (1)  In the finite-dimensional setting, a wedged-string algebra is a Tachikawa algebra, as named by Fuller; see \cite[(2.1)]{Ful}.

\noindent (2) In view of Proposition \ref{duality_radi_soci}(3), an artin algebra $A$ is wedged-string if and only if every indecomposable injective module in $\mmod A$ or $\mmod A^\circ$ has the property dual to the one stated in Definition \ref{3sa_def}.

\end{Remark} 

\smallskip

The following easy result says, in particular, that wedged-string algebras with radical cubed zero are of representation bound at most three.

\begin{Prop}\label{3sa_rb}

Let $A$ be a wedged-string algebra with radical cubed zero.

\begin{enumerate}[$(1)$]

\vspace{-1.5pt}

\item Every indecomposable module in $\mmod A$ is of length at most three.

\item Every almost split sequence in $\mmod A$ has at most two indecomposable direct summands in its middle term.
    
\end{enumerate}\end{Prop}

\vspace{-1.5pt}

\noindent{\it Proof.} Since $J^3=0$, the indecomposable projective or injective modules in $\mmod A$ are of length at most three. In particular, $A$ is a quadri-string algebra. By Theorem \ref{rb_4bsa}, every indecomposable module in $\mmod A$ is of length at most four. Suppose that there exists an indecomposable module $M$ of length four. Then, $M$ is neither projective nor injective. By Theorem \ref{l4-4bsa}, $M$ is N-shaped. Write $M=M_1+M_2,$ where $M_1$ is astride biserial of length three and $M_2$ is uniserial of length two. By Lemma \ref{N_mod_property}, $\rad M_1=\soc M_1=S \oplus \soc M_2,$ where $S$ is simple such that $M/S$ is co-astride biserial with $\soc(M_1/S)\cong \soc M_2.$ By Lemma \ref{4bs_loc_l3}(2), $I_{\soc M_2}$ is biserial. On the other hand, by Lemma \ref{maxi_local}(1), $M_1$ is projective. And by Definition \ref{3sa_def}, $I_{\soc M_2}$ is uniserial, absurd. This establishes Statement (1). In particular, $\mmod A$ has no lozenge module. Thus Statement (2) follows immediately from Theorem \ref{beta_A}. The proof of the proposition is completed.

\smallskip

We are ready to characterize artin algebras of representation bound three.

\begin{Theo}\label{rb3} 

Let $A$ be an artin algebra. Then $A$ is of representation bound three if and only if $A$ is a wedged-string algebra with radical cubed zero, but not a Nakayama algebra with radical squared zero. 

\end{Theo}

\vspace{-1.5pt}

\noindent{\it Proof.} Since algebras of representation bound one have a Loewy length of one, the sufficiency follows immediately from Propositions \ref{3sa_rb}(1) and \ref{rb2}. Conversely, assume that $A$ is of representation bound three. Let $P$ be an indecomposable projective module in $\mmod A.$ If $P$ is uniserial, then $\ell\ell(P)=\ell(P)\le 3.$ Suppose now that $P$ is not uniserial. Then, $\rad P$ is non-uniserial of length two. Thus, $\rad P=S_1\oplus S_2$, where $S_1, S_2$ are simple. So, $\ell\ell(P)=2$. In particular, $\ell\ell(A)\le 3.$

We claim that $I_{S_1}$ and $I_{S_2}$ are uniserial. Otherwise, we may suppose that $I_{S_1}$ is not uniserial. By a dual argument, we infer that $I_{S_1}/S_1\cong Ae_1/Je_1\oplus Ae_2/Je_2,$ where $e_1, e_2$ are primitive idempotents in $A$. Let $e$ be a primitive idempotent in $A$ such that $S_1\cong Ae/Je.$ \vp Since $eA\cong D(I_{S_1})$, in view of Proposition \ref{duality_radi_soci}(2), $eJ\cong D(I_{S_1}/S_1)\cong e_1A/e_1J\oplus e_2A/e_2J.$ \vp Thus, $eJ/eJ^2\cong e_1A/e_1J\oplus e_2A/e_2J.$ By Proposition \ref{S_ART_length}(2), $\ell({\rm Tr}DS_1)\ge 3,$ and by the assumption, $\ell({\rm Tr}DS_1)=3.$ Now, since the inclusion map $j_1: S_1\to P$ is irreducible, we have an almost split sequence \vspace{-12pt}
$$\xymatrix{0\ar[r]& S_1 \ar[r]^-{(j_1, f_1)^T} & P\oplus L_1 \ar[r]^-{(p_1, g_1)} & {\rm Tr}DS_1 \ar[r] & 0.}\vspace{-3pt}$$ 
Since $\ell(P)=3=\ell({\rm Tr}DS_1)$, we see that $p_1$ is an isomorphism, absurd. This establishes our claim. So $P$ has the properties stated in Definition \ref{3sa_def}. Since $A^\circ$ is also  of representation bound three, the indecomposable projective modules in $\mmod A^\circ$ also have these properties. Thus, $A$ is a wedged-string algebra with radical cubed zero, and  by Proposition \ref{rb2}, it is not a Nakayama algebra with radical squared zero. The proof of the theorem is completed.

\begin{Remark} 

In case $A$ is a finite-dimensional algebra over a field, by applying Tachikawa's results in \cite{Tac1}, one easily obtains Theorem \ref{rb3}.

\end{Remark}

\subsection{\sc Algebras of representation bound four} We conclude the paper with the following characterization of artin algebras of representation bound four. 

\begin{Theo}\label{main_rb4}

Let $A$ be an artin algebra. Then $A$ is of representation bound four if and only if $A$ is a quadri-biserial algebra, but not a wedged-string algebra with radical cubed zero.

\end{Theo}

\noindent{\it Proof.} By Proposition \ref{rb2}, algebras of representation bound at most two are 
Nakayama algebras of Loewy length at most two, which are wedged-string algebras with radical cubed zero. So, the sufficiency follows from Theorems \ref{rb_4bsa} and \ref{rb3}. Suppose conversely that $A$ is of representation bound four. Then so is $A^\circ$. In particular, the indecomposable projective modules in $\mmod A$ and $\mmod A^\circ$ have length at most four. %We shall split our proof into three sublemmas. 
Let $P$ be an indecomposable projective module $P$ in $\mmod A$ or $\mmod A^\circ.$ 

\vp

{\sc Sublemma 1.} {\it If $P$ is not uniserial, then it is biserial.}

Since $\ell(P)\le 4,$ in view of Proposition \ref{CL_3}, we may suppose that $\ell(P)=4$. Then $\ell(\rad P)=3.$ Assume first that $\rad P$ is indecomposable.  
Since the inclusion map $q: \rad P\to P$ is irreducible, we obtain an almost split sequence \vspace{-6pt}
$$\xymatrixcolsep{30pt}\xymatrix{0 \ar[r] & \rad P \ar[r]^-{(q, f)^T} & P \oplus N \ar[r]^-{(p, g)} & L \ar[r] & 0.}\vspace{-4.5pt}$$ 

Since $\ell(L)\le 4$, we see that $p: P\to L$ is an irreducible epimorphism, and so is $f\!: \! \rad P\to N$. Therefore, $\ell(L)\le 3$ and $\ell(N)\le 2$. Let $K$ be the kernel of $p: P\to L,$ which is indecomposable, see \cite[(V.5.6)]{ARS}, and
contained in $\rad P.$ Setting $T=\ntop K,$ we obtain a minimal projective presentation $\xymatrixcolsep{20pt}\xymatrix{\hspace{-2pt}P_{\hspace{.5pt}T} \ar[r] & P \ar[r] & L \ar[r] & 0;\hspace{-2pt}}\vspace{-3pt}$ and hence, a minimal injective co-presentation
$\xymatrixcolsep{20pt}\xymatrix{\hspace{-2pt}0 \ar[r] & D{\rm Tr}L \ar[r] & I_T \ar[r] & I_{\ntop p},}\hspace{-2pt}\vspace{-0.5pt}$ where $D{\rm Tr}L\cong \rad P.$ In particular, $\soc(\rad P) \cong T.$ If $\ell(L)=1$, then $N=0,$ and hence, $\rad P \cong K.$ This yields $\ntop(\rad P)\cong T\cong \soc(\rad P)$. Since $\ell(\rad P)=3,$ by Proposition \ref{CL_3}, $\rad P$ is uniserial, and so is $P$. If $2\le \ell(L)\le 3$, then $\ell(K)\le 2.$ Thus $K$ is uniserial; consequently $T$ is simple. Therefore, $\rad P$ is colocal, and by Lemma \ref{soc_rad}, so is $P.$ In view of Corollary \ref{rb4_lozenge}, $P$ is uniserial or lozenge. 

Next, assume that $\rad P$ is decomposable. Suppose that $\rad P=S_1 \oplus S_2 \oplus S_3,$ where the $S_i$ are simple. \vp Since the inclusion map $(q_1, q_2, q_3): S_1\oplus S_2 \oplus S_3 \to P$ is irreducible, \vspace{1pt} we have an irreducible map $(p_1, p_2, p_3)^T: P\to \tau^-S_1\oplus \tau^-S_2 \oplus \tau^-S_3.$ Since $\ell(\tau^-S_i)\ge \ell(P)-\ell(S_i)= 3,$ we have $\ell(\tau^-P) \ge \sum_{i=1}^3 \ell(\tau^-S_i)-\ell(P) \ge 5,$ absurd. Thus, $\rad P=S\oplus M,$ where $S$ is simple and $M$ is uniserial of length two. Thus, $P$ is biserial. This establishes Sublemma 1. So, $A$ is a biserial algebra.

\vp\vp

{\sc Sublemma 2.} {\it If $\ntop(\rad P)=S_1 \oplus S_2,$ then $I_{S_1}$ or $I_{S_2}$ is uniserial. }

We may assume that $P=Ae$ and $Je/J^2e \cong Ae_1/Je_1\oplus Ae_2/Je_2,$ where \vspace{.5pt} $e, e_1, e_2$ are primitive idempotents in $A$. Writing $S=Ae/Je,$ by Proposition \ref{S_ART_length}(1), we have $4\ge \ell(D{\rm Tr} S)\ge \ell(e_1J/e_1J^2)+ \ell(e_2J/e_2J^2)+1.$ Thus, $\ell(e_1J/e_1J^2)=1$ or $\ell(e_2J/e_2J^2)=1.$ By Proposition \ref{loc_cloc_bsm}(1), $e_1A$ or $e_2A$ is not biserial. By Sublemma 1, $e_1A$ or $e_2A$ is uniserial. Thus,
$I_{S_1}$ or $I_{S_2}$ is uniserial. This proves Sublemma 2.

\vp \vp

{\sc Sublemma 3.} {\it If $\rad P=S\oplus M$, where $S$ is simple and $M$ is uniserial of length two, then $I_S$ and $I_{\soc M}$ are uniserial.}

We may assume that $P=Ae$ and $Je=S\oplus M,$ with $S\cong Ae_1/Je_1$ and $M$ uniserial of length two such that $\soc M\cong Ae_2/Je_2$, where $e, e_1, e_2$ are primitive idempotents in $A$. Since $\soc M=\rad M=J^2e,$ there exists some $u\in e_2J^2e\backslash e_2J^3e$ such that $\soc M=Au.$ In particular, $e_2J^2/e_2J^3\ne 0.$
Set $N\!:\hp\hp =Ae/\soc(Ae).$ Since $S$ and $M$ are colocal of Loewy leng length one and two, respectively, we deduce from Lemma \ref{SF_ART_length}(1) that \vspace{-1pt}
$$
\begin{array}{rcl}
4 \hspace{1pt} \ge \hspace{1pt} \ell(D{\rm Tr} N) \hspace{-5pt} & \ge  & \hspace{-5pt}\ell(e_1A/e_1J^2)+\ell(e_2A/e_2J^3) -1 \\
%&=& 1+ \ell(e_3J/e_3J^2)+ {\sum}_{j=1}^{2} \ell(e_4J^j/e_4J^{j+1}) \\ 
&=& \hspace{-5pt} 1+ \ell(e_1J/e_1J^2) + \ell(e_2J/e_2J^2) + \ell(e_2J^2/e_2J^3) \\ 
&\ge & \hspace{-5pt} 2+\ell(e_1J/e_1J^2)+\ell(e_2J/e_2J^2).\vspace{-3pt}
\end{array}$$ 

Thus $\ell(e_1J/e_1J^2)=\ell(e_2J/e_2J^2)=1$. As seen above, $I_{S_1}$ and $I_{S_2}$ are uniserial. This establishes Sublemma 3. Therefore, $A$ is a quadri-biserial algebra; by Proposition \ref{3sa_rb}, $A$ is not a wedged-string algebra with radical cubed zero. The proof of the theorem is completed. 

\bigskip

\noindent {\sc Author contribution}: All the authors contributed equally in this work.

\smallskip

\noindent {\sc Data availability}: This manuscript does not report data generation or analysis.

\smallskip

\noindent {\sc Competing interests' statement:} The authors have no competing interests.

\end{document}